\documentclass{amsart}

\usepackage{graphicx}
\usepackage{amssymb}
\usepackage{dsfont}
\usepackage[all]{xy}
\usepackage{enumitem}
\usepackage{mathrsfs}
\usepackage{color}

 \usepackage{amsbsy}
\usepackage{amsmath}
\usepackage{amstext}
\usepackage{amsthm}
\usepackage{amscd}

\setlist[itemize]{leftmargin=2em}
\setlist[enumerate]{leftmargin=2em}

\usepackage
[pdfauthor={Francesc Castella and Carl Wang-Erickson},
pdftitle={Critical Lambda-adic modular forms and bi-ordinary complexes},
bookmarks=false]
{hyperref}
\hypersetup{
  colorlinks   = true, 
  urlcolor     = blue, 
  linkcolor    = blue, 
  citecolor   = blue 
}

\newtheorem{thm}{Theorem}[subsection]
\newtheorem{cor}[thm]{Corollary}
\newtheorem{lem}[thm]{Lemma}
\newtheorem{prop}[thm]{Proposition}
\newtheorem{ques}[thm]{Question}
\theoremstyle{definition}
\newtheorem{defn}[thm]{Definition}

\newtheorem{eg}[thm]{Example}
\theoremstyle{remark}
\newtheorem{rem}[thm]{Remark}

\numberwithin{equation}{section}

\newcommand{\ep}{\epsilon}

\DeclareMathOperator{\Hom}{\mathrm{Hom}}
\DeclareMathOperator{\End}{\mathrm{End}}

\DeclareMathOperator{\Gal}{{Gal}}

\DeclareMathOperator{\Ext}{Ext}
\DeclareMathOperator{\Tor}{Tor}

\newcommand{\ord}{\mathrm{ord}}

\newcommand{\spl}{\mathrm{spl}}

\DeclareMathOperator{\Ind}{Ind}

\DeclareMathOperator{\Frob}{Fr}

\DeclareMathOperator{\Sym}{Sym}

\DeclareMathOperator{\coker}{coker}

\newcommand{\GL}{\mathrm{GL}}

\DeclareMathOperator{\Spec}{Spec}

\DeclareMathOperator{\Spf}{Spf}

\DeclareMathOperator{\ad}{\mathrm{ad}}

\newcommand{\cA}{{\mathcal A}}

\newcommand{\cC}{{\mathcal C}}
\newcommand{\cD}{{\mathcal D}}
\newcommand{\cE}{{\mathcal E}}
\newcommand{\cF}{{\mathcal F}}
\newcommand{\cG}{{\mathcal G}}
\newcommand{\cH}{{\mathcal H}}
\newcommand{\cI}{{\mathcal I}}

\newcommand{\cO}{{\mathcal O}}

\newcommand{\cW}{{\mathcal W}}
\newcommand{\cX}{{\mathcal X}}
\newcommand{\cY}{{\mathcal Y}}

\newcommand{\frp}{{\mathfrak p}}

\newcommand{\F}{{\mathbb F}}

\newcommand{\bH}{{\mathbb H}}

\newcommand{\bT}{{\mathbb T}}

\newcommand{\frX}{{\mathfrak X}}

\newcommand{\Q}{{\mathbb Q}}
\newcommand{\Z}{{\mathbb Z}}
\newcommand{\C}{{\mathbb C}}

\newcommand{\ra}{\rightarrow}

\newcommand{\lra}{\longrightarrow}
\newcommand{\lrisom}{\buildrel\sim\over\lra}
\newcommand{\risom}{\buildrel\sim\over\ra}
\newcommand{\isoto}{\buildrel\sim\over\ra}
\newcommand{\rinj}{\hookrightarrow}

\newcommand{\rsurj}{\twoheadrightarrow}

\newcommand{\lsurj}{\twoheadleftarrow}
\newcommand{\sm}[4]{\ensuremath{\big(\begin{smallmatrix}#1 & #2 \\ #3 & #4\end{smallmatrix}\big)}}

\newcommand{\dR}{\mathrm{dR}}
\newcommand{\Ad}{\mathrm{Ad}}
\newcommand{\et}{\text{\'{e}t}}

\newcommand{\lb}{{[\![}}
\newcommand{\rb}{{]\!]}}

\newcommand{\CM}{\mathrm{CM}}

\newcommand{\dia}{{\langle-\rangle}}
\newcommand{\lr}[1]{{\langle{#1}\rangle}}

\newcommand{\m}{\mathfrak{m}}

\newcommand{\oQ}{\overline{\Q}}

\newcommand{\ttmat}[4]{\left( \begin{array}{cc}
#1 & #2 \\
#3 & #4
\end{array}
\right)}

\newcommand{\BO}{\mathrm{BO}}
\newcommand{\SBO}{\mathrm{SBO}}
\newcommand{\crit}{\mathrm{crit}}
\newcommand{\Fil}{\mathrm{Fil}}
\newcommand{\tord}{\mathrm{to}}
\newcommand{\aord}{\mathrm{ao}}
\newcommand{\ps}{\mathrm{ps}}
\newcommand{\KS}{\mathrm{KS}}
\newcommand{\cris}{\mathrm{cris}}

\newcommand{\Tr}{\mathrm{Tr}}

\newcommand{\Han}{{H^1_\dR(X_{\Q_p}^\mathrm{an},\cF_k)}}
\newcommand{\Hanpar}{{H^1_\mathrm{par}(X_{\Q_p}^\mathrm{an},\cF_k)}}

\makeatletter
\let\c@equation\c@thm
\makeatother
\numberwithin{equation}{subsection}

\title{Critical $\Lambda$-adic modular forms and bi-ordinary complexes}

\author{Francesc Castella}
\address{Department of Mathematics, University of California, Santa Barbara, CA 93106, USA}
\email{castella@ucsb.edu}

\author{Carl Wang-Erickson}
\address{Department of Mathematics, University of Pittsburgh, Pittsburgh, PA 15260, USA}
\email{carl.wang-erickson@pitt.edu}

\begin{document}


\dedicatory{In memory of Jo\"el Bella\"iche}%

\maketitle

\begin{abstract}

We produce a flat $\Lambda$-module of $\Lambda$-adic critical slope overconvergent modular forms, producing a Hida-type theory that interpolates such forms over $p$-adically varying integer weights. This provides a Hida-theoretic explanation for an observation of Coleman that the rank of such forms is locally constant in the weight. The key to the interpolation is to use Coleman's presentation of de Rham cohomology in terms of overconvergent forms to link critical slope overconvergent modular forms with the part of the first coherent cohomology of modular curves interpolated by Boxer--Pilloni's higher Hida theory. The novelty is that we interpolate a critical period in cohomology using modular forms, complementing the classical Hida-theoretic interpolation of an ordinary period. Using this interpolation, we also interpolate \emph{bi-ordinary complexes} in various weights into a perfect and self-dual complex of length 1 over $\Lambda$. By design, the cohomology of the bi-ordinary complex supports 2-dimensional $p$-adic representations of $\Gal(\oQ/\Q)$ that become reducible and decomposable upon restriction to a decomposition group at $p$. As applications and motivations for the above constructions, we prove ``$R=\bT$'' theorems for the critical and bi-ordinary Hecke algebras, produce a degree-shifting Hecke action on the cohomology of bi-ordinary complexes, and specialize this degree-shifting action to weight 1 to produce, under a supplemental assumption, an action of a Stark unit on the part of weight 1 coherent cohomology over $\Z_p$ that is isotypic for an ordinary eigenform with complex multiplication. 
\end{abstract}

\tableofcontents


\section{Introduction}

The first goal of this paper is to produce a $\Lambda$-adic theory of critical slope overconvergent modular forms, in analogy with Hida's theory of $\Lambda$-adic ordinary modular forms \cite{hida1986a}. The second goal, which was our original motivation coming from a question we asked in previous work \cite[\S1.6]{CWE1}, is to produce a length 1 perfect complex of $\Lambda$-adic forms whose cohomology supports Hecke eigensystems whose associated 2-dimensional Galois representations are reducible and decomposable upon restriction to a decomposition group at $p$. We call this a \emph{bi-ordinary complex}. 

\subsection{Main results about critical $\Lambda$-adic modular forms}

Given a prime number $p \geq 5$ and a tame level $N \in \Z_{\geq 5}$ such that $p \nmid N$, we produce a finite-rank $\Lambda$-module with a continuous Hecke action, admitting a control theorem to a $\Z_p$-lattice of critical slope forms $M_{k,\Z_p}^{\dagger,\crit}$ in the overconvergent modular forms $M_{k,\Q_p}^\dagger$ of weight $k \geq 2$ defined by Katz \cite{katz1973}.  We also provide another version interpolating cuspidal forms $S_{k,\Q_p}^{\dagger,\crit}$. Let $\Lambda \cong \Z_p\lb \Z_p^\times \rb$ denote the Iwasawa algebra that is the coordinate ring of weight space, with integer weights $k$ corresponding to $\phi_k : \Lambda \to \Z_p, \Z_p^\times \ni [z] \mapsto z^{k-1}$. 

\begin{thm}[{Theorem \ref{thm: main construction}}]
\label{thm: main intro}
There exist finitely generated flat $\Lambda$-modules $M_\Lambda^\crit, S_\Lambda^\crit$ along with control isomorphisms for $k \in \Z_{\geq 3}$, 
\[
M_\Lambda^\crit \otimes_{\Lambda, \phi_k} \Z_p \isoto M_{k,\Z_p}^{\dagger,\crit}, \qquad 
S_\Lambda^\crit \otimes_{\Lambda, \phi_k} \Z_p \isoto S_{k,\Z_p}^{\dagger,\crit}. 
\]
There are continuous actions of Hecke operators of level $N$ and a ``$p$-adic dual $U$-operator'' $U'$ for which these isomorphisms are equivariant. 
\end{thm}

A strong hint that there might be such an interpolation is Coleman's observation that the dimension of spaces of critical slope overconvergent forms of level $\Gamma_1(N)$ is locally constant with respect to the weight \cite[Cor.\ 7.2.3]{coleman1996}. 
\begin{rem}
    \label{rem: discontinuous}
    On the other hand, an apparent obstacle to this interpolation is that critical overconvergent forms cannot admit an interpolation in the sense of the Coleman--Mazur eigencurve \cite{CM1998}: critical $U$-eigenvalues have slope $k-1$ in weight $k$, making them impossible to $p$-adically interpolate. This is why we introduce a ``$p$-adic dual $U$-operator,'' denoted by $U'$, which amounts to a Frobenius operator and acts on $M_{k,\Q_p}^{\dagger,\crit}$ with slope $0$. 
\end{rem} 

Coleman's proof of local constancy is based on the short exact sequence of \cite[Cor.\ 7.2.2]{coleman1996} which we call ($\ast_k$) here. Since this short exact sequence is also our main point of departure for proving Theorem \ref{thm: main intro} and constructing the bi-ordinary complex, we now describe Coleman's argument. 

Letting $\theta = q\frac{d}{dq}$ be the Atkin--Serre differential operator on $p$-adic modular forms, the crucial short exact sequence for $k \in \Z_{\geq 3}$ is 
\[
\tag{$\ast_k$}
0 \to M_{2-k,\Q_p}^{\dagger, \ord}(k-1) \buildrel{\theta^{k-1}}\over\lra M_{k,\Q_p}^{\dagger, \crit} \lra \bH^1_\mathrm{par}(X_{\Q_p},\cF_k)^\crit \to 0,
\]
with notation as follows: $\bH^1_\mathrm{par}(X_{\Q_p}, \cF_k)$ is weight $k$ parabolic de Rham cohomology of the closed modular curve $X/\Z_p$ of level $\Gamma_1(N)$, generalizing standard de Rham cohomology which is the case $k=2$. The Tate twist by $(k-1)$ on the source makes $\theta^{k-1}$ Hecke-equivariant, since $n^i \theta^i \circ T_n = T_n \circ \theta^i$; in particular, $\theta^{k-1}$ on $M_{2-k,\Q_p}^\dagger$ preserves overconvergence and sends $U$-ordinary forms to $U$-critical forms. Coleman's key theorem \cite[Thm.\ 5.4]{coleman1996} is that weight $k$ de Rham cohomology $\bH^1_\mathrm{par}(X_{\C_p}, \cF_k)$ admits a presentation as a quotient as in ($\ast_k$), which endows it with a $p$-adic $U$-action. Coleman then sets up a perfect but non-canonical Hecke-compatible (after a twist) duality between $\bH^1_\mathrm{par}(X_{\C_p},\cF_k)^\crit$ and classical $U$-critical cusp forms of weight $k$. Now that both the sub and quotient of ($\ast_k$) have dimensions equal to the dimensions of a space of ordinary forms, Hida theory implies the claimed local constancy. 

Our method is to prove that ($\ast_k$) interpolates, as follows. A key input is Boxer--Pilloni's higher Hida theory for coherent cohomology of the modular curve \cite{BP2022}, which we use, by projection to the Hodge quotient of de Rham cohomology, to upgrade Coleman's description of the quotient of ($\ast_k$) to a \emph{canonical} interpolation. Let $M_\Lambda^\ord$ and $S_\Lambda^\ord$ denote Hida's modules of $\Lambda$-adic $U$-ordinary forms, and let $\cH_\Lambda^{1,\ord}$ denote Boxer--Pilloni's $F$-ordinary (where $F$ stands for Frobenius) $\Lambda$-adic coherent cohomology of degree $1$. Our $U'$ acts on it by $F$. We let $M^\aord_\Lambda$ denote \emph{anti-ordinary} $\Lambda$-adic forms, which are a twist of Hida's module $M_\Lambda^\ord$ interpolating the submodule of ($\ast_k$), $M^{\dagger,\ord}_{2-k, \Z_p}(k-1)$. Here $U'$ acts as $\lr{p}_N U^{-1}$.

\begin{thm}[{Theorem \ref{thm: main extension}}]
\label{thm: main extension intro}
Let $k \in \Z_{\geq 3}$. The short exact sequence $(\ast_k)$ admits a $\Z_p$-lattice that is the specialization along $\phi_k$ of a short exact sequence of flat and finitely generated $\Lambda$-modules 
\[
\tag{$\ast_\Lambda$} 
0 \to M_\Lambda^\aord \buildrel{\Theta_\Lambda}\over\lra M_\Lambda^\crit \buildrel{\pi_\Lambda}\over\lra \cH_\Lambda^{1,\ord} \to 0
\]
that is equivariant for Hecke operators of level $N$ and the $p$-adic operator $U'$. 
\end{thm}

The $\Lambda$-flatness in Theorem \ref{thm: main intro} is a consequence of Theorem \ref{thm: main extension intro}, since (higher) Hida theory provides for the $\Lambda$-flatness of the sub (quotient). This theorem implies that the ``critical eigencurve'' associated to $M_\Lambda^\crit$ has irreducible components that are either ordinary or anti-ordinary Hida families, where the latter refers to the twist by $(k-1)$ of weight $2-k$ interpolated in $M_\Lambda^\aord$.

We can deduce a \emph{formal} $\Lambda$-adic interpolation of the $T_p$-ordinary de Rham and crystalline cohomology of $X/\Z_p$. That is, we interpolate the weakly admissible filtered isocrystal associated to weight $k$ coefficient system $\cF_k$ on $X/\Z_p$ (compare work of Cais \cite{cais2018, cais2018CM}). Here ``formal'' means that we need a transcendental variable $p^{\kappa-1}$ that represents the scalar $p^{k-1}$ in weight $k$, which is $\Lambda$-adically discontinuous in exactly the same sense as in Remark \ref{rem: discontinuous}. This can be thought of as a $\Lambda$-adic interpolation of Coleman's presentation of de Rham cohomology \cite[Thm.\ 5.4]{coleman1996} by differentials of the second kind. 

\begin{thm}[{Theorem \ref{thm: main dR}}]
    \label{thm: main dR intro} 
    There is a finite flat $\Lambda$-module $\dR_\Lambda$ with a decomposition $\dR_\Lambda \cong M_\Lambda^\ord \oplus \cH^{1,\ord}_\Lambda$ that is equivariant for Hecke operators of level $N$ away from $p$. It is equipped with the additional formal data of
    \begin{itemize}[leftmargin=1em]
        \item a Frobenius endomorphism $\varphi_\Lambda$ preserving the decomposition
        \item a $\Lambda$-pure submodule $\Fil^{\kappa-1}$ isomorphic as a Hecke-module to $M_\Lambda^\ord$
        \item a $T_p$-action stabilizing $\Fil^{\kappa-1}$.
    \end{itemize} 
    The specialization of these data along $\phi_k$ for $k \in \Z_{\geq 3}$ realize weight $k$ de Rham cohomology $\bH^1_\dR(X_{\Q_p}, \cF_k)$, its crystalline Frobenius $\varphi_k$, the single non-trivial submodule of its Hodge filtration, and its Hecke action of level $N$. 
\end{thm}

\subsection{Bi-ordinary complexes and Galois representations}
\label{subsec: intro BO}

Let $k \in \Z_{\geq 2}$. Coleman pointed out that while classical critical cusp forms $S_{k,\C_p}^\crit \subset M_{k,\C_p}^{\dagger,\crit}$ are isomorphic as a Hecke module to the quotient module $\bH^1_\mathrm{par}(X_{\C_p},\cF_k)^\crit$ of ($\ast_k$), the composition of this inclusion and projection 
\[
S_{k,\C_p}^\crit \rinj M_{k,\C_p}^{\dagger,\crit} \rsurj \bH^1_\mathrm{par}(X_{\C_p},\cF_k)^\crit
\]
is not an isomorphism. Coleman asked whether all elements of the kernel \cite[\S7, Rem.\ 2]{coleman1996}, which are classical cusp forms in the image of $\theta^{k-1}$ of ($\ast_k$), have CM. 

Coleman's question can be interpreted in terms of Galois representations: Breuil--Emerton proved that for any classical cuspidal $U_p$-critical eigenform of level $\Gamma_1(Np^r)$, it lies in the image of $\theta^{k-1}$ (with the sequence ($\ast_k$) adapted to this level) if and only if its Galois representation $\rho_f : \Gal(\oQ/\Q) \to \GL_2(\oQ_p)$ is reducible and decomposable upon restriction $\rho_f\vert_{G_p}$ to a decomposition group at $p$ \cite[Thm.\ 1.1.3]{BE2010}. Given this equivalence, Coleman's question is independently attributable to Greenberg (see \cite[Ques.\ 1]{GV2004}), who expressed it in terms of $\rho_f\vert_{G_p}$.

One motivation to formulate a bi-ordinary complex in weight $k \in \Z_{\geq 3}$, which is perfect of length 1, is that it realizes the above kernel in its degree $0$ cohomology. Another motivation is that, as we will prove, it satisfies our request \cite[\S1.6]{CWE1} for a theory of modular forms supporting Galois representations that are reducible and decomposable on a decomposition group at $p$. For more discussion of why a length $1$ self-dual complex is a natural object to seek, see \S\ref{subsec: intro weight 1}. 

To define bi-ordinary complexes in weights $k$ and interpolate them over $\Lambda$, we first set up some simplifying notation. Let $M_k^\aord := M_{2-k,\Z_p}^{\dagger,\ord}(k-1)$, the submodule in ($\ast_k$), called the anti-ordinary forms of weight $k$. Let $M_k^\tord := M_{k,\Z_p}^\crit$, the \emph{twist-ordinary} forms of weight $k$, injecting under $\zeta_k : M_k^\tord \rinj M_k^\crit$. These latter forms are nothing other than the critical classical forms, but we use the term ``twist-ordinary'' for terminological reasons discussed in Remark \ref{rem: why tord}. The bi-ordinary complex in weight $k \in \Z_{\geq 3}$ is
\[
\tag{$\BO_k^\bullet$} \BO_k^0 := M_k^\aord \oplus M_k^\tord \mathrel{\mathop{\lra}^{\theta^{k-1} + \zeta_k}} M_k^{\dagger,\crit} =: \BO_k^1
\]
where the differential is nothing more than inclusion of submodules with addition. We also define a cuspidal variant $\SBO_k^\bullet$ where $S_k^\tord$ replaces $M_k^\tord$. 

Therefore we can interpret Coleman and Greenberg's question, with trivial level at $p$, in terms of the bi-ordinary complex. 
\begin{ques}
    \label{ques: CG}
    For $k \in \Z_{\geq 2}$, does $H^0(\SBO_k^\bullet)$ consist of CM forms? 
\end{ques}

We can interpolate $\BO_k^\bullet$, by (higher) Hida theory and Theorem \ref{thm: main extension intro}, into the $\Lambda$-adic bi-ordinary complex
\[
\tag{$\BO_\Lambda^\bullet$} \BO_\Lambda^0 := M_\Lambda^\aord \oplus M_\Lambda^\tord \mathrel{\mathop{\lra}^{\theta_\Lambda + \zeta_\Lambda}} M_\Lambda^{\dagger,\crit} =: \BO_\Lambda^1
\]
similarly to the cuspidal variant $\SBO_\Lambda^\bullet$. Crucially, $\SBO_\Lambda^\bullet$ is quasi-isomorphic under the quotient map of ($\ast_\Lambda$) to 
\[
[S_\Lambda^\tord \mathrel{\mathop{\lra}^{\pi_\Lambda \circ \zeta_\Lambda}} \cH_\Lambda^{1,\ord}], 
\]
whose terms are perfectly Hecke-equivariantly dual under Boxer--Pilloni's Serre duality pairing \cite{BP2022} once some twistings are accounted for (see \S\ref{subsec: tord}). Consequently, $\SBO_\Lambda^\bullet$ is perfectly Serre self-dual. This differential $\pi_\Lambda \circ \zeta_\Lambda$ is a Hida-theoretic interpolation (of a projection to the Hodge quotient) of the critical part of the map from cusp forms to de Rham cohomology that Coleman highlighted as the map denoted ``$\iota$'' on \cite[p.\ 232]{coleman1996}. 

Next let us set up some Hecke algebras, which exhibit $\Lambda$-linear dualities with the modules they act on. Let $\bT_\Lambda^\crit, \bT_\Lambda^\aord, \bT_\Lambda^\tord$ denote the Hecke algebras generated by Hecke actions on the analogously decorated flat $\Lambda$-modules $M_\Lambda^?$. In fact, because it arises from the addition of the inclusions of two submodules, the differential of $\BO_\Lambda^\bullet$ is perfectly $\Lambda$-linearly dual to the natural homomorphism of finite flat $\Lambda$-adic Hecke algebras 
\[
\psi_\Lambda : \bT_\Lambda^\crit \to \bT_\Lambda^\aord \times \bT_\Lambda^\tord.
\]
This $\psi_\Lambda$ is the map on coordinate rings corresponding to the inclusions of the ordinary and anti-ordinary eigencurves into the ``critical eigencurve'' (mentioned after Theorem \ref{thm: main extension intro}). 

The results in \S\ref{sec: bi-ordinary} are homological-algebraic consequences of the two dualities: the self-duality of $\SBO_\Lambda^\bullet$ and the dualities between forms and Hecke algebras. To summarize these results, we need some more notation and terminology. 

Let $\bT_\Lambda^\BO$ denote the Hecke algebra of $H^*(\BO_\Lambda^\bullet)$; define $\bT_\Lambda^\SBO$ similarly. When $V$ is a $\Lambda$-module, let $V^\vee$ denote $\Hom_\Lambda(V,\Lambda)$ and let $T_1(V)$ denote its maximal $\Lambda$-torsion submodule. Let $\phi : \Lambda \to \cO$ denote the normalization of the quotient by a height 1 prime of characteristic 0; we call these $p$-adic weights. When $\bT_\Lambda^\star$ is a $\Lambda$-adic Hecke algebra, let $\phi_f : \bT_\Lambda^\star \to \cO$ lie over a $p$-adic weight (for $\star = \crit,\tord, \aord, \BO, \SBO$); we call these \emph{$\star$ Hecke eigensystems}. A \emph{$\star$ Hida family} is a prime of $\bT_\Lambda^\star$ lying over a minimal prime of $\Lambda$. Write $\bT_\Lambda^{\star,\circ}$ for cuspidal variants. 

\begin{thm}[{\S\ref{sec: bi-ordinary}}]
    \label{thm: main BO intro}
    All of the following modules are finitely generated over $\Lambda$, and all of the following maps are Hecke-equivariant:
    \begin{enumerate}
        \item $\bT_\Lambda^\BO \cong \bT_\Lambda^\aord \otimes_{\bT_\Lambda^\crit} \bT_\Lambda^\tord$ and $\bT_\Lambda^\SBO \cong \bT_\Lambda^\aord \otimes_{\bT_\Lambda^\crit} \bT_\Lambda^{\tord,\circ}$.
        \item $H^0(\BO_\Lambda^\bullet)$ and $H^0(\SBO_\Lambda^\bullet)$ are $\Lambda$-flat.
        \item $H^1(\BO_\Lambda^\bullet) \cong H^1(\SBO_\Lambda^\bullet)$, and there is a short exact ``Serre duality sequence'' 
        \[
        0 \to T_1(H^1(\BO_\Lambda^\bullet)) \to H^1(\BO_\Lambda^\bullet) \to H^0(\BO_\Lambda^\bullet)^\vee \to 0.
        \]
        \item There is a perfect $\Lambda$-bilinear self-duality pairing
        \[
        T_1(H^1(\BO_\Lambda^\bullet)) \times T_1(H^1(\BO_\Lambda^\bullet)) \to Q(\Lambda)/\Lambda.
        \]
        \item $\bT_\Lambda^\SBO$ is isomorphic to the Hecke algebra of $H^1(\SBO_\Lambda^\bullet)$, and this Hecke algebra is $\Lambda$-torsion if and only if $H^0(\SBO_\Lambda^\bullet) = 0$. 
        \item The Hecke eigensystems factoring through the Hecke algebra of $H^0(\BO_\Lambda^\bullet)$ are exactly those Hecke eigensystems of $\bT_\Lambda^\crit$ that lie in a bi-ordinary Hida family.
        \item The Hecke eigensystems supporting $T_1(H^1(\BO_\Lambda^\bullet))$ are exactly the Hecke eigensystems that are bi-ordinary but do not lie in a bi-ordinary Hida family. 
    \end{enumerate}
\end{thm}

\begin{proof}
    (1): Corollary \ref{cor: BO is to and ao}. (2): 
    Lemma \ref{lem: PD of H}. (3): Theorem \ref{thm: SD for BO}. (4): Corollary \ref{cor: SD on HSBO}. (5-7): Proposition \ref{prop: BO T to forms duality}. 
\end{proof}

The weight $k$ instance of Theorem \ref{thm: main dR intro} is already sufficient to provide a ``Coleman's presentation variant'' of the proof of the level $\Gamma_1(N)$ case \cite[Thm.\ 4.3.3]{BE2010} of Breuil--Emerton's theorem.  We can interpret the theorem in terms of the bi-ordinary complex and also generalize it to Hida families. 

\begin{thm}[{Theorems \ref{thm: BE alternate}, \ref{thm: BO equals split in Hida families}}]
    \label{thm: BO=split intro}
    Let $f$ denote a classical $U_p$-critical $p$-adic eigenform of level $\Gamma_1(N) \cap \Gamma_0(p)$ and weight $k \in \Z_{\geq 3}$ (resp.\ $\cF$ a twist-ordinary Hida family). The following conditions are equivalent: 
    \begin{enumerate}
        \item $f$ is in the image of the map $\theta^{k-1}$ of $(\ast_k)$. (resp.\ $\cF \in \mathrm{image}(\Theta_\Lambda)$ in $(\ast_\Lambda)$). 
        \item $\rho_f\vert_{G_p}$ (resp. $\rho_\cF\vert_{G_p}$) is reducible and decomposable.
        \item The $f$-isotypic part of $H^*(\BO_k^\bullet)$ (resp.\ the $\cF$-isotypic part of $H^*(\BO_\Lambda^\bullet)$) is non-zero.
    \end{enumerate}
\end{thm}

\subsection{Deformation theory of bi-ordinary and critical Galois representations}

Next we specify universal Galois deformation problems that we expect to characterize the Galois representations arising from $\bT_\Lambda^\BO$ and $\bT_\Lambda^\crit$. This occurs in the classical and broadly known case of ordinary deformation rings. For simplicity, we work in the case where the residual representation $\rho : G_\Q \to \GL_2(\F)$ satisfies conditions under which its universal ordinary deformation ring $R^\ord$ is known to satisfy $R^\ord \cong (\bT_\Lambda^\ord)_\rho$, where $(-)_\rho$ designates localization at the residual Hecke eigensystem corresponding to $\rho$: Taylor--Wiles conditions and $p$-distinguishedness. The twist- and anti-ordinary conditions are straightforward twists of the ordinary condition, formulated by Mazur \cite[\S2.5]{mazur1989}, which insists that a deformation $\tilde \rho$ of $\rho$ becomes reducible upon its restriction $\tilde \rho\vert_{G_p}$ to a decomposition group $G_p \subset G_\Q$ at $p$, with an unramified quotient. 

The bi-ordinary condition, which adds decomposability to the reducibility condition, is also standard. It was first discussed directly by Ghate--Vatsal \cite{GV2011}, who also proved the following characterization of Hida families supporting bi-ordinary Galois representations \cite{GV2004}. We express their result in the language of the bi-ordinary complex, using ``CM'' to stand for ``complex multiplication,'' and deduce implications. 

\begin{thm}[{\cite[Thm.\ 3]{GV2004}, Corollary \ref{cor: SBO is CM}}]
    \label{thm: intro GV}
    Under the Taylor--Wiles and $p$-distinguished hypotheses on $\rho$, $H^0(\SBO_\Lambda^\bullet)_\rho$ is naturally isomorphic to the CM twist-ordinary forms $(S_\Lambda^{\tord,\CM})_\rho \subset (S_\Lambda^\tord)_\rho$. In addition, $T_1(H^1(\SBO_\Lambda^\bullet))_\rho$, which is isomorphic to $\ker(\bT_{\Lambda,\rho}^\SBO \rsurj \bT_{\Lambda, \rho}^{\tord,\CM})$, is supported exactly at the non-CM Hecke eigensystems with residual eigensystem $\rho$ supporting a bi-ordinary Galois representation. 
\end{thm}

Using these results, we can use Theorem \ref{thm: main BO intro}(7) to interpret Question \ref{ques: CG} in terms of bi-ordinary cohomology. 

\begin{cor}[{Corollary \ref{cor: T1H1 is exceptional}}]
    Assume that all of the bi-ordinary Hida families (that is, in  $H^0(\SBO_\Lambda^\bullet)$) are either Eisenstein or CM, as we know is true on the local components with residual Hecke eigensystems $\rho$ as in Theorem \ref{thm: intro GV}. Then the answer to Question \ref{ques: CG}  is ``yes'' if and only if $T_1(H^1(\BO_\Lambda^\bullet))$ has support away from weights $k \in \Z \smallsetminus \{1\}$. 
\end{cor}

In contrast, our formulation of a critical deformation problem is novel. The Galois representations of critical overconvergent \emph{eigenforms} $f$ has been understood since work of Kisin \cite{kisin2003}: even though $\rho_f\vert_{G_p}$ may not be crystalline, it admits a crystalline $G_p$-quotient with a crystalline eigenvalue matching its $U$-eigenvalue. But in order to do deformation theory, we must understand the Galois-theoretic interpretation of critical overconvergent \emph{generalized} eigenforms and the nilpotent elements they produce in $\bT_\Lambda^\crit$. This kernel of interest is exactly the kernel of $\psi_\Lambda$, which controls the extension class of ($\ast_\Lambda$). 

The explicit description of the Galois representation associated to generalized eigenforms with  CM eigensystems by C.-Y.\ Hsu \cite{hsu2020}, which we recapitulate in Example \ref{eg: hsu}, provides the template for our deformation condition. As we discuss in \S\ref{subsec: Gal rep discussion}, the idea we extract from Example \ref{eg: hsu} is that one should abandon the \emph{reducible} condition on $\tilde \rho\vert_{G_p}$ and instead define the critical deformation problem by insisting that the \emph{pseudorepresentation} of $\tilde \rho\vert_{G_p}$ is reducible. This implies reducibility after taking the quotient by a square-nilpotent ideal. For the Galois representation supported by $\bT_\Lambda^\crit$, we expect that this ideal is $\ker \psi_\Lambda$. 

\begin{thm}[{Theorems \ref{thm: R=T BO}, \ref{thm: R=T crit}}]
    \label{thm: R=T intro}
    If $R^\ord \cong (\bT_\Lambda^\ord)_\rho$, then $R^\BO \cong (\bT_\Lambda^\BO)_\rho$ and $R^\crit \cong (\bT_\Lambda^\crit)_\rho$. 
\end{thm}

This theorem and Theorem \ref{thm: BO=split intro} show that our original motivation for bi-ordinary complexes \cite[\S1.6]{CWE1} is satisfied. 

\subsection{Degree-shifting maps and weight 1 coherent cohomology}
\label{subsec: intro weight 1}

We were motivated to search for a length 1 $\Lambda$-perfect complex of Hecke modules with the properties of $\BO_\Lambda^\bullet$ to answer our question \cite[\S1.6]{CWE1} by the observation that the bi-ordinary (also known as $p$-split) deformation problem represented by $R^\BO$ has $\ell_0$-invariant $1$ in the sense recorded in \cite[\S2.8]{calegari2020}, i.e.\ expected $\Lambda$-codimension $\ell_0 = 1$. The philosophy of derived enrichments to the global Langlands correspondence, proposed by Venkatesh with collaborators \cite{PV2021, GV2018, venkatesh2019, HV2019}, suggests this. In contrast, the $\ell_0$-invariant of $R^\ord$ is $0$. 

In our initial attempts we let $\BO_\Lambda^1$ be the $\Lambda$-saturation of the sum $M_\Lambda^\tord + M_\Lambda^\aord$ within $\Lambda$-adic $q$-series, making its $H^1$ be $\Lambda$-torsion. But this has shortcomings that are addressed by letting $\BO_\Lambda^1 := M_\Lambda^\crit$. We were guided toward the formulation of $\Lambda$-adic critical forms for their purpose in $\BO_\Lambda^\bullet$ by the emerging proposals about the degree-shifting motivic ``hidden action'' of a Stark unit group on weight 1 coherent cohomology of Harris--Venkatesh \cite{HV2019}, further established by Darmon--Harris--Rotger--Venkatesh \cite{DHRV2022}  and R.\,Zhang \cite{robin-zhang-I} in the modulo $p^n$ Taylor--Wiles setting, Horawa \cite{horawa2023} in the complex Hodge setting, and others. In particular, ordinary weight 1 coherent cohomology and the bi-ordinary complex have in common the key $p$-local decomposability property of the Galois representation, which is what accounts for $\ell_0 = 1$. This analogy also showed us that CM forms ought to show up in both cohomological degrees, which led us to focus on Coleman's work on the critical case \cite[\S7]{coleman1996}. 

While the predicted $\Lambda$-codimension $\ell_0 = 1$ of the bi-ordinary cohomology might appear to be in contradiction with the existence of CM Hida families -- which are $\Lambda$-flat and thus of codimension $0$ -- this is accounted for by the existence of a $\Lambda$-flat degree-shifting action of $\ker \psi_\Lambda$. 

\begin{thm}[{Theorem \ref{thm: flat derived action}}]
    \label{thm: flat DS action intro}
        There is a degree-shifting action of $\ker \psi_\Lambda$ on $H^*(\SBO_\Lambda^\bullet)$ realizable as a $\bT_\Lambda^\SBO$-linear isomorphism 
        \[
        \ker \psi_\Lambda \isoto \Hom_{\bT_\Lambda^\SBO}(H^1(\SBO_\Lambda^\bullet), H^0(\SBO_\Lambda^\bullet)).
        \]
\end{thm}

Naturally, we also wanted to know to what extent weight 1 coherent cohomology is a $p$-adic degeneration of the bi-ordinary complex to this uniquely singular weight, and to look for the action of the Stark unit group. We can use $R^\crit \cong (\bT_\Lambda^\crit)_\rho$ from Theorem \ref{thm: R=T intro} to relate $\ker \psi_\Lambda$ to arithmetic. 

Due to Theorem \ref{thm: intro GV}, under the Taylor--Wiles and $p$-distinguished hypotheses on the residual eigensystem $\rho$, the target of the action of Theorem \ref{thm: flat DS action intro} vanishes unless $\rho$ has CM, i.e.\ is induced from a character of an imaginary quadratic field $K/\Q$. So we put ourselves in the CM case, letting $\rho \simeq \Ind_K^\Q \eta$. Let $\eta^-$ denote $\eta (\eta^c)^{-1}$, the ratio of $\eta$ with its composition with the action of complex conjugation. Among the objects of the anti-cyclotomic of Iwasawa theory over $K$ is the $\eta^-$-branch of the global unit group $\cE_\infty^-(\eta^-)$ and the (unramified) class group $\cX_\infty^-(\eta^-)$. When $f$ is as in the statement below, let $U_f$ denote a Stark unit group associated to its trace-zero adjoint Galois representation $\Ad^0 \rho_f$. 

\begin{thm}[{Theorems \ref{thm: elliptic action} and \ref{thm: classical Stark unit action}}]
Under the Taylor--Wiles and $p$-distinguished hypotheses on $\rho \simeq \Ind_K^\Q \eta$ and a choice of a $T_p$-ordinary cuspidal weight 1 CM eigenform $f$ of level $\Gamma_1(N)$ defined over an integer ring $\cO_E/\Z_p$ with $U_p$-stabilization in the congruence class of $\rho$, the action of Theorem \ref{thm: flat DS action intro} gives rise to two actions
\begin{gather*}
    \cE_\infty^-(\eta^-) \lrisom \Hom_{\bT[U']_\Lambda}(H^1(\SBO_\Lambda^\bullet)_\rho, H^0(\SBO_\Lambda^\bullet)_\rho) \cong \Hom_{\bT[U']_\Lambda}(\cH^{1,\ord}_{\Lambda,\CM,\rho}, S_{\Lambda,\rho}^{\tord,\CM}) \\ 
    U_f \rinj \Hom_{\cO_E}(e(T_p)H^1(X_{\cO_E},\omega)_f, e(T_p)H^0(X_{\cO_E},\omega(-C))_f),
\end{gather*}
where the construction of the latter action (of $U_f$) is conditioned upon the vanishing of $\cX_\infty^-(\eta^-)$. 
\end{thm}

\begin{rem}
    It is possible that there are congruences between (a $p$-stabilization of) $f$ and \emph{both} twist-ordinary and anti-ordinary non-CM forms, but still $\cX_\infty^-(\eta^-)$ vanishes. 
\end{rem}

\subsection{Outline of the paper}

In section~\ref{sec:modular curves}, we express our conventions for modular curves, recall cohomology theories for modular curves and their relations, and recall results from the higher Hida theory of Boxer--Pilloni \cite{BP2022}. In section~\ref{sec: p-integral cric oc}, we establish the good behavior of the exact sequence ($\ast_k$), which involves comparing $\Z_p$-lattices between coherent, de Rham, and crystalline cohomology of $X$ with weight $k$ coefficients. In section~\ref{sec: construction}, we formulate critical $\Lambda$-adic modular forms, prove its control theorem, and find it in two short exact sequences, one of which is $(\ast_\Lambda)$. We also discuss $\Lambda$-adic interpolation of de Rham cohomology. In section~\ref{sec: bi-ordinary}, we formulate bi-ordinary complexes and study their Serre self-duality and Hecke actions. In section~\ref{sec: galois}, we make initial observations about Galois representations attached to bi-ordinary cohomology, provide a proof of Breuil--Emerton's result \cite[Thm.\ 4.3.3]{BE2010}, and characterize bi-ordinary Hida families in terms of their Galois representations. In section~\ref{sec: deformations}, we specify all of the relevant Galois deformation problems and establish bi-ordinary and critical $R = \bT$ theorems. In section~\ref{sec:deg-shifting}, we discuss Ghate--Vatsal's result \cite{GV2004} that shows that all bi-ordinary Hida families have CM and construct the Stark unit group action on weight 1 $T_p$-ordinary coherent cohomology associated to a weight 1 eigenform. 

Section 9 is a correction to our previous article \cite{CWE1}, specifically, adding an additional assumption to a claim in commutative algebra that appeared as \cite[Prop.\ 6.1.2]{CWE1}. We remark that this correction does not affect those results of \cite{CWE1} which are used in this article. 

\subsection{Acknowledgements} 
It is an honor to dedicate this paper to the memory of Jo\"el Bella\"iche, with gratitude for his personal and mathematical influence. In retrospect, some email exchanges with Jo\"el in 2011 on his landmark work \cite{bellaiche2012L} (constructing $p$-adic $L$-functions for critical forms) played a significant role in shaping FC's later research. CWE is thankful for Jo\"el's mentorship while a postdoc at Brandeis University, and also for introducing Preston Wake to him at the Glenn Stevens birthday conference in 2014. This introduction lead to a number of joint works, with \cite{WWE1} in particular influencing CWE's perspective on the present work. Both authors would like to thank George Boxer, Henri Darmon, Aleksander Horawa, and Preston Wake for helpful conversations. The authors thank Shaunak Deo and Bharathwaj Palvannan for pointing out the mistake in \cite[\S6]{CWE1} that is addressed in \S\ref{sec: correction}. They also thank the anonymous referee for useful comments. During the work on this paper, FC was partially supported by the National Science Foundation through grant DMS-2101458; CWE was partially supported by the Simons Foundation through grant 846912 and the National Science Foundation through grant DMS-2401384.

\section{Modular curves and modular forms}\label{sec:modular curves}

In this section, we recall fundamental notions about modular curves, their cohomology, and modular forms. Foremost are the Hida-theoretic interpolation of ordinary modular forms \cite{hida1986, wiles1988}, higher Hida theory of the modular curve due to Boxer--Pilloni \cite{BP2022}, and the overconvergent modular forms-based presentation ``of the second kind'' of analytic de Rham cohomology due to Coleman \cite{coleman1996}. 

\subsection{Algebraic number theoretic context}
\label{subsec: ANT context}

For use throughout the paper, let $\oQ/\Q$ denote a fixed algebraic closure, and likewise $\oQ_\ell/\Q_\ell$ for each prime number $\ell$. We let these be equipped with embeddings $\oQ \rinj \oQ_\ell$, and the resulting decomposition subgroups of $G_\Q := \Gal(\oQ/\Q)$ be denoted $G_\ell := \Gal(\oQ_\ell/\Q_\ell)$. Accordingly, we fix arithmetic Frobenius elements $\Frob_\ell \in G_\ell \subset G_\Q$. 

We select a distinguished prime $p \geq 5$. Mostly we consider modular curves and modular forms over $\Z_p$, but when we work over number fields, we use the embedding $\oQ \subset \oQ_p$ to pass to the $p$-adic setting. For example, this determines the notion of an ordinary modular form over $\oQ$. 

\subsection{Conventions and notation for modular forms and Hecke algebras}
\label{subsec: conventions}

Let $N \geq 5$ be a positive integer relatively prime to $p$, so that $X = X_1(N)/\Z_p$ is a fine moduli space for generalized elliptic curves. We use the multiplicative (Deligne--Rapoport) model, parameterizing elliptic curves $E$ with injections $\mu_N \rinj E$ in the open locus $Y_1(N)$. We will also use the modular curve $X_0(p) := X(\Gamma_1(N) \cap \Gamma_0(p))/\Z_p$, with its two degeneration maps $\pi_1, \pi_2 : X_0(p) \rsurj X$. Occasionally we will use the additive models $X', X'_0(p)/\Z_p$ ($X' = X'_1(N)$), parameterizing $\Z/N\Z \rinj E$ in the open locus $Y'_1(N)$. We refer to \cite[\S1.4]{FK2012} for details and comparisons of the models. We have the usual modular line bundle $\omega$ over each of these modular curves and the divisor $C$ of the cusps (see e.g.\ \cite[\S2, p.\ 217]{coleman1996}).

We have the following standard Hecke correspondences operating on cohomology theories over $X$ and $X'$: 
\begin{itemize}
\item $T_n$ for $(n,N) = 1$ over $X$, and for $(n,Np) = 1$ over $X_0(p)$
\item $U_\ell$ for primes $\ell \mid N$ over $X$, and for $\ell \mid Np$ over $X_0(p)$
\item We also sometimes write $T_n$ for any $n \in \Z_{\geq 1}$ for the usual (but context-dependent) $n$th Hecke operator, as we explain more in Definition \ref{defn: Tn for general n} below.
\item $\langle d \rangle_N$ for $d \in (\Z/N\Z)^\times$ over both $X$ and $X_0(p)$
\item Dual operators: $T_n^* = T_n \langle n\rangle_N^{-1}$, $U_\ell^*$, $\lr{d}_N^* = \lr{d}_N^{-1}$. 
\end{itemize}

We write $\bT[T_p]$ for the \emph{abstract Hecke algebra of level $\Gamma_1(N)$ over $\Z_p$}, which is the (commutative) polynomial ring generated over $\Z_p$ by the symbols $T_n$, $U_\ell$, and $\langle d\rangle$ above. 

Our notation emphasizes the operator $T_p$ because we will often replace it with operators of $p$-adic origin $U,U'$ by writing $\bT[U]$ or $\bT[U']$. Correspondingly, we write $\bT[\,]$ for the Hecke algebra $\bT$ deprived of $T_p$ and $\bT[T_p,U]$ for the Hecke algebra with both $T_p$ and $U$. And $\bT[-]$ denotes a generic instance of these. Likewise, write $\bT^*[-^*]$ for the analogous dual versions. 

On the other hand, we write $\bT_{\Gamma_0(p)}$ for the abstract Hecke algebra of level $\Gamma_1(N) \cap \Gamma_0(p)$. This has the operator $U_p$ at $p$. We often consider $\bT_{\Gamma_0(p)}$-modules to be $\bT[U]$ modules via the homomorphism $\bT[U]\to \bT_{\Gamma_0(p)}, U \mapsto U_p$, the rest of the symbols $T_n, U_\ell, \lr{d}_N$ matching as usual. Later we specify a different compatibility for $U'$. 

When $M$ is a $\bT[-]$-module, we write $\bT[-](M)$ for the image of $\bT[-]$ in $\End_{\Z_p}(M)$. When $R$ is a $\Z_p$-algebra, $\bT[-]_R$ denotes $\bT[-] \otimes_{\Z_p} R$. 

\begin{defn}
\label{defn: Tn for general n} 
For this definition, let $\bT'$ stand in for the abstract Hecke algebras that have a ``$U$-type'' operator at $p$: $\bT_{\Gamma_0(p)}$, $\bT[U]$, or $\bT[U']$. For notational simplicity, usually related to $q$-series, we sometimes write $T_n \in \bT'$ for \emph{any} $n \in \Z_{\geq 1}$ to denote the usual multiplicative combination of the $T$- and $U$-type operators listed above. Let $U \in \bT'$ be the operator indexed by $p$ for the moment. That is, factoring $n$ as $n = n' \cdot p^i \cdot \prod_{\ell \mid N} \ell^{i_\ell}$ (over prime divisors $\ell \mid N$) such that $(pN, n') = 1$ , we write 
\[
T_n = T_{n'} \cdot U^i \cdot \prod_{\ell \mid N} U_\ell^{i_\ell}. 
\]
\end{defn}

We use the usual $q$-series expansion of modular forms (see e.g.\ \cite[p.\ 451]{gross1990}). We call modular forms (resp.\ cusp forms) of weight $k$ with \emph{coefficients} in a $\Z_p$-algebra $R$ as sections $f \in H^0(X,\omega^k \otimes_{\Z_p} R)$ (resp.\ $H^0(X,\omega^k(-C) \otimes_{\Z_p} R)$). Indeed, thanks to the use of the multiplicative model, there is compatibility between $q$-series and coherent cohomology, for $\Z_p$-algebras $R' \subset R$, 
\begin{equation}
    \label{eq: q vs coh}
    f = \sum_{n \geq 0} a_n(f)q^n \text{ with } a_n \in R' \iff f \in H^0(X,\omega^k \otimes_{\Z_p} R');
\end{equation}
see \cite[Prop.\ 2.7]{gross1990}, \cite[\S1.5.10]{FK2012}. There is an action of $\bT[T_p]$ on these cohomology modules, and we have the ``$a_1$-pairing''
\[
\lr{} : \bT[T_p] \times H^0(X,\omega^k \otimes_{\Z_p} R) \to R, \quad (T,f) \mapsto a_1(T \cdot f)
\]
and similarly with $\omega^k(-C)$ in place of $\omega^k$.

Let $Q(R)$ denote the total ring of fractions of a ring $R$. Let $M$ be a finitely generated $R$-submodule of $Q(R)\lb q\rb$ -- the $M$ we will deal with arise as the $q$-series realization of some module of modular forms. Let $M$ have a $\bT[-]$-action where $\bT[-]$ stands for one of the Hecke algebras above. We will frequently work in cases in which the pairing $\lr{}$ induces perfect pairings of finitely generated $R$-modules
\begin{equation}
    \label{eq: a1 duality}
    \lr{} : \bT[-](M)_R \times M \to R. 
\end{equation}
Assuming that this pairing is $R$-valued and perfect, we are able to define $a_0 \in \bT[-](M) \otimes_{R} Q(R)$, the \emph{constant term operator}, to be the element given by $M \ni f \mapsto a_0(f)$. 

We have set up our operators $T_n$ for general $n \in \Z_{\geq 1}$ in Definition \ref{defn: Tn for general n} so that, when the Hecke algebra is $\bT[U]$ or $\bT_{\Gamma_0(p)}$, $a_n(f) = a_1(T_n \cdot f) = \lr{T_n,f}$. However, when we use the alternate operator $U'$ at $p$ that is not compatible with $U_p$, we will apply the following notion of alternate ($U'$-based) $q$-series, following the pattern of Definition \ref{defn: Tn for general n}.  
\begin{defn}
    \label{defn: alternate q-series}
    Let $R$ be a $\Z_p$-algebra. Let $M$ be a $\bT[U']_R$-module. Assume that there is a perfect $R$-valued duality pairing as in \eqref{eq: a1 duality}. For $f \in M$, its \emph{alternate ($\bT[U']$-based) $q$-series} is 
    \[
    f(q) = a_0(f) + \sum_{n \geq 1} \lr{T_n,f} q^n
    \]
    with $T_n \in \bT[U']$ as in Definition \ref{defn: Tn for general n}.
    \end{defn}

\begin{rem} 
The rationale for Definition \ref{defn: alternate q-series} is that we will be able to $p$-adically interpolate some alternate $q$-series when the usual $q$-series are impossible to interpolate. See Definition \ref{defn: Zp crit oc forms}. 
\end{rem}

We will use a slightly different notion of modular forms \emph{over} a $\Z_p$-algebra $R$ that amounts to asking that its $q$-series coefficients other than the constant term are valued in $R$. This difference is trivial for cusp forms. 
\begin{defn}
\label{defn: classical forms}
Let $R$ be a $\Z_p$-algebra and let $k \in \Z$. 
\begin{itemize}
\item The $\bT[T_p]$-module of modular forms of weight $k$ (and level $\Gamma_1(N)$) over $R$, denoted $M_{k,R}$, is the $R$-submodule consisting of $f \in H^0(X_R, \omega^k \otimes_{\Z_p} Q(R))$ such that $\lr{T,f} \in R \subset Q(R)$ for all $T \in \bT[T_p]$. For brevity, let $M_k := M_{k,\Z_p}$.  
\item The $\bT[T_p]$-module of cusp forms of weight $k$ (and level $\Gamma_1(N)$) over $R$, denoted $S_{k,R}$, is $H^0(X, \omega^k(-C) \otimes_{\Z_p} R)$. For brevity, let $S_k := S_{k,\Z_p}$. 
\item Modular/cusp forms of weight $k$ and relative level $\Gamma_0(p)$ over $R$, $M_{k,R}(\Gamma_0(p))$ and $S_{k,R}(\Gamma_0(p))$, are defined similarly. 
\item When $f$ is an eigenform for the $\lr{d}_N$-operators over $R$, write $\chi = \chi_f : (\Z/N\Z)^\times \to R^\times$ its character. 
\end{itemize}
\end{defn}

When $R$ is a $\Z_p$-algebra domain, all of the modules of classical forms of Definition \ref{defn: classical forms} are known to make the pairing \eqref{eq: a1 duality} perfect. 

We use the notation $e(T) = \lim_{n \to \infty}T^{n!}$ for a $p$-integral endomorphism $T$ of a $\Z_p$-module when the axioms of explained in \cite[\S2.2]{BP2022} are satisfied, making $e(T)$ an idempotent \emph{$T$-ordinary projector}. By mild abuse of notation, we also use $e(T)V$, when $V$ is a finite-dimensional $\Q_p$-vector space, to refer to the natural construction of the $T$-ordinary summand of $V$. 

\begin{rem}
We mostly use standard (not dual) Hecke operators on the multiplicative model $X$. In comparison with some of our principal references, Coleman uses the additive model in \cite{coleman1996}, while Boxer--Pilloni use the multiplicative model in \cite{BP2022}. We will apply \cite[\S1.4]{FK2012} to translate between these when needed. 
\end{rem}

The divisor of supersingular points $SS$ of $X_{\F_p}$ has open complement called the ordinary locus $X_{\F_p}^\ord$, which is connected. We write $X^\ord/\Spf \Z_p$ for the $p$-adic completion of $X$ along $X_{\F_p}^\ord$. Because $X_{\F_p}$ is connected, $p$-adic modular forms $f \in H^0(X^\ord, \omega^k)$ continue to be characterized by their $q$-series. 

\subsection{Coherent cohomology and the Gauss--Manin connection}
\label{subsec: coherent cohom}

We will study Hecke actions on both degrees of coherent cohomology of $\omega^k$ and $\omega^k(-C)$ over $X$, which will use the Kodaira--Spencer map to access Serre duality. We will do this $p$-integrall, over $\Z_p$. We write $\Omega^1$ for the canonical sheaf $\Omega_{X/\Z_p}^1$ when the context is clear. 

For this we need the relative de Rham cohomology bundle of the universal elliptic curve with logarithmic singularities at the cusps, denoted $\cH = \cH_1$ (see \cite[p.\ 218]{coleman1996}),  along with the following associated objects. We will use its symmetric powers $\cH_{k-2} := \Sym^{k-2}\cH$ for $k \in \Z_{\geq 2}$, its Gauss--Manin connection $\nabla = \nabla_1: \cH \to \cH \otimes \Omega^1_X(\log C)$, its Poincar\'e duality self-pairing $\lr{}_1$, and the natural extensions of $\nabla$ and $\lr{}_1$ to symmetric powers, denoted $\nabla_{k-2} : \cH_{k-2} \to \cH_{k-2} \otimes \Omega^1_X(\log C)$ and $\lr{}_{k-2}$. The Hodge filtration on $\cH$ induces a Hodge filtration on all $\cH_{k-2}$, and can be viewed as a short exact sequence
\[
0 \to \omega \to \cH \to \omega^{-1} \to 0. 
\]
All of these objects, as well as a splitting of $\cH$ into $\omega^{-1} \oplus \omega$, are well defined over $\Z_p$ due to our running assumption that $p \geq 5$ and $p \nmid N$ (see Lemma \ref{lem: coord integrality}).

From these objects, we construct the Kodaira--Spencer isomorphism 
\[
\mathrm{KS} : \omega \otimes \omega \to \Omega^1(C), \quad f \otimes g \mapsto \langle f, \nabla g\rangle_1
\]
where we consider $f,g$ to be sections of $\cH$ under the inclusion $\iota : \omega \subset \cH$. And its extension to weights $k$ greater than 2, 
\[
\mathrm{KS}_k = \mathrm{KS} \otimes \iota_{k-2} : \omega^k \cong \omega \otimes \omega \otimes \omega^{k-2} \buildrel{\mathrm{KS} \otimes \iota_{k-2}}\over\lra  \cH_{k-2} \otimes \Omega^1(C),
\]
where we use the canonical inclusion $\iota_{k-2}$, the $(k-2)$nd tensor power of $\iota$. 

We will need the following standard Hecke equivariance formula for Serre duality, using $\mathrm{KS}$ for the isomorphism $\omega^2(-C) \cong \Omega^1$ with the canonical sheaf. 
\begin{prop}
    \label{prop: Serre duality}
    Let $k \in \Z_{\geq 2}$. Serre duality and the Kodaira--Spencer map produce perfect pairings of finitely generated flat $\Z_p$-modules 
    \begin{gather*}
        \lr{}_\mathrm{SD} : H^0(X, \omega^k) \times H^1(X,\omega^{2-k}(-C)) \to H^1(X,\Omega^1_X) \cong \Z_p \\
        \lr{}_\mathrm{SD} : H^0(X, \omega^k(-C)) \times H^1(X,\omega^{2-k}) \to H^1(X,\Omega^1_X) \cong \Z_p
    \end{gather*}
    under which each Hecke operator's dual Hecke operator is its adjoint. That is, 
    \begin{align*}
        \lr{T_n f,g}_\mathrm{SD} &= \lr{f, T_n^* g}_\mathrm{SD} \quad \text{for all } n \in \Z_{\geq 1} \\ 
        \lr{T_n f, g}_\mathrm{SD} &= \lr{f, T_n \lr{n}^{-1}g}_\mathrm{SD} \quad \text{ if } (n,N) = 1 \\
        \lr{\lr{d}_N f,g}_\mathrm{SD} &= \lr{f,\lr{d}_N^{-1}g}_\mathrm{SD} \quad (d,N) = 1.
    \end{align*}
\end{prop}

\subsection{Overconvergent modular forms}

\begin{defn}
\label{defn: oc forms over Cp}
For $k \in \Z$, we let $M_{k,\C_p}^\dagger$ denote the overconvergent modular forms over $X$ (that is, of level $\Gamma_1(N)$) defined in \cite[\S2]{coleman1996} as the sections of $\omega^k$ over the analytic subspace denoted by $W_1 \subset X^\mathrm{an}$ there. Let $S_{k,\C_p}^\dagger$ denote its submodule of cuspidal overconvergent modular forms. That is, $S_{k,\C_p}^\dagger$ consists of those sections that vanish when evaluated at $C$. This pair is equipped with an action of $\bT[U]$ where $U$ is defined in \cite[\S3]{coleman1996}. 
\end{defn}

When $M'_{\C_p} \subset M_{k,\C_p}^\dagger$ is a $\C_p$-subspace, we let $M'^0$ denote its submodule consisting of forms with trivial residues at every supersingular annulus, as introduced in \cite[\S6]{coleman1996}. 
As explained in \cite[p.\ 226]{coleman1996}, there are a natural inclusions
\begin{equation}
    \label{eq: classical inclusion}
    M_k(\Gamma_0(p))_{\C_p} \rinj M^\dagger_{k,\C_p}, \quad 
    S_k(\Gamma_0(p))_{\C_p} \rinj S^\dagger_{k,\C_p}
    \end{equation}
equivariant with respect to the usual map of abstract Hecke algebras $\bT[U_p] \to \bT[U]$. The intersection $S_k(\Gamma_0(p))_{\C_p}^0$ is known to consist exactly of those cusp forms that are $p$-old \cite[Thm.\ 9.1]{coleman1994}. 

The normalized valuation of the $U$-eigenvalue of a $U$-eigenform $f \in M_{k,\C_p}^\dagger$ is called its \emph{slope}. The term \emph{ordinary} or \emph{$U$-ordinary} refers to slope $0$. When $k \in \Z_{\geq 1}$, the slope $k-1$ in weight $k$ is dubbed \emph{critical} or \emph{$U$-critical}. In general, when $M' \subset M_{k,\C_p}^\dagger$ is a $\bT[U]$-submodule, we write $M'^\ord, M'^\crit$ for the $U$-ordinary and $U$-critical subspaces, respectively. In particular, we will use the notation
\[
M_{k,\C_p}^{\dagger,\ord}, M_{k,\C_p}^{\dagger,\crit}, \qquad 
S_{k,\C_p}^{\dagger,\ord}, S_{k,\C_p}^{\dagger,\crit}
\]
for the $\bT[U]$-submodules on which $U$ is ordinary (resp.\ critical). These spaces are known to be finite-dimensional. We will discuss the ordinary case more in \S\ref{subsec: hida higher background}. There are also the classical counterparts $M_{k,R}^{\ord}$, $M_{k,R}^{\crit}$, etc., for $R \subset \C_p$, where ``classical'' is taken to be the image under the inclusions \eqref{eq: classical inclusion}.

The notion of $q$-series extends to overconvergent modular forms, and the $q$-series continues to characterize a form because the ordinary locus $X_{\F_p}^\ord$ is connected.

\begin{rem}
    The distinction between ``over $R$'' and ``coefficients in $R$'' is no more serious than in the case of classical modular forms as in Definition \ref{defn: classical forms}. The reason for this is that the space of (classical) Eisenstein series $\mathrm{Eis}_k(\Gamma_0(p))_{\C_p}$ has a natural isomorphism onto the space of values of residues at the cusps achieved by $M_{k,\C_p}^\dagger$. This follows from \cite[Lem.\ 6.7]{coleman1996}, where we see that $\dim_{\C_p} \mathrm{Eis}_k(\Gamma_0(p))_{\C_p}$ equals the number of cusps when $k > 2$, is one less than the number of cusps when $k=2$, and the the complementary Eisenstein Hecke eigensystem in weight 2 is known to be associated to a $p$-adic modular form, sometimes denoted ``$E_2$,'' which is not overconvergent \cite{CGJ1995}. 
\end{rem}

Since the Hecke action on $M^{\dagger, \crit}_{k,\C_p}$ is known to be non-diagonalizable, we carefully record its perfect duality with its Hecke algebra. 
\begin{prop} 
    \label{prop: oc duality}
    Let $k \in \Z_{\geq 1}$. There are perfect pairings
    \begin{gather*}
    \lr{} : \bT[U](M^{\dagger,\crit}_{k,\C_p})  \times M^{\dagger,\crit}_{k,\C_p} \to \C_p \\    
    \lr{} : \bT[U](S^{\dagger,\crit}_{k,\C_p}) \times S^{\dagger,\crit}_{k,\C_p} \to \C_p 
    \end{gather*}
    defined by $\lr{T,f} = a_1(T \cdot f)$ as in \eqref{eq: a1 duality}. The pairings are Hecke-compatible in the sense that for all $T' \in \bT[U]$, $\lr{T' \cdot T, f} = \lr{T, T' \cdot f}$. 
\end{prop}

\begin{proof}
    We will address the first pairing; the second pairing can be dealt with in the same way. 
    
    We have mentioned that overconvergent forms are known to be characterized by their $q$-series. Because the $n$th Hecke operator $T_n \in \bT[U]$ (using the notation of Definition \ref{defn: Tn for general n} satisfies $a_1(T_n \cdot f) = a_n(f)$ for all $n \geq 1$, and there are no forms with constant $q$-series when $k \neq 0$, we know that the right kernel of the pairing is zero. 

    On the other hand, the left kernel of $\lr{}$ is trivial because $\bT[U](M^{\dagger,\crit}_{k,\C_p})$ acts faithfully on $M^{\dagger,\crit}_{k,\C_p}$. This follows from a standard argument we now recite. For non-zero $T \in \bT[U](M^{\dagger,\crit}_{k,\C_p})$, there exists $g \in M^{\dagger,\crit}_{k,\C_p}$ such that $T \cdot g \neq 0$. Therefore $T\cdot g$ has some non-zero Fourier coefficient, say $a_n(T \cdot g) \neq 0$. Thus 
    \[
    \lr{T, T_n \cdot g} = \lr{T_n \cdot T, g} = a_1(T_n \cdot T \cdot g) = a_n(T \cdot g) \neq  0,
    \]
    showing $T$ is not in the left kernel. 
\end{proof}

\subsection{De Rham Cohomology}

\label{subsec: dR cohom}

We will use the following modules of $p$-adic analytic de Rham cohomology in weights $k \in \Z_{\geq 2}$ studied by Coleman \cite[\S5]{coleman1996}, along with comparisons with algebraic de Rham cohomology and crystalline cohomology. Here we will mostly follow \cite{coleman1996} except that we work over $\Q_p$ and sometimes over $\Z_p$, but we will align our notation with \cite{BP2022}. Let $\cI_C \subset \cO_X$ denote the ideal sheaf of the cusps. We will refer to $W_1 \subset X^\mathrm{an}$ as in Definition \ref{defn: oc forms over Cp}. 

\begin{defn}
    \label{defn: dR cohomology}
    We set up analytic de Rham cohomology of weight $k \in \Z_{\geq 2}$, $H^1_\dR(X_{\Q_p},\cF_k)$ and $H^1_\mathrm{par}(X_{\Q_p},\cF_k)$, following Coleman \cite[\S5]{coleman1996}. 
    We let $\cF_k^\mathrm{an}$ be the complex (with differential from degree 0 to degree 1)
    \[
    \cF_k^\mathrm{an} = [\cH_{k-2}(W_1) \buildrel{\nabla_{k-2}}\over\lra (\cH_{k-2} \otimes \Omega^1(\log C))(W_1)]
    \]
    and use the following notation for its 1st cohomology and certain subspaces. 
    \begin{gather*}
        \Han := \frac{(\Omega^1(\log C) \otimes \cH_{k-2})(W_1)}{\nabla_{k-2} \cH_{k-2}(W_1)}\\ 
        \Han^0 := \ker(H^1_\dR(X_{\Q_p}, \cF_k) \buildrel\mathrm{res}\over\lra [\text{SS annuli}]) \\
        \Hanpar := \ker(\Han \buildrel\mathrm{res}\over\lra [\text{SS annuli } \cup C]),
    \end{gather*}
    where we remark that 
    \begin{itemize}
        \item the annuli and the divisor $C$ can be defined over $\Q_p$ (using the Eisenstein series $E_{p-1}$ for the supersingular annuli, as in \cite{GK2009})
        \item by \cite[Prop.\ 2.21]{MC2010} we see that a model for $W_1$ over $\Q_p$ exists, even though the points of $\widetilde{SS}$ are defined only over the unramified quadratic extension $\Q_{p^2}/\Q_p$.
    \end{itemize}

    The containments of analytic cohomology $\Han \supset \Han^0 \supset \Hanpar$ is stable under the action of $\bT[T_p]$ as well as additional $p$-adic operators 
    \begin{itemize}
        \item $U$, associated to Verschiebung
        \item $F'$, associated to Frobenius: we are using the notation of \cite[\S4.2.3]{BP2022}; this $F'$ is denoted by ``$F$'' in \cite{coleman1996} 
        \item $F$, the Frobenius map defined in \cite[\S4.2.3]{BP2022}
        \item obeying relations $U \circ F' = F' \circ U = p^{k-1}$.
    \end{itemize}

    There is also an analytic Poincar\'e self-duality pairing $\lr{}_\dR^\mathrm{an}$ on $\Hanpar$ that is perfect and under which $U$ and $F'$ are adjoint, discussed more in Proposition \ref{prop: FV on dR}.
\end{defn}

Now we state a comparison between analytic de Rham cohomology and algebraic de Rham cohomology. For brevity, for $k \in \Z_{\geq 2}$ we write $\cF_k$ for the coefficient system over $X$ consisting of the $(k-2)$th symmetric power $\nabla_{k-2}$ of the Gauss--Manin connection introduced in \S\ref{subsec: coherent cohom}. We call its algebraic de Rham cohomology $\bH^*_\dR(X_{\Q_p}, \cF_k)$ weight $k$ de Rham cohomology. We set up the parabolic variant $\bH^1_\mathrm{par}(X_{\Q_p}, \cF_k)$, the image of $\bH^1(X_{\Q_p}, \cF_k \otimes \cI_C) \to \bH^1_\dR(X_{\Q_p}, \cF_k)$.  

Let $\widetilde{SS}$ be the divisor on $X_{\Q_p}$ cut out by $E_{p-1}$, lifting the divisor $SS \subset X_{\F_p}$. We remark that this divisor is $\bT[T_p]$-stable. 

\begin{lem}
\label{lem: dR equivalence}
Let $k \in \Z_{\geq 2}$. There are natural $\bT[T_p]$-equivariant isomorphisms between $p$-adic analytic de Rham cohomology and the algebraic de Rham cohomology modules as below, arising from restriction from $X_{\Q_p}$ to $W$, 
\begin{enumerate}[label=(\roman*)]
\item $\bH^1(X_{\Q_p}, \cF_k(\log \widetilde{SS})) \cong \Han$, 
\item $\bH^1_\mathrm{dR}(X_{\Q_p}, \cF_k) \cong \Han^0$, 
\item $\Hanpar \subset \Han$ receives an isomorphism from the image under (i) of $\bH_\mathrm{par}^1(X_{\Q_p}, \cF_k)$. 
\end{enumerate}    
\end{lem}

\begin{proof}
See \cite[Thm.\ 5.4]{coleman1996} for (i). Then (ii) and (iii) follow from an analysis of residues (as in e.g.\ \cite[\S6]{coleman1996}) following \cite[Prop.\ 4.4]{coleman1989}. These references address the situation over $\C_p$. We can descend the comparisons to $\Q_{p^2}$ due to the fact that the points of $\widetilde{SS} \subset X(\C_p)$ are valued in $\Q_{p^2}$ by applying \cite[Cor.\ 4.10c]{coleman1996} (here we have a non-trivial coefficient system unlike \textit{ibid.}, but the coefficient system is defined over $\Q_p$.) The descent from $\Q_{p^2}$ to $\Q_p$ follows from the density statement of \cite[Prop.\ 2.21]{MC2010}. 
\end{proof}

We will also need a comparison with crystalline cohomology over $\Z_p$, which is classical. We let $H^*_\cris(X_{\F_p} / \Z_p, \cF_k)$ denote the crystalline cohomology with coefficients associated to the Gauss--Manin connection $(\cH_{k-2}, \nabla_{k-2})$. 

\begin{prop}
    \label{prop: cris dR}
    There is a canonical isomorphism 
    \[
    \bH^*_\mathrm{par}(X, \cF_k) \cong H^*_\cris(X_{\F_p} / \Z_p, \cF_k).
    \]
\end{prop}
\begin{proof}
    See \cite[Thm.\ 7.1.]{BO1978}. 
\end{proof}

We will need to understand how the Poincar\'{e} duality pairings on $H^1_\dR(X_{\Q_p}^\mathrm{an},\cF_k)$ and $\Hanpar$ are calculated. 
\begin{prop}
    \label{prop: PD and SD}
    Under the isomorphisms of Lemma \ref{lem: dR equivalence}, the following Poincar\'{e} duality pairings on the parabolic parts are equal. 
    \begin{itemize}
        \item The algebraic pairing $\lr{}_\dR$ given by $\langle \rangle_{k-2}$ on $\cH_{k-2}$ and the Serre duality pairing $\Omega_X^\bullet \otimes \Omega_X^\bullet \to \Q_p[1]$. 
        \item The analytic pairing $\lr{}_\dR^\mathrm{an}$ given by 
        \begin{equation}
            \label{eq: dR pairing}
            \langle f,g\rangle_\dR^\mathrm{an} = \sum_{x \in SS} \mathrm{res}_x(\lambda_x \cdot g)    
        \end{equation}
        where $\lambda_x$ is a local primitive of $f$ in the annulus around $x \in SS$, that is, $\nabla_{k-2} \lambda_x = f_x$ on the annulus. 
    \end{itemize}
\end{prop}

When it is safe to identify the analytic and algebraic presentations of cohomology, we will denote both of these pairings by $\lr{}_\dR$. 

\begin{proof}
    The local primitive $\lambda_x$ exists because, by Lemma \ref{lem: dR equivalence}, parabolic classes have trivial residues around $SS$ annuli. See \cite[Prop.\ 4.5]{coleman1989} which, along with adding the coefficient system $\cF_k$ defined over $X_{\Q_p}$, establishes the desired result over $\C_p$. Thus we can deduce the descent of the analytic pairing to $\Q_p$ from the descent of the algebraic pairing. 
\end{proof}

We will need to explicitly calculate the maps realizing the Hodge filtration on $H^1_\dR(X^\mathrm{an},\cF_k)^0$ via both algebraic and analytic presentations.

\begin{prop}
    \label{prop: Hodge SES}
    Let $k \in \Z_{\geq 3}$. There are the following two canonical short exact sequences expressing the Hodge filtration on de Rham cohomology, 
    \begin{gather*}
    0 \to H^0(X_{\Q_p},\omega^k) \to \begin{matrix}\bH_\dR^1(X_{\Q_p},\cF_k) \\ \parallel \\ \Han^0\end{matrix} \to H^1(X_{\Q_p},\omega^{2-k}) \to 0 \quad \text{and} \\
    0 \to H^0(X_{\Q_p},\omega^k(-C)) \to \begin{matrix}\bH_\mathrm{par}^1(X_{\Q_p},\cF_k) \\ \parallel \\ \Hanpar \end{matrix} \to H^1(X_{\Q_p},\omega^{2-k}) \to 0. 
    \end{gather*}
    The analytic/algebraic isomorphisms of Lemma \ref{lem: dR equivalence}, rendered as the vertically oriented equals signs in the two short exact sequences above, restrict to isomorphisms between the analytic and algebra formulations of the single non-trivial step in the Hodge filtration on de Rham cohomology, which are as follows. 
    \begin{itemize}
        \item Algebraic formulation: $\nabla_{k-2}$ induces an isomorphism (\cite[Lem.\ 4.2]{coleman1996}) 
        \[
        \Fil^1 \cH_{k-2} \isoto \frac{\cH_{k-2} \otimes \Omega^1_X(\log C)}{ \Fil^{k-2} \cH_{k-2} \otimes \Omega_X^1(\log C)}
        \]
        resulting in degeneration of the hypercohomology spectral sequence to the short exact sequence above, with the left map realized by applying $H^0(X_{\Q_p}, -)$ to $\mathrm{KS}_k : \omega^k \to \cH_{k-2} \otimes \Omega_X^1(\log C)$ and the quotient map realized by the canonical isomorphism $\cH_{k-2}/\Fil^1 \cH_{k-2} \cong \omega^{2-k}$.
        \item Analytic formulation: the left map is $\mathrm{KS}_k$ followed by restriction to $W$; the right map arises from Serre duality (Proposition \ref{prop: Serre duality}), sending $\eta \in H^1_\dR(X_{\Q_p}^\mathrm{an},\cF_k)^0$ to the map $H^0(X_{\Q_p}, \omega^k(-C)) \to \Q_p$ given by 
        \[
        H^0(X_{\Q_p}, \omega^k(-C)) \ni f \mapsto \langle\mathrm{KS}_k(f), \eta \rangle_\dR \in \Q_p. 
        \]
    \end{itemize}
    In particular, the $\lr{}_\dR$ respects the Hodge filtration (in that the Hodge subspace is isotropic) and induces Serre duality on the graded factors of the Hodge filtration of $H^1_\mathrm{par}(X_{\Q_p}, k)$. Also, the left kernel of $\lr{}_\dR$ is the subspace spanned by Eisenstein series.
\end{prop}

\begin{proof}
    The compatibility between the algebraic and analytic realizations of the Hodge subspace follows from \cite[Thm.\ 4.10]{coleman1989}. The compatibility between the algebraic and analytic realizations of the Hodge quotient arises from the compatibility of Proposition \ref{prop: PD and SD} between Serre duality and the analytic expression of $\lr{}_\dR$. The final statement follows from \cite[\S3, Remark]{coleman1994}. 
\end{proof}

Later, we will also need an integral refinement of the Hodge filtration. 
\begin{cor}
    \label{cor: integral Hodge SES}
    Let $k \in \Z_{\geq 3}$. There is a sequence of finitely generated $\Z_p$-modules
    \[
    H^0(X, \omega^k) \rinj \bH^1_\mathrm{par}(X, \cF_k) \rsurj H^1(X, \omega^{2-k})
    \]
    where the composition is zero. After $\otimes_{\Z_p} \Q_p$, it realizes the short exact sequence of Proposition \ref{prop: Hodge SES}. 
\end{cor}
\begin{proof}
    Although the spectral sequence of $\bH^1_\mathrm{par}(X, \cF_k)$ may no longer fully degenerate, it nonetheless begins with this injection and ends with this surjection claimed. (More explicitly, failures of Griffiths transversality are at most torsion, as is described in Lemma \ref{lem: torsion complement}.) By naturality of these maps, the final claim is true. 
\end{proof}

And we discuss Hecke equivariance of the Poincar\'{e} duality pairing. 
\begin{prop} 
    
    \label{prop:PD-hecke}
    The Hodge filtration is $\bT[T_p]$-stable. The same adjunction formula for Hecke operators holds for the Poincar\'{e} duality pairing $\lr{}_\dR$ as for the Serre duality pairing $\lr{}_\mathrm{SD}$ as given in Proposition \ref{prop: Serre duality}. 
\end{prop}

\begin{proof}
See \cite[\S{8}]{coleman1994}.
\end{proof}

Coleman gives more description of the $U$ and $F'$ actions on cohomology introduced in Definition \ref{defn: dR cohomology}. 

\begin{prop}
    \label{prop: FV on dR}
    The actions $F'$ and $U$ are adjoint on $\bH^1_\mathrm{par}(X_{\Q_p},\cF_k)$ with respect to $\lr{}_\dR$, that is, $\lr{U \cdot \alpha, \beta}_\dR = \lr{\alpha, F'\cdot \beta}_\dR$. Moreover, 
    \[
    F' \circ U = U \circ F' = p^{k-1}
    \]
    and 
    \[
    \lr{F' \cdot \alpha, F' \cdot \beta}_\dR = \lr{U \cdot \alpha, U \cdot \beta}_\dR = p^{k-1}\lr{\alpha,\beta}_\dR. 
    \]
    Under the comparison with crystalline cohomology of Proposition \ref{prop: cris dR}, $F'$ is compatible with the crystalline Frobenius $\varphi$. 
\end{prop}

\begin{proof}
    These statements are recorded in \cite[\S5]{coleman1996}, other than the final claim which appears in \cite[Lem.\ 4.3.1]{BE2010}. 
\end{proof}

\subsection{Hida theory and higher Hida theory}
\label{subsec: hida higher background}

We recall the results of Hida theory and the higher Hida theory for the modular curve, following Boxer--Pilloni \cite{BP2022}. We choose a slightly different $p$-integral lattice than Boxer--Pilloni. 

We record this standard result. 
\begin{lem}
    \label{lem: oc Hida}
    Let $k \in \Z$. 
    $U$-ordinary global sections of $\omega^k, \omega^k(-C)$ on $X^\ord$ are overconvergent, that is, the natural restriction map for $X^{\ord,\mathrm{an}} \subset W$ induces isomorphisms
    \[
    M_{k,\Q_p}^{\dagger,\ord} \cong e(U)H^0(X^\ord,\omega^k) \otimes_{\Z_p} \Q_p, \quad 
    S_{k,\Q_p}^{\dagger,\ord} \cong e(U)H^0(X^\ord,\omega^k(-C)) \otimes_{\Z_p} \Q_p.
    \]
\end{lem}

\begin{defn}
Let $k \in \Z$. For $k \neq 0$, let $M_k^{\dagger,\ord}, S_k^{\dagger,\ord}$ denote the the $\Z_p$-lattices in $M_{k,\Q_p}^{\dagger,\ord}, S_{k,\Q_p}^{\dagger,\ord}$ given by pairing with $\bT[U]$, that is, 
\[
M_k^{\dagger,\ord} := \{f \in e(U)H^0(X^\ord,\omega^k) \otimes_{\Z_p} \Q_p : a_1(T \cdot f) \in \Z_p \text{ for all } T \in \bT[U]\}
\]
and similarly for $S^{\dagger,\ord}_k$. Let $M_k^\ord$ denote 
\[
\{f \in e(U)H^0(X_0(p),\omega^k) \otimes_{\Z_p} \Q_p : a_1(T \cdot f) \in \Z_p \text{ for all } T \in \bT_{\Gamma_0(p)}\}
\]
and similarly for $S^\ord_k$. 
\end{defn}
Hida's classicality theorem ensures that the inclusion $M_k^\ord \rinj M_k^{\dagger,\ord}$ is an isomorphism for $k \geq 3$. 

Here are the stabilization isomorphisms of classical and higher Hida theory. 

\begin{prop}
    \label{prop: duality wt k}
    For $k \in \Z_{\geq 3}$, there are canonical $\bT[\,]$-equivariant stabilization isomorphisms
    \begin{gather*}
        e(T_p) H^0(X,\omega^k) \risom e(U) H^0(X^\ord, \omega^k) \\
        e(T_p) H^0(X,\omega^k(-C)) \risom e(U) H^0(X^\ord, \omega^k(-C)) \\
        e(F)H^1_c(X^\ord,\omega^{2-k}) \risom e(T_p) H^1(X, \omega^{2-k}) \\
        e(F)H^1_c(X^\ord,\omega^{2-k}(-C)) \risom e(T_p) H^1(X, \omega^{2-k}(-C)) 
    \end{gather*}
    Under these isomorphisms $T_p$ acts over $X$ as $U + F$ acts over $X^\ord$. 
    Moreover, there are $\Z_p$-perfect $p$-adic Serre duality pairings 
    \begin{gather}
        \label{eq: SD ord}
        \lr{}_\mathrm{SD} : e(U)H^0(X^\ord, \omega^k) \times e(F)H^1_c(X^\ord,\omega^{2-k}(-C)) \to  \Z_p \\
        \label{eq: SD cusp ord}    
        \lr{}_\mathrm{SD} : e(U)H^0(X^\ord, \omega^k(-C)) \times e(F)H^1(X^\ord,\omega^{2-k}) \to  \Z_p    
    \end{gather}
    that are compatible with the Serre duality pairings of Proposition \ref{prop: Serre duality} under the stabilization isomorphisms.
\end{prop}

\begin{proof}
    For classical Hida theory, see \cite[Prop.\ 1.3.2]{ohta2005}, while for higher Hida theory, see Boxer--Pilloni \cite[Thm.\ 4.18]{BP2022}. 
\end{proof}

We are interested in a slightly different lattice, differing only in the Eisenstein part. Indeed, we have already specified this by defining $M_k^{\dagger,\ord}$, and we need an appropriate definition in degree 1. 
\begin{defn} 
    \label{defn: Hecke lattice for H1ord}
    Let $\cH^{1,\ord}_k(-C) \subset e(F)H^1_c(X^\ord, \omega^{2-k}(-C)) \otimes_{\Z_p} \Q_p$ denote the $\Z_p$-lattice that is the $\Z_p$-perfect dual under \eqref{eq: SD ord} of $M^{\dagger,\ord}_k \subset e(U) H^0(X^\ord, \omega^k) \otimes_{\Z_p} \Q_p$. 
\end{defn}

Now we set up the $\Lambda$-adic interpolations of classical and higher Hida theory \cite{hida1986, wiles1988, BP2022}. 
\begin{itemize}[leftmargin=1.2em]
\item We let $\Lambda = \Z_p\lb \Z_p^\times \rb$, writing $[z] \in \Z_p^\times$ for the group elements of $\Z_p^\times$. We let $\Lambda$ act in weight $k \in \Z$ by $\phi_k : \Lambda \to \Z_p, [z] \mapsto z^{k-1}$. 
\item $\Lambda$-adic $U$-ordinary modular forms of tame level $N$ are $M_\Lambda^\ord$ denoted $M_\Lambda^\ord$ and have the interpolation property
\begin{equation}
\label{eq: M specialize}
M_\Lambda^\ord \otimes_{\Lambda,\phi_k} \Z_p \isoto M_k^{\dagger,\ord}, \quad k \in \Z. 
\end{equation}
We remark that this definition of $M_\Lambda^\ord$ differs from the module ``$M$'' of \cite[Thm.\ 1.2]{BP2022} because we interpolate the lattices $M_k^{\dagger,\ord}$ while the authors of \textit{ibid.}\ interpolate $e(U)H^0(X^\ord,\omega^k)$. 
\item The cusp forms are $S_\Lambda^\ord$, which interpolate the lattices $S_k^{\dagger,\ord}$. 
\item Let $\bT_\Lambda^\ord := \bT[U]_\Lambda(M_\Lambda^\ord)$, $\bT_\Lambda^{\ord,\circ} := \bT[U]_\Lambda(S_\Lambda^\ord)$.
\item The perfect $\Z_p$-bilinear duality pairings of \eqref{eq: a1 duality} interpolate under \eqref{eq: M specialize} into perfect $\Lambda$-bilinear duality pairings
\begin{equation}
    \label{eq: a1 duality Lambda}
    \lr{} : \bT_\Lambda^\ord \times M_\Lambda^\ord \to \Lambda, \qquad 
    \lr{} : \bT_\Lambda^{\ord,\circ} \times S_\Lambda^\ord  \to \Lambda
\end{equation}
which are $\bT[U]$-compatible. 
\item $M_\Lambda^\ord$ has a $q$-series realization compatible with the usual action of $\bT[U]_\Lambda$, $M_\Lambda^\ord \rinj Q(\Lambda) \oplus q \Lambda\lb q\rb$, of the form $\cF \mapsto a_0(\cF) + \sum_{n \geq 1} a_1(T_n \cdot \cF)q^n$ (using the constant term operator). 
\item $\Lambda$-adic $F$-ordinary cuspidal coherent cohomology in degree 1 is 
\[
\cH^{1,\ord}_\Lambda := e(F)H^1_c(X^\ord, \omega^{2-\kappa^\mathrm{un}}),
\]
where $\omega^{\kappa^\mathrm{un}}$ is the Igusa sheaf \cite[\S4.2.1]{BP2022}. It is characterized by $\bT[F]$-equivariant specialization maps
\[
\cH^{1,\ord}_\Lambda \otimes_{\Lambda, \phi_k} \Z_p \cong e(F)H^1_c(X^\ord,\omega^{2-k}). 
\]
\end{itemize}

\begin{thm}[{Hida, Boxer--Pilloni \cite[Thms.\ 1.1, 1.2]{BP2022}}]
    \label{thm: BP}
    $M_\Lambda^\ord$, $S_\Lambda^\ord$,  and $\cH^{1,\ord}_\Lambda$ are finitely generated flat $\Lambda$-modules. There are canonical specialization isomorphisms for $k \in \Z$, \eqref{eq: M specialize} and 
    \[
    S_\Lambda^\ord \otimes_{\Lambda,\phi_k} \Z_p \cong S_k^{\dagger,\ord}, \qquad 
\cH^{1,\ord}_\Lambda \otimes_{\Lambda, \phi_k} \Z_p  \cong e(F)H^1_c(X^\ord, \omega^{2-k}). 
    \]
    There is a perfect $\Lambda$-adic Serre duality pairing 
    \[
    \lr{}_{\mathrm{SD},\Lambda} :  S_\Lambda^\ord \times \cH^{1,\ord}_\Lambda \to \Lambda
    \]
    that interpolates the $p$-adic Serre duality pairings in weight $k \in \Z_{\geq 3}$ of Proposition \ref{prop: duality wt k} upon specialization along $\phi_k$. 
\end{thm}

In addition, the duality of Hecke actions under Serre duality (of Proposition \ref{prop: Serre duality}) interpolates and also extends to the $p$-adic operators $U$ and $F$ as follows. 

\begin{prop}[{\cite[Thm.\ 4.18]{BP2022}}]
\label{prop: interpolated Serre duality}
Under the Serre duality pairings of Theorem \ref{thm: BP}, one has the same adjoint Hecke operator formulas as Proposition \ref{prop: Serre duality} for Hecke operators relatively prime to $p$, while duality for the remaining Hecke operators is determined by
\[
\lr{\lr{p}_N^{-1} U \cdot f, g}_{\mathrm{SD},\Lambda} = \lr{f, F\cdot g}_{\mathrm{SD},\Lambda},
\]
which also specializes to the Serre duality pairing at each weight $k \in \Z$. 
\end{prop}

\subsection{Presentation of cohomology via overconvergent forms}

Let $\theta = q\frac{d}{dq}$ denote the Atkin--Serre differential operator on $p$-adic modular forms, which increases weight by $2$. When $X$ is a $\bT[U]$-module we let $X(i)$ denote its $i$th Tate twist for $i \in \Z$, which we can view (for $i \geq 0$) as a tensor product as in \cite[p.\ 233]{coleman1996}. It is known that for $k \in \Z_{\geq 2}$ we have a $\bT[U]$-equivariant map $\theta^{k-1} : M^\dagger_{2-k, \Q_p}(k-1) \to M^\dagger_{k,\Q_p}$, that is, $\theta^{k-1}$ preserves overconvergence when it acts on forms of weight $2-k$. Its image also has vanishing residues along supersingular annuli. The $\bT[U]$-equivariance follows from the straightforward verification that $\theta^{k-1}$ satisfies 
\begin{equation}
    \label{eq: theta twist}
    T_n \circ \theta^{k-1} = n^{k-1}\theta^{k-1} \circ T_n \quad \text{for all } n \in \Z_{\geq 1}. 
\end{equation}
In particular, $\theta^{k-1}$ increases $U$-slopes by $k-1$. Thus we have an $\bT[U]$-equivariant map of $U$-critical modules 
\begin{equation}
\label{eq: theta}
\theta^{k-1} : M^{\dagger,\ord}_{2-k, \Q_p}(k-1) \to M^{\dagger,\crit}_{k,\Q_p},    
\end{equation}
which will be a main object of study.

The principal theorem of \cite{coleman1996} presents de Rham cohomology of weight $k$ by realizing overconvergent forms as differentials of the second kind. 
\begin{prop}[{Coleman \cite[Thm.\ 5.4]{coleman1996}}]
    \label{prop: Coleman pres dR}
    There are canonical isomorphisms induced by $\KS_k: M_k^\dagger \to (\cH_{k-2} \otimes \Omega^1)(W_1)$, 
    \[
     \frac{
M_{k,\Q_p}^{\dagger}}{\theta^{k-1} M_{2-k,\Q_p}^{\dagger}} \cong H^1_\dR(X_{\Q_p}^\mathrm{an},\cF_k),
  \qquad   \frac{
M_{k,\Q_p}^{\dagger,0}}{\theta^{k-1} M_{2-k, \Q_p}^{\dagger}} \cong H^1_\dR(X_{\Q_p}^\mathrm{an},\cF_k)^0
\]
\[
    \frac{
S_{k,\Q_p}^{\dagger,0}}{\theta^{k-1} M_{2-k,\Q_p}^{\dagger}} \cong \Hanpar
    \]
    which are $\bT[U]$ equivariant. 
\end{prop}

\begin{proof}
All of these statements over $\C_p$ follow in a straightforward way from \cite[Thm.\ 5.4]{coleman1996}, and the third statement is discussed on [\emph{loc.\ cit.}, p.\ 226]. The descent to $\Q_p$ follows from \cite[Prop.\ 4.10(a-c)]{coleman1989}. 
\end{proof}

Following Proposition \ref{prop: FV on dR}, this means that $U$ determines the crystalline Frobenius action under the comparison of Proposition \ref{prop: cris dR}.  We are particularly interested in the $U$-critical (equivalently, Frobenius-ordinary) part, which Coleman studied in \cite[\S7.2]{coleman1996}. We call a generalized Hecke eigenspace \emph{non-trivial} when it properly contains the eigenspace. 

\begin{prop}[{Coleman \cite[\S7.2]{coleman1996}}]
\label{prop: coleman non-split}
Let $k \in \Z_{\geq 3}$ and let $E/\Q_p$ be a finite extension. There is a short exact sequence of $\bT[U]$-modules of finite $\Q_p$-dimension,
\[
0 \to M_{2-k,\Q_p}^{\dagger, \ord}(k-1) \buildrel{\theta^{k-1}}\over\to M_{k,\Q_p}^{\dagger, \crit} \to \bH_\dR^1(X_{\Q_p}, \cF_k)^\crit \to 0,
\]
and $\bH_\dR^1(X_{\Q_p}, \cF_k)^\crit$ is (non-canonically) isomorphic as a $\bT[U]$-module to the classical critical cusp forms $S_{k,\Q_p}^\crit$. 
\begin{enumerate}
    \item If an eigenform $f \in M_{k,E}^{\dagger,\crit}$ is not in the image of $\theta^{k-1}$, then $f$ is classical. 
    \item A classical cuspidal eigenform $f \in S_{k,E}^{\crit}$ has non-trivial generalized eigenspace in $f \in M_{k,E}^{\dagger,\crit}$ if and only if $f$ lies in the image of $\theta^{k-1}$. 
\end{enumerate}
\end{prop}

\begin{proof}
    The short exact sequence and the characterization of $\bH_\dR^1(X_{\C_p}, \cF_k)^\crit$ are found in \cite[\S7.2]{coleman1996}. Statements (1) and (2) follow from those claims, and are also found \textit{ibid.}\ over $\C_p$. The descent to $\Q_p$ (or $E$) follows from Proposition \ref{prop: Coleman pres dR}. 
\end{proof}

 We now spell out some immediate corollaries of Coleman's results: stating these is just a matter of emphasis. First, all three variants of de Rham cohomology in Proposition \ref{prop: Coleman pres dR} coincide on their critical parts.

\begin{cor}
\label{cor: Eisenstein theta}
Let $k \in \Z_{\geq 3}$. The inclusion of short exact sequences 
\begin{gather*}
\xymatrix{
0 \ar[r] & M_{2-k,\Q_p}^{\dagger, \ord}(k-1) \ar[r]^(.6){\theta^{k-1}} & M_{k,\Q_p}^{\dagger, \crit} \ar[r] & \bH^1(X_{\Q_p}, \cF_k(\log \widetilde{SS}))^\crit \ar[r] & 0 \\
0 \ar[r] & S_{2-k,\Q_p}^{\dagger, \ord}(k-1) \ar[r]^(.6){\theta^{k-1}} \ar[u] & S_{k,\Q_p}^{\dagger, \crit} \ar[r] \ar[u] & \bH^1_\mathrm{par}(X_{\Q_p}, \cF_k)^\crit \ar[r] \ar@{=}[u] & 0 
}
\end{gather*}
has cokernel spanned by (classical $U_p$-critical) Eisenstein series. In particular, both instances of $\theta^{k-1}$ are isomorphisms on the part of the $U$-ordinary space with an Eisenstein generalized eigensystem. 
\end{cor}

\begin{proof}
The equality on the right follows from Proposition \ref{prop: coleman non-split}. The classical Eisenstein series are a complementary subspace to the subspace $S_{k,\Q_p}^{\dagger, \crit} \subset M_{k,\Q_p}^{\dagger, \crit}$ by \cite[Lem.\ 6.7]{coleman1996}, which says that the dimension of classical critical Eisenstein series equals the number of cusps. 
\end{proof}

When $f$ is a $\bT[-]$-eigenform with Hecke field $E/\Q_p$ and $X$ is a $\bT[-]_E$-module, we let $X_{(f)}$ denote the generalized eigenspace for the eigensystem of $f$. Let $X_f$ denote the eigenspace. 

\begin{eg}[{\cite{hsu2020}}]
    \label{eg: CM}
    Let $k \in \Z_{\geq 2}$ and let $f'$ be a normalized classical cuspidal eigenform of level $\Gamma_1(N)$ with complex multiplication (``CM'') by an imaginary quadratic field $K$ in which $p$ splits. Then $f'$ is $T_p$-ordinary and we can let $f =\sum_{n \geq 1} a_n q^n \in S_{k,\C_p}^\crit$ be its $U_p$-critical $p$-stabilization. A characterization of CM-ness is that $a_n \chi_D(n) = a_n$, where $\chi_D$ is the quadratic Dirichlet character associated to $D$. In particular, $T_\ell \cdot f = 0$ for all prime numbers $\ell$ that are inert in $K/\Q$. 

    C.-Y.\ Hsu calculated the $q$-series of a strictly generalized overconvergent eigenform $g = \sum_{n \geq 1} b_n q^n$ in the generalized eigenspace of $f$ \cite{hsu2020}. They satisfy the identity $-b_n \chi_D(n) = b_n$. In particular, $b_1 = b_p = 0$ and $\{f,g\}$ is a $\C_p$-basis for $(S^{\dagger,\crit}_{k,\C_p})_{(f)}$ in many cases. 

    Assembling these facts, we see that $T_\ell$ (including $T_p$ standing for $U_p$) acts
    \begin{itemize}
        \item non-trivially nilpotently if $\ell$ is inert in $K$ 
        \item semi-simply if $\ell$ is split in $K$
    \end{itemize}
    Indeed, this is visible because, within $(S^{\dagger,\crit}_{k,\C_p})_{(f)}$, we can distinguish the span of $g$ by the vanishing of the 1st coefficient.
\end{eg}

Coleman noticed that, due to Hida theory (Theorem \ref{thm: BP}), spaces of critical overconvergent modular forms have locally constant rank. 
\begin{cor}[{Coleman \cite[Cor.\ 7.2.3]{coleman1996}}]
    \label{cor: local constant rank Coleman}
    Let $k,k' \in \Z_{\geq 3}$. If $k \equiv 2-k' \pmod{p-1}$, then 
    \[
    \dim_{\Q_p} M_k^{\dagger,\crit} = \dim_{\Q_p} M_{k',\Q_p}^{\dagger,\ord} + \dim_{\Q_p} S_{k,\Q_p}^{\dagger,\ord}. 
    \]
\end{cor}

\section{Critical overconvergent forms: $p$-integral aspects}
\label{sec: p-integral cric oc}

The goal of this section is to prove that the $\Z_p$-lattice $M_k^{\dagger,\crit} \subset M_{k,\Q_p}^{\dagger,\crit}$ that we define in Definition \ref{defn: Zp crit oc forms} is well-behaved with respect to cohomology. After this, we set up the interpolation properties and relations with cohomology of two important sublattices of $M_k^{\dagger,\crit}$, called the anti-ordinary forms and the twist-ordinary forms. 

Here is the key statement about $M_k^{\dagger,\crit}$. 

\begin{thm}
\label{thm: lattice main}
Let $k \in \Z_{\geq 3}$. There is a short exact sequence of $\bT[U]$-modules which are finitely generated and flat as $\Z_p$-modules,
\begin{equation}
    \label{eq: main SES k}
    0 \to M_{2-k}^{\dagger, \ord}(k-1) \buildrel{\theta^{k-1}}\over\lra M_k^{\dagger, \crit} \buildrel{\pi_k}\over\lra e(F)H_c^1(X^\ord, \omega^{2-k}) \to 0 
\end{equation}
with $\theta^{k-1}$ as in \eqref{eq: theta} and $\pi_k$ the composition of
\begin{itemize}
    \item Coleman's presentation of Proposition \ref{prop: Coleman pres dR}, $M_{k,\Q_p}^{\dagger} \rsurj H^1_\dR(X_{\Q_p}^\mathrm{an},\cF_k)$, under which the image of $M_k^{\dagger,\crit}$ lies in $H^1_\mathrm{par}(X_{\Q_p}^\mathrm{an}, \cF_k)^\crit$
    \item the map $H^1_\mathrm{par}(X^{\rm an}_{\Q_p},\cF_k)^\crit \to \Hom_{\Z_p}(S_k^\ord,\Q_p)$ given by Coleman's analytic Poincar\'e duality pairing $\lr{}_\dR$ of \eqref{eq: dR pairing} along with the restriction of the presentation of \ref{prop: Coleman pres dR} to  $S_k^\ord \rinj H^1_\dR(X^\mathrm{an}_{\Q_p}, \cF_k)^\ord$, and 
    \item Boxer--Pilloni's perfect pairing \eqref{eq: SD cusp ord} between $S_k^\ord$ and $e(F)H^1_c(X^\ord, \omega^{2-k})$. 
\end{itemize}  
\end{thm}

Along with establishing compatibility between canonically-defined $\Z_p$-lattices, this theorem upgrades the non-canonical isomorphism of Proposition \ref{prop: coleman non-split} by identifying the image of $\pi_k$. 

\begin{rem}
    Note that the properties of $M_k^{\dagger,\crit}$ stated in Theorem \ref{thm: lattice main} do not uniquely characterize the $\Z_p$-lattice. Namely, it pins down a saturated sublattice and a corresponding quotient lattice, while the extension class is still flexible within $M_{k,\Q_p}^{\dagger,\crit}$. Of course, we will actually specify $M_k^{\dagger,\crit}$ in Definition \ref{defn: Zp crit oc forms}, simply choosing the lattice that interpolates best with respect to duality with its $\bT[U']$-action. We thank George Boxer for asking about the extent to which $q$-series are required to characterize $M_k^{\dagger,\crit}$. We use $q$-series in Definition \ref{defn: Zp crit oc forms} to express the duality, and then we prove Theorem \ref{thm: lattice main}; but one could alternatively define $M_k^{\dagger,\crit}$ using the defining characteristics of Theorem \ref{thm: lattice main} along with interpolability of the extension class with respect to $U'$, and then use our theorem to show that the expected dualities with Hecke algebras hold. The main justification of this choice of extension class is that we can construct the map $\delta$ of Theorem \ref{thm: flat derived action}, which is sensitive to precisely this choice. 
\end{rem}

\subsection{A lattice of overconvergent forms}

First we develop a \emph{$p$-adic dual operator} $U'$ that makes critical $U$-actions look ordinary with respect to a new operator $U'$. This is nothing other than a Frobenius operator, but we refer to it with this notation and ``dual'' terminology to emphasize how we apply it. 

\begin{defn}
\label{defn: avatar}
The \emph{$p$-adic dual $U$-operator} in weight $k \in \Z_{\geq 1}$ is 
\[
 U' := \lr{p}_N p^{k-1}U^{-1}.
\]
\end{defn}

    The operator $U'$ relates to $U$ just as the Frobenius operator $F$ relates to $U$ on cohomology, according to Propositions \ref{prop: FV on dR} and \ref{prop: interpolated Serre duality} and the compatibility of $\lr{}_\dR$ and $\lr{}_\mathrm{SD}$ of Proposition \ref{prop: Hodge SES}. We summarize the situation. 
    \begin{prop}
    \label{prop: SD p-dual}
    The following adjunction formula holds for the Serre duality pairing $\lr{}_\mathrm{SD}$ between $e(U)H^0(X^\ord, \omega^k(-C))$ and $e(F)H^1_c(X^\ord,\omega^{2-k})$, as well as the Poincar\'e duality self-pairing $\lr{}_\dR$ on $e(T_p)\bH^1_\mathrm{par}(X_{\Q_p}, \cF_k)$:
    \[
    \lr{\lr{p}_N^{-1} U \cdot f, g}_\star = \lr{f, U' \cdot g}_\star. 
    \]    
    \end{prop}

\begin{rem}
\label{rem: U' is not U*}
The passage between $U$- and $U'$-actions swaps critical slope with ordinary slope. In particular, it swaps a critical $U$-stabilization of a classical form of level $\Gamma_1(N)$ with its ordinary $U'$-stabilization \emph{in terms of the eigenvalue}. This classical case is our primary case of interest, but we point out that the $U$ to $U'$ swap on is not generally achieved by applying an Atkin--Lehner operator at $p$, nor any other Atkin--Lehner operator. As Atkin--Lehner operators exchange Hecke operators with their duals, this is the well-known phenomenon of some standard vs.\ dual Hecke operators not commuting on oldforms. 
\end{rem}

\begin{rem}
    \label{rem: why p-adic dual}
    Proposition \ref{prop: SD p-dual} motivates calling $U'$ a ``$p$-adic dual'' of $U$, extending from $e(F)H^1_c(X^\ord, \omega^{2-k}) \otimes_{\Z_p} \Q_p$  to $e(T_p)\bH^1_\mathrm{par}(X_{\Q_p}, \cF_k)$ under the putative Hodge quotient map $\pi_k$, and then even further to $M_k^{\dagger,\crit}$ under Coleman's presentation (Proposition \ref{prop: Coleman pres dR}). Likewise, on the 2-dimensional space of $U_p$-stabilizations of a $T_p$-ordinary classical eigenform $f$, use the stabilizations as a basis, their $U_p$-eigenvalues being the roots $\alpha, \beta$ of $X^2 - a_p(f) X - \lr{p}_N p^{k-1}$. Then $U$ has matrix $\sm{\alpha}{}{}{\beta}$, while $U'$ has matrix $\sm{\beta}{}{}{\alpha}$. 

    On the other hand, $U'$ relates to $U$ as the usual dual Hecke operator $U^*_p$ relates to $U_p$ when acting on eigenforms that are primitive of some level divisible by $p$. 
\end{rem}

\begin{defn}
    \label{defn: Zp crit oc forms}
    Let $M_{k}^{\dagger,\crit} \subset M_{k,\Q_p}^{\dagger,\crit}$ denote the $\bT[U']$-submodule on which the abstract Hecke algebra $\bT[U']$ pairs into $\Z_p$, called \emph{critical overconvergent modular forms over $\Z_p$}. That is, 
    \[
    M_{k}^{\dagger,\crit} := \{f \in M_{k,\C_p}^{\dagger,\crit} : a_1(t \cdot f) \in \Z_p \ \forall t \in \bT[U']
    \}.
    \]   
    Similarly define $S_k^{\dagger,\crit}$, the \emph{critical overconvergent cusp forms over $\Z_p$}. 
    
    Denote by $\bT_k^{\dagger,\crit} := \bT[U'](M_k^{\dagger,\crit})$ and $\bT_k^{\dagger,\crit,\circ}:= \bT[U'](S_k^{\dagger,\crit})$ the \emph{critical overconvergent (cuspidal) Hecke algebra of weight $k$}. 
\end{defn}

By design, $M_k^{\dagger,\crit}$ is the $\Z_p$-lattice featuring the following perfect $\Z_p$-bilinear duality pairings. 

\begin{prop} 
\label{prop: pairing wt k}
There are perfect $\bT[U']$-equivariant pairings
    \begin{gather*}
    \bT_k^{\dagger,\crit} \times M^{\dagger,\crit}_k \to \Z_p, \\    
    \bT_k^{\dagger,\crit,\circ} \times S^{\dagger,\crit}_k \to \Z_p 
    \end{gather*}
    defined by $(t,f) \mapsto a_1(t \cdot f)$, in particular, realizing the pairing of Proposition \ref{prop: oc duality} under $\otimes_{\Z_p} \C_p$ and the replacement of $U' \in \bT_k^{\dagger,\crit}$ by $\lr{p}_N p^{k-1}U^{-1} \in \bT[U](M_{k,\C_p}^{\dagger,\crit})$. 
\end{prop}

For the sake of thoroughness, we give the argument, which works over PIDs. 

\begin{proof}
The pairing exists by definition. Both $\bT_k^{\dagger,\crit}$ and $M^{\dagger,\crit}_k$ are finitely generated torsion-free $\Z_p$-modules, and therefore are free. By Proposition \ref{prop: oc duality}, the maps
\begin{gather*}
    \bT_k^{\dagger,\crit} \to \Hom_{\Z_p}(M^{\dagger,\crit}_k, \Z_p), \\
    M^{\dagger,\crit}_k \to \Hom_{\Z_p}(\bT_k^{\dagger,\crit}, \Z_p)
    \end{gather*}
    are injective with torsion cokernels. Therefore the sources have equal $\Z_p$-rank. The second arrow is an isomorphism by a standard argument, as in the structure theorem for modules over a PID: there exists a basis of the target so that the image is spanned by a basis consisting of scalar multiples of the original basis; therefore, by definition of the source, we can divide out by these scalars on the source, showing that the second arrow is surjective. The first arrow is an isomorphism by a similar argument. 
\end{proof}

Because the slopes of $\lr{d}_N \in \bT[U']$ are always zero and the remaining generators of $\bT[U']$ are $T_n \in \bT[U']$ for $n \in \Z_{\geq 1}$ \emph{as defined in} Definition \ref{defn: Tn for general n} (for example, $T_p = U'$), we can also consider the definition to be 
    \[
    M_k^{\dagger,\crit} := \{f \in M_{k,\C_p}^{\dagger,\crit} : a_n(f) \in \Z_p \ \forall n \in \Z_{\geq 1}\},
    \]
and similarly for $S_k^{\dagger,\crit}$. These $a_n(f)$ are coefficients of the \emph{alternate $q$-series} $\sum_{n \geq 1} a_n(f)$ as in Definition \ref{defn: alternate q-series}, producing $M_k^{\dagger,\crit} \rinj q \Z_p\lb q\rb$ because the constant term vanishes.

\begin{rem}
    \label{rem: possible difference in lattice}
    We have mentioned that the idea here is to choose the lattice so that its Hecke action interpolates well, so what is the difference between the lattices defined in terms of $\bT[U]$ vs.\ $\bT[U']$? If the $U$-action is semi-simple on every generalized eigenspace (as in Example \ref{eg: CM}, a CM case) there is none; but in a hypothetical case where $U$ acts non-semi-simply on $M_{k,\C_p}^{\dagger,\crit}$, there may be a difference. 
\end{rem}

It will be important to note that the standard $q$-series of $M_k^{\dagger,\crit}$ are still $p$-integral. We will prove this below (Lemma \ref{lem: F-integral implies integral}).

\subsection{Compatibility of sublattices}

Now that we have defined $M_k^{\dagger,\crit}$, our goal of proving Theorem \ref{thm: lattice main} requires us to study the compatibility of the lattices declared there. The sublattice is straightforward. 
\begin{prop}
    \label{prop: ao sublattice}
    Let $k \in \Z_{\geq 3}$. The restrictions of $\theta^{k-1}: M^{\dagger,\ord}_{2-k, \Q_p}(k-1) \to M^{\dagger,\crit}_{k,\Q_p}$ of \eqref{eq: theta} to the submodules $M^{\dagger,\ord}_{2-k}(k-1)$ and $S^{\dagger,\ord}_{2-k}(k-1)$ of $M^{\dagger,\ord}_{2-k, \Q_p}(k-1)$ produce $\Z_p$-saturated $\bT[U']$-equivariant injections 
    \[
    \theta^{k-1} : M^{\dagger,\ord}_{2-k}(k-1) \rinj  M^{\dagger,\crit}_k, \qquad S^{\dagger,\ord}_{2-k}(k-1) \rinj  S^{\dagger,\crit}_k.
    \]
\end{prop}

\begin{proof}
    First we note that each of the $\Z_p$-lattices $M^{\dagger,\ord}_{2-k}(k-1)$ and $M^{\dagger,\crit}_k$, as a submodule of its $\Q_p$-extension of scalars, is characterized by perfect duality over $\Z_p$ with the action of $\bT[U']$. In the latter case, this is Definition \ref{defn: Zp crit oc forms}. For the former case, we argue from the perfect duality between $\bT_{2-k}^{\dagger,\ord}$ and $M_{2-k}^{\dagger,\ord}$: we calculate from \eqref{eq: theta twist} that $U'$ acts on $M^{\dagger,\ord}_{2-k, \Q_p}(k-1)$ as $U^{-1}\lr{p}_N$ acts on $M^{\dagger,\ord}_{2-k, \Q_p}$, preserving slope $0$; and the Hecke operators away from $p$ have constant slopes under twisting by $(k-1)$. Thus $\bT[U'](M^{\dagger,\ord}_{2-k}(k-1))$ is in perfect duality with $M^{\dagger,\ord}_{2-k}(k-1)$. For the same reasons, starting from the perfect $\Z_p$-linear duality between the surjection $\bT^{\dagger,\ord}_{2-k} \rsurj \bT^{\dagger,\ord,\circ}_{2-k}$ and the injection $S^{\dagger,\ord}_{2-k} \rinj M^{\dagger,\ord}_{2-k}$, one deduces a perfect duality between the surjection and injection pair
    \[
    \bT[U'](M^{\dagger,\ord}_{2-k}(k-1)) \rsurj \bT[U'](S^{\dagger,\ord}_{2-k}(k-1)) \quad \text{and} \quad 
    S^{\dagger,\ord}_{2-k}(k-1) \rinj M^{\dagger,\ord}_{2-k}(k-1). 
    \]
    
    Having established these perfect dualities, the key observation is that the $\bT[U']$-compatibility of the dualities implies that the inclusion $\theta^{k-1}$ on $M^{\dagger,\ord}_{2-k}(k-1)$ (resp.\ $S^{\dagger,\ord}_{2-k}(k-1)$) is dual to the natural finite flat $\Z_p$-algebra surjections $\bT_k^{\dagger,\crit} \rsurj \bT[U'](M^{\dagger,\ord}_{2-k}(k-1))$ (resp.\ $\bT_k^{\dagger,\crit} \rsurj \bT[U'](S^{\dagger,\ord}_{2-k}(k-1))$). Now the claim follows from the fact that $\Z_p$ is a Dedekind domain and therefore torsion-free finitely generated $\Z_p$-modules are projective: each of the images 
    \[
    \theta^{k-1}(M^{\dagger,\ord}_{2-k}(k-1)), \theta^{k-1}(S^{\dagger,\ord}_{2-k}(k-1)) \subset M_k^{\dagger,\crit}
    \]
    has a $\Z_p$-projective complement because its dual surjection above have $\Z_p$-linear section. 
\end{proof}

We give these images a name, and point out that we have perfect duality with its Hecke algebra thanks to the proof above. 
\begin{defn}
    \label{defn: ao wt k}
    Let $k \in \Z_{\geq 2}$. Let $M_k^\aord := \theta^{k-1}(M^{\dagger,\ord}_{2-k}(k-1))$ thought of as a $\bT[U']$-module, where ``ao'' stands for ``anti-ordinary.'' Let $\bT_k^\aord := \bT[U'](M_k^\aord)$. Likewise, let there be cuspidal versions $S_k^\aord := \theta^{k-1}(S^{\dagger,\ord}_{2-k}(k-1))$ and $\bT_k^{\aord,\circ} := \bT[U'](S_k^\aord)$. 
\end{defn}

\begin{cor}
    \label{cor: aord duality}
    Let $k \in \Z_{\geq 3}$. We have perfect ``$a_1$'' duality $\lr{} : \bT_k^\aord \times M_k^\aord \to \Z_p$ that is $\bT[U']$-compatible. It is also compatible with critical $a_1$ duality, in the sense that the $\bT[U']$-algebra surjection $\bT_k^{\dagger,\crit} \rsurj \bT_k^\aord$ is the $a_1$-dual of $M_k^\aord \rinj M_k^{\dagger,\crit}$. 
\end{cor}

Likewise, classical critical forms $M_k^\crit := e(U')H^0(X_0(p),\omega^k)$ constitute another saturated submodule. 
\begin{prop}
	\label{prop: classical is saturated}
    Let $k \in \Z_{\geq 2}$. The $\bT[U']$-submodules 
    \[
    M_k^\crit \subset M_k^{\dagger,\crit}, \quad S_k^\crit \subset S_k^{\dagger,\crit}
    \]
    are $\Z_p$-saturated. Under perfect $a_1$-dualities with Hecke algebras, the surjections $\bT_k^{\dagger,\crit} \rsurj \bT[U'](M_k^\crit)$, $\bT_k^{\dagger,\crit, \circ} \rsurj \bT[U'](S_k^\crit)$ are $a_1$-dual to the inclusions $\zeta_k : M_k^\crit \rinj M_k^{\dagger,\crit}$, $S_k^\crit \rinj S_k^{\dagger,\crit}$, respectively. 
\end{prop}

\begin{proof}
    The argument runs in exactly the same way as the proof of Proposition \ref{prop: ao sublattice}, beginning with the observation that $U'$ acts with slope $0$ on $M_k^\crit$. 
\end{proof}

We conclude by showing that the standard (in addition to the the alternate) $q$-series of $M_k^{\dagger,\crit}$ are valued in $q \Z_p\lb q \rb$, and discuss that this choice is a certain choice of extension class of lattices. 
\begin{lem}
    \label{lem: F-integral implies integral}
    Let $k \in \Z_{\geq 2}$. The image of the standard $q$-series map $M_k^{\dagger,\crit} \rinj \Q_p\lb q\rb$ is contained in $q\Z_p\lb q \rb$. 
\end{lem}

\begin{proof}
    It will suffice to prove the statement after replacing $\Q_p$ by a finite extension $K$ so that all Hecke eigenvalues are contained in $K$ and replacing $\Z_p$ by its integral closure $\cO_K$ in $K$. This gives us access to generalized eigenbases. 

    Proposition \ref{prop: ao sublattice} already showed that $\theta^{k-1}(M^{\dagger,\ord}_{2-k}) \subset M_k^{\dagger,\crit}$ is saturated. The standard $q$-series of $\theta^{k-1}(M^{\dagger,\ord}_{2-k})$ are $\Z_p$-valued because $\theta^{k-1}$ maintains this property and kills the constant coefficient. 

    Due to the $\bT[U']$-isomorphism 
    \[
    \frac{M_k^{\dagger,\crit}}{\theta^{k-1}(M_{2-k}^{\dagger,\ord})} \otimes_{\Z_p} \Q_p \simeq S_{k,\Q_p}^\crit
    \]
    of Proposition \ref{prop: coleman non-split}, $\bT[U']$ acts semi-simply on these modules since this is known for Hecke actions on classical forms. Choose an $\cO_K$-basis of $M_{k,\cO_K}^{\dagger,\crit}$ which is a $\bT[U']$-generalized eigenbasis. Therefore this basis $X$ has a subset $Y$ that projects to an eigenbasis of $M_{k,\cO_K}^{\dagger,\crit}/\theta^{k-1}(M_{2-k,\cO_K}^{\dagger,\ord})$.  

    Let $f = f_a$ be an element of $Y$, where $\{f_1, f_2, \dotsc, f_a\} \subset X$ is the subset which is the generalized eigenbasis with $\bT[U']$-eigensystem equal to that of $f$. Consider the form of the matrix of $U'$ with respect to this basis, considered as a block matrix with respect to the partition $\{f_1, \dotsc, f_{a-1}\} \cup \{f_a\}$.
    \[
    \begin{pmatrix}
        A_{(a-1) \times (a-1)} & B_{(a-1) \times 1} \\
        0_{1 \times (a-1)} & d
    \end{pmatrix} \in \GL_a(\cO_K)
    \]
    We can also assume that $A$ is upper-triangular. Then each of the diagonal entries of $A$ are equal to $d$, which is the $U'$-eigenvalue of $f$. The matrix is in $\GL_d(\cO_K)$ because the $U'$-eigenvalue is integral. 

    Consequently, the matrix of $U$ has the same block form, but is scaled by $\chi_f(p)p^{k-1}$. Since all of these matrix entries are valued in $\cO_K$, the $p$-integrality of $a_1(U' \cdot f_i)$ implies the $p$-integrality of $a_1(U \cdot f_i)$ for $1 \leq i \leq a$.
\end{proof}

\subsection{Compatibility of quotient lattices}
\label{subsec: quotient lattices}

Next we prove the desired property of the quotient map $\pi_k$ defined in the statement of Theorem \ref{thm: lattice main}. Along with the content above, particularly Proposition \ref{prop: classical is saturated}, this completes the proof of Theorem \ref{thm: lattice main}. 


\begin{prop}
\label{prop: integral quotient map}
Let $k \in \Z_{\geq 3}$. The $\pi_k$ map defined in Theorem \ref{thm: lattice main} as a $\Q_p$-linear composition 
\[
M_{k,\Q_p}^{\dagger,\crit} \to H^1_\dR(X_{\Q_p}, \cF_k)^\crit \to \Hom_{\Z_p}(S_k^\ord, \Q_p) \to  e(F)H^1_c(X^\ord,\omega^{2-k}) \otimes_{\Z_p} \Q_p
\]
restricts to a surjective map of $\Z_p$-lattices
\[
\pi_k : M_k^{\dagger,\crit} \rsurj e(F)H^1_c(X^\ord,\omega^{2-k}). 
\]
\end{prop}

We will prove Proposition \ref{prop: integral quotient map} by a brief argument deducing it from Proposition \ref{prop: pi-prime surjective} immediately after this proposition is stated. Our main task now is to build up some theory of $p$-integral analytic de Rham cohomology in weight $k$, culminating in Proposition \ref{prop: pi-prime surjective}. 

\subsubsection{Coleman's work on $p$-integral analytic de Rham cohomology -- the case of weight $2$}

The key to proving the proposition's claim about $\pi_k$ is to set up an analytic formulation of integral $p$-adic de Rham cohomology following the approach of Coleman in \cite[\S3, \S A]{coleman1994-barsotti}. The main idea is that the arguments of \emph{loc.\ cit.}, which addresses the case of weight 2 by studying the trivial coefficient system $\cF_2$, can be straightforwardly generalized to weight $k \geq 3$. So our goal is to specify the general weight analogue of Coleman's arguments to such an extent that it is clear that the generalization holds. 
\begin{rem}
    The applications to modular forms that Coleman aims at in \cite[\S4-5]{coleman1994-barsotti} involve more intricate rigid analysis than we do because he wants to work on the modular curve with level $\Gamma_1(p)$ at $p$, which is ramified at $p$. In contrast, we work with a modular curve $X$ that is smooth over $\Z_p$ along with a Gauss--Manin connection coming from its universal elliptic curve, so we have access to results like Proposition \ref{prop: cris dR}. 
\end{rem}

We will use the following notation, mostly matching \cite[\S3, \S A]{coleman1994-barsotti}: the main exception is that we write $X^\ord \subset X^\mathrm{an}$ for the ordinary locus, while Coleman uses ``$Y$.'' We also include some of the basic results of \textit{loc.\ cit.}\ about orientations of annuli in this list. We warn the reader that, unlike in \emph{loc.\,cit.}, we use the standard normalized $p$-adic valuation $v$ on $\C_p$, along with the standard $p$-adic norm $|z| = p^{-v(z)}$. 

\begin{defn}[{\cite[\S3, \S A]{coleman1994-barsotti}}]
    \label{defn: coleman-barsotti}
    \begin{itemize}
        \item Let $\cA(r,s)$ denote the standard open annulus $\{y \in \C_p : r < |y| < s\}$. 
        
        \item Let $E = W_1 \smallsetminus X^\ord$, a disjoint union of \emph{supersingular annuli} (for the notation $W_1 = X^\mathrm{an}_{p^{-p/(p+1)}}$, see Definition \ref{defn: oc forms over Cp}). We index these by $x \in SS$, so $E = \coprod_{x \in SS} E_x$. Let $\F_x$, $W(\F_x) \subset K_x$ denote the field of definition of $x$ and the ring of integers within the unramified extension of $\Q_p$ with residue field $\F_x$. So $K_x = \Q_p$ or $\Q_{p^2}$. 
        \item Each $E_x$ (resp.\ $E'_x$) is an open oriented annulus of width $p/(p+1)$ (resp.\ width $1$), where width is defined as in \cite[p.\ 126]{coleman1994-barsotti} and the choice of orientation is specified as in \cite[Cor.\ 3.7a]{coleman1989}. For the notion of orientation, see  \cite[Lem.\ 2.1]{coleman1989}: it is an equivalence class of uniformizing parameters, where a uniformizing parameter $z_x$ gives an isomorphism $z_x : E'_x \isoto \cA(1/p,1)$ defined over $K_x$. We also use such a $z_x$ to give a uniformizing parameter $E_x \isoto \cA(p^{-p/(p+1)},1)$. For instance, one may let $z_x$ be a local trivialization of the Eisenstein series of weight $p-1$ when $p \geq 5$. 
        \item The orientation has the property that the \emph{end} of $E_x$ approaching $X^\ord$ is the end where $|z_x|$ approaches $1^-$. We call this the ``ordinary end'' of $E_x$ and the other end the ``supersingular end.''
        \item Let $A_{E_x}$ be the analytic functions $A(E_x)$ on $E_x$ that are ``bounded by 1,'' that is, integrally valued. Likewise, let $\Omega_{E_x}$ denote the submodule of the module of analytic differentials $\Omega^1(E_x)$ that are bounded by 1. 
        \item Let $A_{\eta_x}$ denote the analytic functions $h$ on $E_x$ which satisfy
        \[
        \lim_{y \in E_x, |z_x(y)| \to 1^-} |h(y)| \leq 1.
        \]
        That is, such $h$ are $p$-integral near the ordinary ends of $E$, but there is no restriction on their norm at the supersingular ends of $E$. Define $\Omega_{\eta_x}$ with exactly the same conditions on the norm at ordinary ends. 
        \item Let $A(W_1)_{X^\ord}$ (resp.\ $\Omega(W_1)_{X^\ord}$) denote those analytic functions (resp.\ differentials) on $W_1$ that are $p$-integral after restriction to $X^\ord$. 
        \item Fix a choice $Z = \{z_x\}_{x \in SS}$ of uniformizing parameters as above. Let $\Omega_{z_x}$ denote those differentials on $E_x$ of the form
        \begin{equation}
            \label{eq: omega z form}
            f(z_x) \frac{dz_x}{z_x} +dg(z_x)    
        \end{equation}
        where $f \in W(\F_x)\lb T \rb$ and $g \in W(\F_x)\lb T, T^{-1}\rb$. 
    \end{itemize}
\end{defn}

Coleman's key result about $\Omega_{z_x}$, allowing for an analytic-algebraic comparison, is the following lemma establishing the independence of $z_x$. We state a version that makes sense for an open annulus $V$ defined over a general $p$-adic field $K$, so the definition of $\Omega_z$ for $z$ a uniformizing parameter of $V$ is \eqref{eq: omega z form} with the ring of integers $\cO_K$ replacing $W(\F_x)$. 

\begin{lem}[{Coleman \cite[Lem.\ A1.2]{coleman1994-barsotti}}]
\label{lem: Omega z elements}
Let $V$ be an annulus defined over $K$ that has width at least $1/p$. Let $z$ be a uniformizing parameter for $V$. Then 
\begin{equation}
    \label{eq: Omega z elements}
    \Omega_{z} = \{\nu + dg : \nu \in \Omega_V, g \in A_\eta\}. 
\end{equation}    
\end{lem}
Consequently, because the annuli $E_x$ in Definition \ref{defn: coleman-barsotti} satisfy the width condition in the lemma, we have $\Omega_{z_x} = \Omega_{E_x} + dA_{\eta_x}$. Moreover, we deduce that given any alternate uniformizing parameter $z'_x$ of $E_x$, $\Omega_{z_x} = \Omega_{z'_x}$ if and only if $z_x$ and $z'_x$ determine the same orientation. 

\subsubsection{Defining weight $k$ integral analytic de Rham cohomology}

Using the notation above, we want to define a weight $k \in \Z_{\geq 2}$ generalization of the integral de Rham complex Coleman defines in \cite[p.\ 135]{coleman1994-barsotti}. It is the twist of Coleman's complex by the $(k-2)$th symmetric power $\nabla_{k-2} : \cH_{k-2} \to \cH_{k-2} \otimes \Omega^1$ of the Gauss--Manin connection of \S\ref{subsec: coherent cohom}. Indeed, this is defined over $X$ (that is, over $\Z_p$), so we can make sense of $\cH_{k-2}(W_1)_{X^\ord}$ just like $A(W_1)_{X^\ord}$ above.

Now we can state the definition of the $p$-integral weight $k$ de Rham complex. 
\begin{defn}
\label{defn: cC cF_k}
Let $k \in \Z_{\geq 2}$. Let $\cC(\cF_k) = \cC(\cF_k)_{W_1, X^\ord}$ be the complex
\begin{gather*}
\cH_{k-2}(W_1)_{X^\ord} \to (\cH_{k-2} \otimes \Omega)(W_1)_{X^\ord} \oplus \frac{(\cH_{k-2})_\eta}{(\cH_{k-2})_E} \to \frac{(\cH_{k-2} \otimes \Omega)_\eta}{(\cH_{k-2} \otimes \Omega)_E}, 
\end{gather*}
where the differentials are given by $h\mapsto (\nabla_{k-2}(h),h\vert_E)$ and $(\beta,f)\mapsto\beta\vert_E-\nabla_{k-2}(f)$, respectively. We call it the \emph{$p$-integral analytic de Rham complex of weight $k$}. It begins in degree 0 and terminates in degree 2. 
\end{defn}

The well-definedness of the differentials is clear other than perhaps the following point: when $\cG$ is a vector bundle defined over $X$, $p$-integrality (boundedness by 1) on $X^\ord$ of $\sigma \in \cG(W_1)$ implies that $\sigma\vert_E$ becomes $p$-integral in the limit approaching the ordinary ends of $E_x$. Thus restriction maps $\cG(W_1)_{X^\ord} \to \cG_\eta$ exist. 
\begin{rem}
    The key virtue of Coleman's definition of $\cC(\cF_2)$, which we have generalized to $\cC(\cF_k)$, is that it allows for some flexibility in the notion of integrality of a closed 1-form. We are working on a curve, so all 1-forms are closed (in a standard de Rham complex). But integrality of $q$-series of an overconvergent form does not obviously make its limit on the supersingular ends of $E$ integral. What helps us is that the notion of ``integral closed 1-form'' in $\cC(\cF_k)$ is more flexible: when $\beta$ is a $\cH_{k-2}$-valued 1-form that is integral on $X^\ord$, there exists $f \in (\cH_{k-2})_\eta$ making $(\beta,f)$ a 1-form for $\cC(\cF_k)$ as long as the failure of $\beta$ to be integral on $E$ is \emph{exact on $E$} along each of the supersingular ends of $E$. And exactness on $E$ is measurable very concretely using residue maps and the trivialization of $\nabla_{k-2}$ on $E$ discussed in \eqref{eq: trivializations}. 
\end{rem}

By construction, there is a natural map $H^1(\cC(\cF_k)) \to H^1_\dR(X^\mathrm{an}_{\Q_p}, \cF_k)$ coming from sending a closed 1-form $(\beta,f)$ for $\cC(\cF_k)$ to $\beta$. The image of this map, and its various restrictions, are our weight $k$ integral de Rham cohomology. This generalizes Coleman's construction in the case $k=2$ \cite[p.\ 136]{coleman1994-barsotti} to general $k \in \Z_{\geq 2}$.  

\begin{defn} 
    \label{defn: integral analytic dR}
    Let $k \in \Z_{\geq 2}$.  We denote by $H^1_{\rm dR}(X_{\Z_p}^{\rm an},\cF_k)$ the natural image of $H^1(\mathcal{C}_k)$ in $H^1_{\rm dR}(X_{\Q_p}^{\rm an},\cF_k)$, and put 
    \begin{align*}
    H^1_\mathrm{dR}(X^\mathrm{an}_{\Z_p}, \cF_k)^0&=H^1_{\rm dR}(X_{\Z_p}^{\rm an},\cF_k)\cap H^1_{\rm dR}(X_{\Q_p}^{\rm an},\cF_k)^0,\\
    H^1_\mathrm{par}(X^\mathrm{an}_{\Z_p}, \cF_k)&=H^1_{\rm dR}(X_{\Z_p}^{\rm an},\cF_k)\cap H^1_{\rm par}(X_{\Q_p}^{\rm an},\cF_k). 
    \end{align*}
\end{defn}

Now we characterize this lattice, like the case $k=2$ \cite[Thm.\ 3.1, 3.2]{coleman1994-barsotti}. 

\begin{thm}
    \label{thm: p-integral dR cris}
    Let $k \in \Z_{\geq 2}$. Then under the isomorphisms specified in Lemma~\ref{lem: dR equivalence}, the modules 
    \[
    H^1_\mathrm{dR}(X^\mathrm{an}_{\Z_p}, \cF_k) \supset H^1_\mathrm{dR}(X^\mathrm{an}_{\Z_p}, \cF_k)^0 \supset H^1_\mathrm{par}(X^\mathrm{an}_{\Z_p}, \cF_k)
    \]
    are sent isomorphically to $\Z_p$-lattices within 
    \[
    \bH^1_{\rm dR}(X_{\Q_p},\cF_k) \supset \bH^1_{\rm dR}(X_{\Q_p},\cF_k)^0 \supset \bH^1_\mathrm{par}(X_{\Q_p},\cF_k),
    \]
    respectively. Moreover, there are natural isomorphisms 
    \begin{align*}
    H^1_\mathrm{dR}(X^\mathrm{an}_{\Z_p}, \cF_k) &\cong H^1_\mathrm{log{\text -}cris}(X/\Z_p, \cF_k, \log C \cup SS)/(\mathrm{tors}),\\
    H^1_\mathrm{dR}(X^\mathrm{an}_{\Z_p}, \cF_k)^0 &\cong H^1_\mathrm{log{\text -}cris}(X/\Z_p, \cF_k, \log C)/(\mathrm{tors}),\\
    H^1_\mathrm{par}(X^\mathrm{an}_{\Z_p}, \cF_k) &\cong H^1_\mathrm{cris}(X/\Z_p, \cF_k)/(\mathrm{tors}),
    \end{align*}
    under which the $U'$-action matches the $\varphi$-action, and where the log structure in crystalline cohomology are as indicated.
\end{thm}

To prove this theorem, we will show that the Coleman's argument in the case $k=2$ from \cite[\S A]{coleman1994-barsotti} generalizes in a straightforward way.

\subsubsection{Coleman's argument characterizing integral analytic de Rham cohomology}
The additional tool that we require for our weight $k$ generalization is the $p$-integral trivialization of $\nabla_{k-2}$ over $E_x$. Namely, the module of horizontal sections of $\nabla_{k-2}$ is canonically identifiable with $\Sym_{W(\F)}^{k-2} H^1_\cris(B_x/W(\F_x))$, where $B_x/\F_x$ is the supersingular elliptic curve labeled by $x$ and $\F_x$ is the residue field of $x \in SS \subset X_{\F_p}$. 
Thus we can identify (allowing ourselves to write $\Omega$ short for $\Omega^1$) 
\begin{equation}
\label{eq: trivializations}
\begin{split}
(\cH_{k-2})_{\square_x} \cong A_{\square_x} \otimes_{W(\F_x)} \Sym_{W(\F)}^{k-2} H^1_\cris(B_x/W(\F_x))  \\
(\cH_{k-2} \otimes \Omega)_{\square_x} \cong \Omega_{\square_x} \otimes_{W(\F_x)} \Sym_{W(\F)}^{k-2} H^1_\cris(B_x/W(\F_x))  
\end{split}
\end{equation}
for all three of the options for $\square_x$, 
\[
\square_x = E_x, z_x, \eta_x
\]
(excluding the case ``$A_{z_x}$'' which is not defined). To make notation more convenient, let
\[
(\cH_{k-2})_{E},\  (\cH_{k-2} \otimes \Omega)_{E},\  (\cH_{k-2})_{\eta},\  (\cH_{k-2} \otimes \Omega)_{\eta},\  (\cH_{k-2} \otimes \Omega)_{z_x}
\]
denote the direct sum over $x \in SS$ of these modules, respectively. 

For future use, we carefully define $(\cH_{k-2} \otimes \Omega)_{z_x}$ generalizing \eqref{eq: omega z form}. As we do this, we specify notation for local ``$z_x$-series'' expansions of sections of $\cH_{k-2} \otimes \Omega$ and $\cH_{k-2}$ on $E_x$. 
\begin{defn}
    \label{defn: z_x series}
    Let $f \in \cH_{k-2}(E_x)$ and let $\beta \in (\cH_{k-2} \otimes \Omega)(E_x)$. There exist $a_n, b_n \in \Sym^{k-2} H^1_\cris(E_x/W(\F_x)) \otimes_{W(\F_x)} \C_p$ such that 
    \[
    f = \sum_{n \in \Z} a_n z_x^n, \qquad \beta = \sum_{n \in \Z} b_n z_x^n \frac{dz_x}{z_x}. 
    \]
    We say that $\beta \in (\cH_{k-2} \otimes \Omega)_{z_x}$ when $\beta$ has the form \eqref{eq: omega z form} where $f \in W(\F_x)\lb T \rb \otimes_{W(\F)} \Sym^{k-2} H^1_\cris(E_x/W(\F_x))$ and $g \in W(\F_x)\lb T,T^{-1} \rb \otimes_{W(\F)} \Sym^{k-2} H^1_\cris(E_x/W(\F_x))$. 

    Let $(\cH_{k-2} \otimes \Omega)_Z(W_1)$ denote the submodule of $(\cH_{k-2} \otimes \Omega)(W_1)_{X^\ord}$ consisting of sections $\sigma$ such that $\sigma\vert_{E_x} \in (\cH_{k-2} \otimes \Omega)_{z_x}$. 
\end{defn}
Naturally, one can formulate conditions for the membership in all of the various modules above in terms of the $z_x$-series coefficients $a_n, b_n$. 

The definition of $(\cH_{k-2} \otimes \Omega)_Z(W_1)$ plays an important role in the analytic-algebraic comparison needed to establish Theorem \ref{thm: p-integral dR cris}. It gives rise to a formulation of cohomology 
\[
H^1_\dR((W_1, Z), \cF_k) := \frac{(\cH_{k-2} \otimes \Omega)_Z(W_1)}{\nabla_{k-2} \left( \cH_{k-2}(W_1)_{X^\ord}\right)}
\]
which, due to the generalization of Lemma \ref{lem: Omega z elements} to weight $k$ via the trivialization \eqref{eq: trivializations}, receives a natural map
\[
H^1(\cC(\cF_k)) \to H^1_\dR((W_1, Z), \cF_k), \quad (\beta,f) \mapsto \beta. 
\]

The reason that $(\cH_{k-2} \otimes \Omega)_Z(W_1)$ is a useful bridge to algebraic cohomology is that it is closely comparable to the formal ($p$-adic) completion of the de Rham complex with logarithmic poles at the (or any) set of lifts $\widetilde{SS} \subset X(\Z_{p^2})$ of $SS \subset X(\F_{p^2})$, as we will now discuss. Let $\mathrm{red} : X^\mathrm{an} \to X(\overline{\F}_p)$ be the reduction map and let $D_x := \mathrm{red}^{-1}(x)$, an open disk containing the annulus $E_x$. Then there exists $z_x \in A(D_x)$ that vanishes exactly and simply at the lift $\tilde x$ of $x$ and restricts to a uniformizing parameter on $E_x$ (consistent with the orientation on $E_x$ we have specified). Fix such a choice $Z = \{z_x\}_{x \in SS}$. Then an algebraic (or formal algebraic) section of $\cH_{k-2} \otimes \Omega(\log \widetilde{SS})$, upon restriction to $E$, becomes an element of $(\cH_{k-2} \otimes \Omega)_Z$. 

Thus the only remaining step Coleman requires to prove the weight $k=2$ version of Theorem \ref{thm: p-integral dR cris} is to formulate a sheaf-theoretic intermediary, called ``$[\cA \to \cW_Z$]'' in \cite[Lem.\ A2.2]{coleman1994-barsotti}, between the analytic cohomology $H^1_\dR((W_1, Z), \cF_2)$ and the formal algebraic de Rham cohomology with logarithmic poles at $\widetilde{SS}$, $\bH^1_\dR(\hat X, \cF_2(\log \widetilde{SS}))$. Here $\hat X/\Spf \Z_p$ denotes the $p$-adic completion of $X$; by formal GAGA, coherent cohomology over $X/\Z_p$ is isomorphic to its $p$-completed version on $\hat X$. Here we will describe the weight $k$ version of this intermediary in this proof sketch. The proof of \cite[\S A]{coleman1994-barsotti} goes through essentially verbatim: most of the arguments have to do with $z_x$-series on $E_x$ of sections of $\cH_{k-2}$, $\cH_{k-2} \otimes \Omega$, and these generalize straightforwardly from Coleman's case $k=2$ to our general case once the trivialization \eqref{eq: trivializations} is applied to get the $z_x$-series of Definition \ref{defn: z_x series}. 

\begin{proof}[{Proof sketch (Theorem \ref{thm: p-integral dR cris})}]
We delineate the adaptations from the arguments of \cite[\S A2]{coleman1994-barsotti} needed to move from weight $k=2$ to general weight $k \in \Z_{\geq 2}$. 
\begin{itemize}
    \item In the discussion immediately above, we have chosen the set $Z$ of uniformizing parameters and used Lemma \ref{lem: Omega z elements} to produce a
    \[
    H^1(\cC(\cF_k)) \to H^1((W_1,Z), \cF_k)
    \]
    that is clearly surjective.
    \item Next we formulate a weight $k$ analogue, which we now call $\cD(\cF_k)$, of the complex ``$[\cA \to \cW_Z$]'' of Zariski sheaves on $X_{\F_p}$ of \cite[Lem.\ A2.2]{coleman1994-barsotti}. One simply replaces the sheaves with weight $k$ versions, leaves the integrality conditions the same, and uses the analogous $z_x$-coordinate condition defining $(\cH_{k-2} \otimes \Omega)_Z(W_1)$ as in Definition \ref{defn: z_x series}. Namely, 
    \begin{itemize}
        \item One replaces the value sheaf $A$ of $\cA$ by $\cH_{k-2}$, keeping exactly the same integrality condition: integrality on $X^\ord$.
        \item One replaces the value sheaf $\Omega$ of $\cW_Z$ by $\cH_{k-2} \otimes \Omega$, keeping exactly the analogous integrality and $z_x$-series conditions: integrality on $X^\ord$ and membership in $(\cH_{k-2} \otimes \Omega)_{z_x}$ for all $x \in SS$. 
        \item The differential of $\cD(\cF_k)$ is given by $\nabla_{k-2}$.
    \end{itemize}
    \item One proves the weight $k$ generalization of \cite[Lem.\ A2.2]{coleman1994-barsotti} in exactly the same way, concluding that there exists a natural isomorphism $H^1((W_1,Z),\cF_k) \isoto \bH^1(X_{\F_p}, \cD(\cF_k))$. As Coleman points out, all that we have to do is show that the coherent cohomology module $H^1(X_{\F_p}, \cH_{k-2})$ vanishes. The proof has nothing to do with the particular sheaf $\cH_{k-2}$ and simply relies on the triviality of coherent cohomology on the \emph{affine} space $X \smallsetminus \widetilde{SS}$. 
    \item Following \cite[p.\ 148]{coleman1994-barsotti}, there is a map of complexes
    \[
    \cF_k(\log \widetilde{SS}) \to \cD(\cF_k) \qquad \text{over } \hat X,
    \]
    which makes sense because the underlying topological space of $\hat X$ is $X_{\F_p}$. Applying the inverse isomorphism $\bH^1(X_{\F_p}, \cD(\cF_k)) \isoto H^1((W_1,Z),\cF_k)$ and formal GAGA, one obtains a map $h: \bH^1_\dR(X, \cF_k(\log \widetilde{SS})) \to H^1((W_1,Z),\cF_k)$ and considers the composite 
    \begin{equation}
        \label{eq: alg to an}
        \bH^1_\dR(X, \cF_k(\log \widetilde{SS})) \buildrel{h}\over\to H^1((W_1,Z),\cF_k) \to H^1_\dR(X_{\Q_p}^\mathrm{an}, \cF_k(\log \widetilde{SS})). 
    \end{equation}
    \item By applying the algebraic-analytic comparison isomorphism of Lemma \ref{lem: dR equivalence} and the de Rham-crystalline isomorphism of Proposition \ref{prop: cris dR}, all that is left to do is to show that 
    \begin{enumerate}
        \item $h$ is surjective; because we know the kernel of \eqref{eq: alg to an} is torsion, this is already good enough to get \textit{some} choice of the three isomorphisms with crystalline cohomology (modulo torsion), as claimed in Theorem \ref{thm: p-integral dR cris}.
        \item the composition $H^1(\cC(\cF_k)) \rsurj H^1_\dR((W_1,Z), \cF_k) \to H^1_\dR(X_{\Q_p}^\mathrm{an}, \cF_k(\log \widetilde{SS}))$ is independent of the choice of $Z$.
    \end{enumerate} 
    \end{itemize}
    
    The argument for (1) in the case $k=2$ appears on \cite[p.\ 148]{coleman1994-barsotti} and relies entirely on the $z_x$-series condition defining $\Omega_{z_x}$ as in \eqref{eq: omega z form}. The same argument works in weight $k$ by using our $z_x$-series of weight $k$; again, our $z_x$ series and their $p$-integrality properties are defined according to the trivialization \eqref{eq: trivializations}, which allows for the notion of $(\cH_{k-2} \otimes \Omega)_{z_x}$ appearing in Definition \ref{defn: z_x series}. 

    Likewise, the argument for (2) in the case $k=2$ appears as the proof of \cite[Lem.\ A3.1]{coleman1994-barsotti}, justifying that a $Z$-based construction $s_Z : H^1_\dR(X_{\Z_p}^\mathrm{an}, \cF_k) \to H^1(\cC(\cF_k))$ of [p.\ 136, \textit{loc.\ cit.}]\ is a section of \eqref{eq: alg to an} that is independent of $Z$ up to torsion and also induces the same isomorphism as \eqref{eq: alg to an} with crystalline cohomology (modulo torsion). (In addition, Coleman shows that $s_Z$ is also independent of $Z$ on torsion, provided that it is restricted to $H^1_\dR(X_{\Z_p}^\mathrm{an}, \cF_k)^0$.) Coleman's proof relies entirely on arguments using $z_x$-series computations; the next steps in our proof will present more details of Coleman's arguments for use in our application. Again using Definition \ref{defn: z_x series}, which generalizes to weight $k$ the weight 2 construction of \eqref{eq: omega z form}, the same argument works. 
\end{proof}

\subsubsection{More about integral and $Z$-coordinate structures}

In preparation for our application of the above theory to characterize the $\pi_k$ map, there is a need for more information about the integral structures and $Z = \{z_x\}$-coordinate structures above. We will begin by following Coleman's exposition \cite[p.\ 136]{coleman1994-barsotti} of an implication of the final part, labeled (2) immediately above, of the proof. 

The upshot of this argument is that $Z$ determines a section $s_Z$ of the natural map of complexes $\rho : \cC(\cF_k) \to \cF_k^\mathrm{an}$ defined image of this map (where $\cF_k^\mathrm{an}$ was defined in Definition \ref{defn: dR cohomology}). The definition of $\rho^1$ is 
\[
\rho^1 : \cC(\cF_k)^1 = (\cH_{k-2} \otimes \Omega)(W_1)_{X^\ord} \oplus \frac{(\cH_{k-2})_\eta}{(\cH_{k-2})_E} \ni (\beta, \alpha) \mapsto \beta \in (\cH_{k-2} \otimes \Omega)(W_1) = (\cF_k^\mathrm{an})^1. 
\]
As we do this, we also extend Coleman's definition of $s_Z$ \cite[p.\ 136]{coleman1994-barsotti} by defining $\int_{E,Z} : (\cH_{k-2} \otimes \Omega)_\eta^{\mathrm{res}=0} \to (\cH_{k-2})_\eta$, which is valued in $(\cH_{k-2})_\eta$ instead of $(\cH_{k-2})_\eta/(\cH_{k-2})_E$. 

\begin{defn}
    \label{defn: s_Z and int}
    We define a map of complexes 
    \[
    s_Z : \mathrm{image}(\rho) \to \cC(\cF_k)
    \]
    that is a section of $\rho$, as follows. In degrees $0$ and $2$ there is a natural choice, since $\rho^0$ is an isomorphism onto its image and $(\cF_k^\mathrm{an})^2 = 0$. In degree 1, let $\beta \in \cC(\cF_k)^1$ be in the image of $\rho^1$. Using the $z_x$-series expansions $\beta\vert_{E_x} = \sum_{n \in \Z} b_{x,n} z_x^n \frac{dz_x}{z_x} \in (\cH_{k-2} \otimes \Omega)_{\eta_x}$ of Definition \ref{defn: z_x series}, let $s_Z^1(\beta) := (\beta, \alpha')$, where
    \[
    \alpha' := \left(\sum_{n=-\infty}^{-1} \frac{b_{x,n}}{n}z_x^n \right)_{x \in SS} \in \frac{(\cH_{k-2})_\eta}{(\cH_{k-2})_E}. 
    \]
\end{defn}

One really nice aspect of Coleman's argument for (2), which generalizes straightforwardly to weight $k$, is that the 1-cocycle condition for $(\beta,\alpha) \in Z^1(\cC(\cF_k))$ implies that $\beta \in (\cH_{k-2}\otimes \Omega)_Z(W_1)$. The key statement implying this is Lemma \ref{lem: Omega z elements}: the 1-cocycle condition is exactly the condition necessary to make $\beta$ fit the criteria for membership in $(\cH_{k-2}\otimes \Omega)_Z(W_1)$ according to the lemma. Thus, the form of $s_Z^1(\beta)$ can be controlled, from which Coleman deduces that the section 
\[
s_Z^* : H^1_\mathrm{par}(X_{\Z_p}^\mathrm{an}, \cF_k) \to H^1(\cC(\cF_k))
\]
of the surjection $H^1(\cC(\cF_k)) \rsurj H^1_\mathrm{par}(X_{\Z_p}^\mathrm{an}, \cF_k) \subset H^1_\mathrm{par}(X^\mathrm{an}, \cF_k)$ (see Definition \ref{defn: integral analytic dR}) is independent of $Z$. 

With the maps $s_Z$ in hand, we prepare to show that the $\KS_k$ map presenting parabolic de Rham cohomology in Proposition \ref{prop: Coleman pres dR} respects integral structures, i.e.\ the map $S_k^{\dagger,\crit} \to H^1_\mathrm{par}(X_{\Q_p}^\mathrm{an},\cF_k)$ has image contained in $H^1_\mathrm{par}(X_{\Z_p}^\mathrm{an},\cF_k)$. 

\begin{lem}
    \label{lem: coord integrality}
    Let $X^\mathrm{an} \supset W \supset X^{\ord}$ be an open analytic subspace with underlying affinoid $X^\ord$. Under the ``canonical, but not functorial'' splitting 
    \[
    \cH_{k-2} \cong \omega^{2-k} \oplus \omega^{(2-k)+2} \oplus \dotsm \oplus \omega^{k-2}, 
    \]
    the $X^\ord$-integrality property of elements of $\cH_{k-2}(W)$ is characterized by $X^\ord$-integrality with respect to the coordinates $\omega^i(W)$. The same result holds under the splitting and Kodiara--Spencer isomorphism
    \[
    \cH_{k-2} \otimes \Omega \cong \bigoplus_{i=0}^{k-2} (\omega^{2-k+2i} \otimes \Omega) \buildrel{\KS_k}\over\cong \bigoplus_{i=0}^{k-2} \omega^{4-k+2i}. 
    \]
\end{lem}
\begin{proof}
    This splitting, described in \cite[\S A1.2]{katz1973} at level $1$, exists over $X \otimes_\Z \Z[1/6N]$ (for our $X$, of level $\Gamma_1(N)$) according to \cite[\S3]{coleman1994}. Therefore the splitting exists over $\Z_p$ under our running assumptions $p \geq 5$, $p \nmid N$.
\end{proof}

\begin{lem}
    \label{lem: form integrality}
    Let $k, k', k'' \in \Z$ such that $k', k'' \geq 2$. Let $f \in M_{k,\Q_p}^\dagger = \omega^k(W_1)_{\Q_p}$. Then $f$ has $p$-integral $q$-series if and only if $f \in \omega^k(W_1)_{X^\ord}$ if and only if an element of $\cH_{k'-2}(W_1)$ or $(\cH_{k''-2} \otimes \Omega)(W_1)$ arising from $f$ according to either of the coordinate expansions of Lemma \ref{lem: coord integrality} lies in its $X^\ord$-integral submodule. In particular, if in addition $k \geq 2$, then the equivalent conditions above are also equivalent to $\KS_k(f) \in (\cH_{k-2} \otimes \Omega)(W_1)_{X^\ord}$. 
\end{lem}

\begin{proof}
    The first equivalence is a consequence of \eqref{eq: q vs coh} because $X^\ord$ is the unique minimal underlying affinoid of $W_1$.  The rest of the equivalences follow from Lemma \ref{lem: coord integrality}. 
\end{proof}

\subsubsection{Application: Proof of Proposition \ref{prop: integral quotient map}}

We are finally prepared to deduce our main goal, Proposition \ref{prop: integral quotient map}, from these developments. First, we show that the map $S_k^{\dagger,\crit} \to H^1_\mathrm{par}(X_{\Q_p}, \cF_k)$ above is valued in $H^1_\mathrm{par}(X_{\Z_p}, \cF_k)$.

\begin{lem}
    \label{lem: S-integral implies dR-integral}
    Let $k \in \Z_{\geq 2}$ and let $f \in S_k^{\dagger,\crit}$. Then $\KS_k(f) \in (\cH_{k-2} \otimes \Omega)(W_1)$ lies in the image of $Z^1(\cC(\cF_k)) \to (\cF_k^\mathrm{an})^1$. 
\end{lem}

\begin{proof}
    By Lemma \ref{lem: F-integral implies integral}, the lattice $S_k^{\dagger,\crit} \subset S_{k,\Q_p}^{\dagger,\crit}$ is contained in the lattice cut out by the condition that $q$-series are $p$-integral. Therefore, by Lemma \ref{lem: form integrality}, we know that $\KS_k(f) \in (\cH_{k-2} \otimes \Omega)(W_1)_{X^\ord}$. Hence $\KS_k(f)\vert_E \in (\cH_{k-2} \otimes \Omega)_\eta$. Because $f$ is critical, $\KS_k(f)$ has trivial residues at $SS$ \cite[\S7, Rem.\ 3]{coleman1996}. Defining $s^1_Z(\beta)$ as in Definition \ref{defn: s_Z and int}, we see $s^1_Z(\beta)$ is a 1-cocycle for $\cC(\cF_k)$ because its boundary (in $(\cH_{k-2} \otimes \Omega)_\eta/(\cH_{k-2} \otimes \Omega)_E$) consists only of $z_x$-series with positively indexed terms. These positively indexed terms have integral $z_x$-series coefficients due to the fact that $\KS_k(f)\vert_E \in (\cH_{k-2} \otimes \Omega)_\eta$, and thus this boundary vanishes. 
\end{proof}

\begin{rem}
    The argument for Lemma \ref{lem: F-integral implies integral} shows that the choice of $S_k^{\dagger,\crit}$ as those forms that are integral with respect to $U'$ instead of with respect to $U$ does not affect the image of $\pi_k$. It only possibly affects the extension lattice of this image by $M_k^\mathrm{ao}$, and, as discussed in Remark \ref{rem: possible difference in lattice}, there is no difference in the CM case described there. 
\end{rem}

Lemma \ref{lem: S-integral implies dR-integral} has provided a map we call $\pi'_k$,
\[
\pi'_k : S_k^{\dagger,\crit} \to H^1_\mathrm{par}(X_{\Z_p}^\mathrm{an}, \cF_k)^\crit
\]
sending $f \mapsto [\KS_k(f)]$. It remains to show that $\pi'_k$ is surjective.

\begin{prop}
    \label{prop: pi-prime surjective}
    Let $k \in \Z_{\geq 2}$. $\pi'_k : S_k^{\dagger,\crit} \to H^1_\mathrm{par}(X_{\Z_p}^\mathrm{an},\cF_k)^\crit$ is surjective. 
\end{prop}

First let us complete the deduction of Proposition \ref{prop: integral quotient map} from Proposition \ref{prop: pi-prime surjective}. 
\begin{proof}[{Proof of Proposition \ref{prop: integral quotient map}}]
    On one hand, we can consider the composition 
    \[
    S_k^{\dagger,\crit} \buildrel{\pi'_k}\over\rsurj H^1_\mathrm{par}(X_{\Z_p}^\mathrm{an}, \cF_k)^\crit \buildrel{\text{\ref{thm: p-integral dR cris}}}\over\to \frac{H^1_\mathrm{cris}(X/\Z_p, \cF_k)^\mathrm{crit}}{(\mathrm{tors})} \buildrel{\text{\ref{prop: cris dR}}}\over\to \frac{\bH^1_\mathrm{par}(X, \cF_k)^\crit}{(\mathrm{tors})} 
    \]
    using the isomorphisms of Theorem \ref{thm: p-integral dR cris} and Proposition \ref{prop: cris dR}. From there we can project to the Hodge quotient, 
    \[
    \frac{\bH^1_\mathrm{par}(X, \cF_k)^\crit}{(\mathrm{tors})} \to e(T_p)H^1(X, \omega^{2-k}).
    \]
    Because the $U$-critical summand of $e(T_p)H^1_\mathrm{cris}(X/\Z_p, \cF_k)/(\mathrm{tors})$ is the complement of the $U$-ordinary summand, and because the Hodge sub $e(T_p)H^0(X, \omega^k(-C))$ has $U$-ordinary projection given by the stabilization map, the compatibility of Serre duality maps under stabilization (Proposition \ref{prop: duality wt k}) implies that this projection is an isomorphism. We call the resulting surjection $\pi''_k : S_k^{\dagger,\crit} \rsurj e(T_p)H^1(X, \omega^{2-k})$. 

    On the other hand, reading off the definition of $\pi_k$ and recalling that the Poincar\'e and Serre duality maps are compatible in the sense of Proposition \ref{prop: Hodge SES}, Proposition \ref{prop: duality wt k} implies that the result of $\pi_k$, composed with the stabilization isomorphism $j_k : e(F)H^1_c(X^\ord, \omega^{2-k}) \isoto e(T_p)H^1(X, \omega^{2-k})$, yields the same result. That is, $j_k \circ \pi_k = \pi''_k$. Therefore $\pi_k$ is surjective. 
\end{proof}

The proof of Proposition \ref{prop: pi-prime surjective} relies on the following $\Z_p$-integral implication of the statement of Griffiths transversality appearing in \cite[Lem.\ 4.2]{coleman1996}. Namely, the natural map
\begin{equation}
    \label{eq: almost summands}
    \nabla_{k-2}(\Fil^1 \cH_{k-2}(W_1)_{X^\ord}) \oplus \KS_k(\omega^k(W_1)_{X^\ord}) \to (\cH_{k-2} \otimes \Omega)(W_1)_{X^\ord}
\end{equation}
is injective with $\Z_p$-torsion cokernel. In fact, based on the proof of \cite[Lem.\ 4.2]{coleman1996} and the calculation of $\nabla_{k-2}$ in \cite[Eq (9), p.\ 34]{coleman1994}, the exponent of the kernel is $p^a$ where $a = \lfloor \log_p (k-2) \rfloor$. Even more precisely, we see \textit{ibid.}\ that the map on graded pieces induced by $\nabla_{k-2}$ 
\[
\frac{\Fil^i(W_1)_{X^\ord}}{\Fil^{i+1}(W_1)_{X^\ord}} \to \frac{(\Fil^{i-1} \otimes \Omega)(W_1)_{X^\ord}}{(\Fil^i \otimes \Omega)(W_1)_{X^\ord}} \qquad 1 \leq i \leq k-2
\]
is $A(W_1)_{X^\ord}$-linear, equal to $i \cdot \KS_{2-k+2i}$. 

Here is the implication of the discussion above that we will need. 
\begin{lem}
    \label{lem: torsion complement}
    For any $\beta \in (\cH_{k-2} \otimes \Omega)(W_1)_{X^\ord}$, there exists $f \in M_k^\dagger := \omega^k(W_1)_{X^\ord}$ and $a \in \Z_{\geq 0}$ such that $p^a \cdot (\beta + \KS_k(f)) \in \nabla_{k-2}(\Fil^1 \cH_{k-2}(W_1)_{X^\ord})$.
\end{lem} 

The lemma implies that the maximal $\Z_p$-torsion free quotient of 
\[
\frac{(\cH_{k-2} \otimes \Omega)(W_1)_{X^\ord}}{\nabla_{k-2} \left((\cH_{k-2})(W_1)_{X^\ord}\right)}
\]
is spanned by $M_k^\dagger$. This is exactly the kind of conclusion we want to drive to in the following proof, where we must adapt it to study $H^1(\cC(\cF))$, at least its $U$-critical part. 

\begin{proof}[{Proof of Proposition \ref{prop: pi-prime surjective}}]
    In this proof, we let $U$ denote the action on sections of $\cH_{k-2}$ and $\cH_{k-2} \otimes \Omega$ denoted by $V_{k-2}$ in \cite[Prop.\ 4.1]{coleman1996}, which proves the $U$-equivariance of $\KS_k$. We will work with the $U$-critical part of these sections, noting that $\KS_k$ is $U$-equivariant by \textit{ibid.}, as are $\nabla_{k-2}$ and the filtration $\Fil^i \cH_{k-2}$ [\textit{loc.\ cit.}, Eqs (2.2), (3.1)]. We remark that $U$-critical subs are the same as $F$-ordinary summands (where $F$ is Frobenius) by [\textit{loc.\ cit.}, Eq (2.4)], so there exists the usual $F$-ordinary projector, which is well-behaved even on $\Z_p$-torsion modules. Thus $(-)^\crit$ can be thought of as $e(F)(-)$ in what follows. 
    
    Because $M_k^{\dagger,\crit}$ consists of $\Z_p$-valued $q$-series (Lemma \ref{lem: F-integral implies integral}), Lemmas \ref{lem: coord integrality} and \ref{lem: form integrality} and the $U$-compatibilities listed above imply that $\KS_k(M_k^{\dagger,\crit}) \cong (\omega^{k-2} \otimes \Omega)(W_1)_{X^\ord}^\crit \subset (\cH_{k-2} \otimes \Omega)(W_1)_{X^\ord}^\crit$. 
    
    Let $(\beta,\alpha) \in Z^1(\cC(\cF_k))$ such that $\beta \in (\cH_{k-2} \otimes \Omega)(W_1)_{X^\ord}^\crit$. By Lemma \ref{lem: torsion complement} and the listed $U$-equivariances above, there exists $f \in M_k^{\dagger,\crit}$ and $a \in \Z_{\geq 0}$ such that $p^a \cdot (\beta + \KS_k(f)) \in \nabla_{k-2}(\Fil^1 \cH_{k-2}(W_1)_{X^\ord}^\crit)$. Indeed, Lemma \ref{lem: torsion complement} implies that the complementary summand $\KS_k(M_k^\dagger)$ of $\Fil^1 \cH_{k-2}(W_1)_{X^\ord}$ spans the quotient of $H^1(\cC(\cF_k))$ by its $\Z_p$-torsion submodule; and this quotient is identified with $H^1_\dR(X_{\Z_p}^\mathrm{an}, \cF_k)$ by definition of the latter.

    Let $g \in (\Fil^1\cH_{k-2})_{X^\ord}^\crit$ such that $\nabla_{k-2}(g) = p^a(\beta + \KS_k(f))$. Using Definition \ref{defn: s_Z and int} and the existence of $s_Z^1(\KS_k(f))$ proved in Lemma \ref{lem: S-integral implies dR-integral}, we see that 
    \[
    p^a \cdot \left((\beta,\alpha) + s_Z^1(\KS_k(f))\right) \in Z^1(\cC(\cF_k)). 
    \]
    Writing $s_Z^1(\KS_k(f)) = (\KS_k(f), \alpha')$, one can subtract the boundary (in the complex $\cC(\cF_k)$) of $g \in (\cC(\cF_k))^0$ to get the cohomologous element 
    \[
    (0, p^a \cdot(\alpha + \alpha') - g\vert_E) \in Z^1(\cC(\cF_k)).
    \]
    Because $\nabla_{k-2}(p^a \cdot(\alpha + \alpha' - g\vert_E) = 0$, the trivialization \eqref{eq: trivializations} maps it to an element $h \in (\cH_{k-2})_\eta$. It will complete the proof to show that there exists $b \in \Z_{\geq 0}$ such that $p^b \cdot h \in (\cH_{k-2})_E$. This will follow from the functions $h_x \in A_{\eta_x} \otimes_{W(\F_x)} \Sym_{W(\F)}^{k-2} H^1_\cris(B_x/W(\F_x))$, giving the coordinates of $h\vert_{E_x}$ according to \eqref{eq: trivializations}, being bounded in the limit approaching the supersingular end of $E_x$. That is, we must show
    \[
    \lim_{w \to (\frac{p}{p+1})^-} \mathrm{min}_{v(z_x)=w} v(h_x(z_x)) > -\infty.  
    \]
    
    This boundedness will follow from the extension of $\beta$ and $\KS_k(f)$ to some annulus $E'_x$ containing $E_x$, with the proper containment occurring on the supersingular end of $E_x$. We want to find an oriented and coordinate-preserving immersion $E_x \rinj \cA(r',1)$ for some $r' < p^{-p/(p+1)}$ (where $\cA(r,s)$ is the standard annulus of Definition \ref{defn: coleman-barsotti}). 
    
    It will suffice to show that $\beta$ and $f$ have such extensions because this will imply the extended convergence of the anti-derivatives $\alpha, \alpha', g$. Under Lemma \ref{lem: coord integrality}, we can decompose $\beta$ in terms of its coordinates, which are integral on $X^\ord$. These coordinates can be interpreted as overconvergent modular forms lying in $\omega^i(W_1)_{X^\ord}$ for some $i \in \Z$. While $U$ does not preserve this coordinate decomposition of $(\cH_{k-2} \otimes \Omega)(W_1)$, the $U$-operator on overconvergent modular forms is defined as the $U$-action on the associated graded \cite[p.\ 220]{coleman1996}. Thus because $\beta$ is $U$-critical, these overconvergent eigenforms comprising the coordinates of $\beta$ have non-zero $U$-eigenvalues. Now we apply Buzzard's analytic continuation result \cite[Thm.\ 4.2]{buzzard2003}, which finds exactly such annuli intermediately between $X_0(p)^\mathrm{an} \supset \pi_1^{-1}(W_1)$ such that any overconvergent eigenform (pulled back along the isomorphism $\pi_1\vert_{\pi_1^{-1}(W_1)} : \pi_1^{-1}(W_1) \to W_1$) extends to this annulus as long as it has a non-zero $U$-eigenvalue. 
\end{proof}

\begin{rem}
    The only applications of the $U$-critical (i.e.\ $F$-ordinary) assumption in the proof above were that $s_Z^1(\KS_k(f))$ exists and that the $U$-eigenvalue is not zero. Hence it seems possible for the proof to work under weaker assumptions. 
\end{rem}

\subsection{Twist-ordinary forms}
\label{subsec: tord}

Next we aim to $\Lambda$-adically interpolate the classical \emph{critical} forms $M_k^\crit$ and cusp forms $S_k^\crit$, appearing as a saturated $\Z_p$-submodules of $M_k^{\dagger,\crit}$ and $S_k^{\dagger,\crit}$, respectively, in Proposition \ref{prop: classical is saturated}. To do this we use Hida theory. The idea is to change the action of $U=U_p \in \bT_{\Gamma_0(p)}$ on $M_k^\crit$ by letting it act as $U'$ instead, so that the new $U_p$-action has slope zero; we call $M_k^\crit$ with this new $\bT_{\Gamma_0(p)}$-module structure $M_k^\tord$ for ``twist-ordinary.'' 

\begin{defn}
\label{defn: tord forms}
Let $M_k^\tord$ denote $M_k^\crit$, the difference being that $M_k^\crit$ is considered to be a $\bT[U]$-module while $M_k^\tord$ is considered to be a $\bT[U']$-module with $U'$-action as specified in Definition \ref{defn: avatar}. Likewise, write $S_k^\tord$ for $S_k^\crit$ with its $\bT[U']$-action. Then, define $\bT_k^\tord := \bT[U'](M_k^\tord)$ and $\bT_k^{\tord,\circ} := \bT[U'](S_k^\tord)$. We have the usual $\Z_p$-perfect pairings
\[
\lr{} : \bT_k^\tord \times M_k^\tord \to \Z_p \text{ and } 
\bT_k^{\tord,\circ} \times S_k^\tord \to \Z_p,
\quad \lr{T,f} = a_1(T \cdot f) 
\]
and the corresponding alternate $q$-series $f = \sum_{n \geq 1} \lr{T_n, f} q^n$. We write 
\begin{equation}
    \label{eq: zeta}
    \zeta_k : M_k^\tord \rinj M_k^{\dagger,\crit}, \quad S_k^\tord \rinj S_k^{\dagger,\crit}
\end{equation}
for the $\bT[U']$-equivariant inclusions, which are $\Z_p$-saturated by Proposition \ref{prop: classical is saturated}. Our convention is to think of $M_k^\tord, M_k^{\dagger,\crit}$, $S_k^\tord, S_k^{\dagger,\crit}$ primarily as $\bT[U']$-modules while thinking of $M_k^\ord, M_k^\crit$, $S_k^\ord, S_k^\crit$ as $\bT[U]$-modules; see \eqref{eq: identifications} for a concise summary of some of the relations between these. 
\end{defn}

\begin{rem}
    \label{rem: why tord}
    One might ask why we bother to refer to these classical forms as ``twist-ordinary'' when the term ``critical'' makes perfectly good sense. One benefit is that it will keep terminology reasonably consistent for interpolations. For example, we can then use ``critical $\Lambda$-adic forms'' to refer to the interpolation of ``critical \emph{overconvergent} forms $M_k^{\dagger,\crit}$ over weights $k$.'' Another reason is that the interpolation properties of ordinary forms (i.e.\ Hida theory) are very well known, and these twist-ordinary forms have similar properties. Finally, when we discuss classical forms with non-trivial level at $p$ in future work, the need for the ``twist'' will become more apparent: a twist by an Atkin--Lehner operator on modular forms, which invokes a corresponding twist of the Galois representation. 
\end{rem}

\begin{rem}
When formulating the twist-ordinary interpolation, we will skip weight $k=2$ in favor of weights $k \in \Z_{\geq 3}$ to keep things as simple as possible, just like setting up the interpolation in Hida theory. The rationales in the two cases are related, arising from Steinberg (at $p$) forms. See e.g.\ \cite[Cor.\ 4.5(4)]{BP2022} for usual Hida theory. For the twist-ordinary interpolation, note that Steinberg forms of weight $2$ are ordinary for both $U_p$ and $U_p^*$ (see e.g.\ \cite[Note, p.\ 469]{gross1990}). 
\end{rem}

\begin{prop}
    \label{prop: tord is ord}
    Let $k \in \Z_{\geq 3}$. There are isomorphisms $\bT_k^\ord \isoto \bT_k^\tord$, $\bT_k^{\ord,\circ} \isoto \bT_k^{\tord,\circ}$, each of which is equivariant with respect to the $\bT[\,]$-algebra map $\bT[U] \to \bT[U'], U \mapsto U'$. There is an isomorphism of modules lying over $\bT_k^\ord \isoto \bT_k^\tord$, $\iota_k : M_k^\ord \isoto M_k^\tord$, $S_k^\ord \isoto S_k^\tord$, under which the $q$-series of $M_k^\ord$ maps to the alternate $q$-series of $M_k^\tord$ by 
    \[ 
    \sum_{n \geq 0} a_n q^n \mapsto \sum_{n \geq 1} a_n q^n. 
    \]
\end{prop}

\begin{proof}
    The isomorphism $M_k^\tord \isoto M_k^\ord$ is the composition of the inverse $U_p$-critical stabilization map $M_k^\crit \isoto e(T_p)H^0(X,\omega^k)$ followed by ordinary $U_p$-stabilization map $e(T_p)H^0(X,\omega^k) \isoto M_k^\ord$. The argument for Proposition \ref{prop: duality wt k} from \cite[Prop.\ 1.3.2]{ohta2005}, proving that ordinary stabilization is an isomorphism, extends to the critical stabilization. The same maps comprise isomorphisms in the cuspidal case. The main input is that the new-at-$p$ (Steinberg) forms of level $\Gamma_1(N) \cap \Gamma_0(p)$ are neither $U_p$-critical nor $U_p$-ordinary for $k \geq 3$. The claim about $q$-series is immediate from the defintions. 
\end{proof}

With Proposition \ref{prop: tord is ord} in place, we immediately deduce the standard Hida-theoretic result that its Hecke algebra interpolates. 

\begin{defn}
    \label{defn: Lambda-adic tord}
    Let $\bT^\tord_\Lambda$ be the image of $\bT[U']$ in $\prod_{k \geq 3} \bT_k^\tord$. Likewise, let $\bT^{\tord,\circ}_\Lambda$ be the image of $\bT[U']_\Lambda$ in $\prod_{k \geq 3} \bT_k^{\tord,\circ}$. 
\end{defn}

\begin{prop}
    \label{prop: tord control}
    The $\Lambda$-algebras $\bT^\tord_\Lambda$, $\bT^{\tord,\circ}_\Lambda$ are finite and flat.     There are natural specialization isomorphisms for $k \in \Z_{\geq 3}$
    \[
    \bT_\Lambda^\tord \otimes_{\Lambda,\phi_k} \Z_p \isoto \bT_k^\tord, \quad 
    \bT_\Lambda^{\tord,\circ} \otimes_{\Lambda,\phi_k} \Z_p \isoto \bT_k^{\tord,\circ}. 
    \]
    The natural $\Lambda$-algebra map $\bT_\Lambda^\tord \to \bT_\Lambda^{\tord, \circ}$ is surjective. There are isomorphisms $\bT_\Lambda^\ord \isoto \bT_\Lambda^\tord$, $\bT_\Lambda^{\ord,\circ} \isoto \bT_\Lambda^{\tord,\circ}$ which are compatible with the $\bT[\,]$-algebra map $\bT[U]_\Lambda \to \bT[U']_\Lambda, U \mapsto U'$. 
\end{prop}

\begin{defn}
    \label{defn: Lambda-adic tord forms}
    Let $M_\Lambda^\tord := \Hom_\Lambda(\bT_\Lambda^\tord, \Lambda)$ and $S_\Lambda^\tord := \Hom_\Lambda(\bT_\Lambda^{\tord,\circ}, \Lambda)$, so that the natural $\Lambda$-perfect duality pairings are
    \[
    \lr{} : \bT_\Lambda^\tord \times M_\Lambda^\tord \to \Lambda, \quad 
    \bT_\Lambda^{\tord,\circ} \times S_\Lambda^\tord \to \Lambda. 
    \]
    Under these pairings, the surjection $\bT_\Lambda^\tord \rsurj \bT_\Lambda^{\tord,\circ}$ is dual to the inclusion $S_\Lambda^\tord \rinj M_\Lambda^\tord$. Their alternate $q$-series realizations $M_\Lambda^\tord \rinj q\Lambda\lb q\rb$ are given by $f \mapsto \sum_{n \geq 1} \lr{T_n,f} q^n$ (because $U_p$-critical forms have trivial constant terms). 
\end{defn}

Upon specialization $\phi_k : \Lambda \to \Z_p$, the pairings of Definition \ref{defn: Lambda-adic tord forms} recover the pairings in weight $k$ as in Definition \ref{defn: tord forms}. In particular, there are specialization isomorphisms for $k \in \Z_{\geq 3}$
\[
    M_\Lambda^\tord \otimes_{\Lambda,\phi_k} \Z_p \isoto M_k^\tord, \qquad S_\Lambda^\tord \otimes_{\Lambda,\phi_k} \Z_p \isoto S_k^\tord. 
\]
Also, the relation between $q$-series of $M_\Lambda^\ord$ and alternate $q$-series of $M_\Lambda^\tord$ is the straightforward generalization of the weight $k$ instance in Proposition \ref{prop: tord is ord}.

It will also be very useful to show that the Hecke algebras of $M_k^\tord$ and $\cH^{1,\ord}_k(-C)$ (resp.\ $S_k^\tord$ and $\cH^{1,\ord}_k$) are naturally identical. To do this, we will apply the comparison of \cite[\S1.4]{FK2012} between the additive and multiplicative models of modular curves. Recall from \S\ref{subsec: conventions} that these are denoted ``$X'$'' and ``$X$'' respectively. In the following proposition, we use the maps $w_{Np} : X'_0(p) \isoto X_0(p)$ and $v_N : X'_{\Z_p[\zeta_N]} \isoto X_{\Z_p[\zeta_N]}$ of \cite[\S1.4]{FK2012}. 

\begin{prop}
\label{prop: realize Ttord}
Let $k \in \Z_{\geq 3}$. There are $\bT[U']$-algebra isomorphisms 
\begin{gather*}
\bT_k^\tord \cong \bT[U'](e(U_p^*)H^0(X_0(p), \omega^k)) \cong \bT[U'](e(F)H^1_c(X^\ord,\omega^{2-k}(-C))) \\
\bT_k^{\tord,\circ} \cong \bT[U'](e(U_p^*)H^0(X_0(p), \omega^k(-C))) \cong \bT[U'](e(F)H^1_c(X^\ord,\omega^{2-k}))
\end{gather*}
arising from a composition of isomorphisms (described in the proof) of $\bT[U']$-modules (resp.\ $\bT[T_p]$-modules) of classical forms, where the the pair below each item denotes where the pair $(U', T_n)$ in $\bT[U']$ (resp.\ $(T_p, T_n)$ in $\bT[T_p]$) ($p \nmid n \in \Z_{\geq 1}$) is sent, given by 
\begin{gather*}
    \begin{matrix}
        M_k^\tord \\ (U', T_n)
    \end{matrix} \buildrel{s^{-1}}\over\to   
    \begin{matrix}
        e(T_p^*)H^0(X, \omega^k) \\ (T_p^*\lr{p}_N, T_n)
    \end{matrix} \mathrel{\mathop{\to}_{v_N^{-1}/\Z_p[\zeta_N]}^\sim}
    \begin{matrix}
        e(T_p^*)H^0(X', \omega^k) \\ (T_p^*\lr{p}_N, T_n)
    \end{matrix} \buildrel{s}\over\to
    \begin{matrix}
        e(U_p^*)H^0(X'_0(p), \omega^k) \\ (U_p^*\lr{p}_N, T_n)
    \end{matrix} \\ 
    \mathrel{\mathop{\to}_{w_{Np}}^\sim}
    \begin{matrix}
        M_k^\ord \\ (U_p\lr{p}_N^{-1}, T^*_n)
    \end{matrix} \isoto 
    \begin{matrix}
    M_k^{\dagger,\ord} \\ (U\lr{p}_N^{-1}, T^*_n)
\end{matrix} \isoto 
\begin{matrix}
    \Hom_{\Z_p}(\cH^{1,\ord}_k(-C), \Z_p). \\ (U', T_n)
\end{matrix}
\end{gather*}
where the final two maps are the classicality isomorphism followed by the Serre duality isomorphism \eqref{eq: SD ord} as adapted in Definition \ref{defn: Hecke lattice for H1ord}. The cuspidal version, which is ultimately an isomorphism (over $\Z_p[\zeta_N]$) from $S_k^\crit$ to $\cH^{1,\ord}_k$, admits the same maps as isomorphisms. 
\end{prop}

\begin{proof}
Here each map labeled $s^{\pm}$ is a stabilization map or its inverse. The first isomorphism is the inverse $U_p$-critical stabilization map. We substitute $T_p^*\lr{p}_N$ for its equal operator $T_p$ on the second term to prepare notation, and this makes sense because $T_p$-ordinarity is equivalent to $T^*_p$-ordinarity. After this, $v_N^{-1}$ as in \cite[\S1.4]{FK2012} is definable over $\Z_p[\zeta_N]$, but will still induce an isomorphism of the Hecke algebras. Then $s$ is the $U_p^*$-ordinary stabilization and $w_{Np}$ is defined over $\Z_p$. 

The final compatibility of Hecke actions follows from the fact that $U' = F$ on $\cH^{1,\ord}_k(-C)$ and that its adjoint operator under $\lr{}_\mathrm{SD}$ is $U\lr{p}_N^{-1}$ by Proposition \ref{prop: interpolated Serre duality}.  
\end{proof}

\begin{rem}
    We see here that an Aktin--Lehner operator \emph{away from $p$} is compatible with Hida-theoretic interpolation. See \cite{PW2024}, especially Appendix A, for another appearance of this idea. 
\end{rem}

\begin{cor}
    \label{cor: tord serre duality}
    The isomorphisms of Proposition \ref{prop: realize Ttord} result in perfect $\bT[U']$-compatible Serre duality pairings $\lr{}'_\mathrm{SD} : S_k^\tord \times \cH_k^{1,\ord} \to \Z_p$ for $k \in \Z_{\geq 3}$, interpolable into a perfect $\bT[U']_\Lambda$-compatible Serre duality pairing $\lr{}'_\mathrm{SD} : S_\Lambda^\tord \times \cH^{1,\ord}_\Lambda \to \Lambda$. The same result in the modular case, $\lr{}'_\mathrm{SD} : M_\Lambda^\tord \times \cH_\Lambda^{1,\ord}(-C) \to \Lambda$, holds as well. 
\end{cor}

This proof is the same for both the cuspidal and modular cases, and we give it for the cuspidal case. 
\begin{proof}
    Proposition \ref{prop: realize Ttord} has already proved this over $\Z_p[\zeta_N]$, so we want to carry out descent to $\Z_p$. We have the perfect Serre duality isomorphism of \eqref{eq: SD cusp ord} between $S_k^\ord$ and $\cH^{1,\ord}_k$. What we need to check is that the composite $\bT[U']$-equivariant map of Proposition \ref{prop: realize Ttord} from $S_k^\tord$ to $S_{k,\Z_p[\zeta_N]}^{\dagger,\ord}$ is valued in $S_k^\ord$. Because $S_k^\tord$ is in perfect duality with its $\bT[U']$-Hecke algebra, it will suffice to show that $S_k^\ord$ is also the perfect dual with respect to its Hecke algebra under the non-standard $\bT[U']$-action labeled ``$(U\lr{p}_N^{-1}, T_n^*)$''. 
    
    Since again $S_k^\ord$ is in perfect duality with its $\bT[U]$-Hecke algebra, we mainly need to concern ourselves with the non-standard action of $U_\ell \in \bT[U']$ as $U_\ell^*$ for primes $\ell \mid N$. The operators away from $N$ are more  straightforward: replacing $U$ by $U\lr{p}_N^{-1}$ does not change the perfectly dual $\Z_p$-lattice within $S_k^\ord \otimes_{\Z_p} \Z_p[\zeta_N]$ since $\lr{p}_N^{-1}$ acts with slope zero. The same idea applies to replacing $T_n$ with $T_n^* = T_n\lr{n}_N^{-1}$ for $n \in \Z_{\geq 1}$ such that $(n,Np) = 1$. Now for $\ell \mid N$, as noted in \cite[p.\ 467]{gross1990}, $U^*_\ell$ is definable over $\Z_p[1/\ell] = \Z_p$. By \eqref{eq: q vs coh}, we deduce that our (perfectly dual under the non-standard $\bT[U']$-action) $\Z_p$-lattice in $S_k^\ord \otimes_{\Z_p} \Z_p[\zeta_N]$, thought of as a $\Z_p$-sublattice of $S_k^\ord \otimes_{\Z_p} \Q_p[\zeta_N]$, is contained in $S_k^\ord \otimes_{\Z_p} \Q_p$. Because we already know the lattices are equivalent over $\Z_p[\zeta_N]$, this suffices. 
\end{proof}

\subsection{Interpolation of anti-ordinary forms}
\label{subsec: aord}

We will also record the $\Lambda$-adic interpolation of the $M_k^\aord$ that appeared in Definition \ref{defn: ao wt k}. 
\begin{defn}
    \label{defn: ao hecke algebra}    
    For $k \in \Z_{\geq 3}$, let $\bT_k^\aord := \bT[U'](M_k^\aord)$, admitting $\bT[U']$-compatible $\Z_p$-perfect duality $\lr{} : T_k^\aord \times M_k^\aord \to \Z_p$ due to Corollary \ref{cor: aord duality} and the resulting alternate $q$-series, where the constant term vanishes on the image of $\theta^{k-1}$, 
    \[
    M_k \rinj q \Z_p\lb q \rb, \quad f \mapsto \sum_{n \geq 1} a_1(T_n \cdot f) q^n. 
    \]
    Let $\bT_\Lambda^\aord$ denote the image of $\bT[U']$ in $\prod_{k \geq 3} \bT_k^\aord$. 
\end{defn}

The $\Lambda$-adic interpolation will follow from Hida theory once we express the anti-ordinary Hecke algebras in terms of the usual ordinary Hecke algebras. 
\begin{prop}
    \label{prop: aord TU algebra}
    Giving $\bT_\Lambda^\ord$ (resp. $\bT_k^\ord$ for $k \in \Z_{\geq 3}$) a $\bT[U']_\Lambda$-algebra (resp.\ $\bT[U']$-algebra) structure by 
    \begin{align*}
        T_n &\mapsto [n]^{-1}T_n \quad (p \nmid n \in \Z_{\geq 1}) \\
        \lr{d}_N &\mapsto \lr{d}_N \\
        U' &\mapsto U_p^{-1}\lr{p}_N \\
        [a] &\mapsto [a]^{-1} \quad (a \in \Z_p^\times), 
    \end{align*}
    there are isomorphisms of $\bT[U']_\Lambda$-algebras $\bT_\Lambda^\ord \cong \bT_\Lambda^\aord$ and $\bT_k^\ord \isoto \bT_k^\aord$ (for $k \in \Z_{\geq 3}$) interpolating the $\theta^{k-1}$-twisted action of \eqref{eq: theta twist}. 
\end{prop}
\begin{proof}
    Recalling that $\phi_k([n]) = n^{k-1}$, the claim follows from the twisted Hecke equivariance recorded around \eqref{eq: theta twist} along with the relation $U' = p^{k-1}\lr{p}_N U_p^{-1}$ in weight $k$. 
\end{proof}

\begin{defn}
    Let $M_\Lambda^\aord := \Hom_\Lambda(\bT_\Lambda^\aord,\Lambda)$. It has its natural $\bT[U']$-compatible $\Lambda$-perfect pairing
    \[
    \lr{} : \bT_\Lambda^\aord \times M_\Lambda^\aord \to \Lambda 
    \]
    and alternate $q$-series realization $M_\Lambda^\aord \rinj q\Lambda\lb q\rb$, $f \mapsto \sum_{n \geq 1} \lr{T_n,f}q^n$. 
\end{defn}

\begin{cor}
    \label{cor: aord control}
    The $\Lambda$-algebra $\bT_\Lambda^\aord$ is finite and flat. There are natural control isomorphisms for $k \in \Z_{\geq 3}$
    \[
    \bT_\Lambda^\aord \otimes_{\Lambda,\phi_k} \Z_p \isoto \bT_k^\aord, \quad M_\Lambda^\aord \otimes_{\Lambda,\phi_k} \Z_p \isoto M_k^\aord
    \]
    compatible with the duality $\lr{}$ of Definition \ref{defn: ao hecke algebra}. 
\end{cor}

\begin{rem}
\label{rem: not a p-adic modular form} 
These alternate (that is, $\bT[U']$-based) $q$-series are only sometimes the (usual, $\bT[U]$-based) $q$-series of a $p$-adic modular form. This is in contrast with the twist-ordinary case, where the isomorphisms of Proposition \ref{prop: realize Ttord} can be used to show that the alternate $q$-series of $M_k^\tord$ is the $q$-series of some ordinary $p$-adic form. Consider the case of an eigenform: a $\bT[U']$-eigenform in $M_{k,\oQ_p}^\aord = M_{2-k,\oQ_p}^{\dagger,\ord}(k-1)$ may not be classical (equivalently, the associated Galois representation may not be de Rham at $p$). Thinking of its alternate $q$-series as a $q$-series, if this $q$-series was the $q$-series of a $p$-adic modular $\bT[U]$-eigenform, it would have $U$-slope zero, and therefore it would be classical. The upshot is that while it is possible to consider these alternate $q$-series for the purpose of interpolation because they have $U'$-slope zero, they are not the $q$-series of $p$-adic modular forms. 
\end{rem}

\section{Construction of $\Lambda$-adic critical overconvergent forms}
\label{sec: construction}

We follow the pattern of Hida \cite{hida1986a, hida1986},  beginning with constructing a $\Lambda$-adic Hecke algebra. 

\subsection{The construction}

We produce the critical $\Lambda$-adic Hecke algebra following the pattern of Definition \ref{defn: Lambda-adic tord}. 

\begin{defn}
    \label{defn: crit Lambda Hecke alg}
    The \emph{$\Lambda$-adic critical overconvergent Hecke algebra} $\bT_\Lambda^\crit$ is the image of $\bT[U']$ in 
    \[
    \prod_{k \geq 3} \End_{\Z_p}\left(M^{\dagger,\crit}_k\right). 
    \]
    Likewise, the \emph{cuspidal $\Lambda$-adic critical overconvergent Hecke algebra} $\bT_\Lambda^{\crit,\circ}$ is the image of $\bT[U']$ in 
    \[
    \prod_{k \geq 3} \End_{\Z_p}\left(S^{\dagger,\crit}_k\right). 
    \]
\end{defn}

\begin{thm}
    \label{thm: main construction}
    $\bT_\Lambda^\crit$ (resp.\ $\bT_\Lambda^{\crit,\circ}$) is a finite flat $\Lambda$-algebra. For any $k \in \Z_{\geq 3}$, there is a canonical isomorphism of $\bT[U']$-quotients 
    \[
    \bT_\Lambda^\crit \otimes_{\phi,k} \Z_p \lrisom \bT[U'](M_k^{\dagger,\crit}), \qquad 
    \bT_\Lambda^{\crit,\circ} \otimes_{\phi,k} \Z_p \lrisom \bT[U'](S_k^{\dagger,\crit}).
    \]
\end{thm}
This control theorem finds a Hida-type theory underlying Coleman's observation about locally constant ranks (Corollary \ref{cor: local constant rank Coleman}). 

Because an extension of finitely generated flat $\Lambda$-modules is also finitely generated and flat, Theorem \ref{thm: main construction} follows immediately from the following theorem, which says that the Hecke actions on the short exact sequence presenting de Rham cohomology \eqref{eq: main SES k} interpolate as well as could be desired. 
\begin{thm}
    \label{thm: main extension}
    The surjection of $\Lambda$-algebras $\bT_\Lambda^\crit \rsurj \bT_\Lambda^\aord$ induced by the sum over $k \in \Z_{\geq 3}$ of the maps $\theta^{k-1}: M_k^\aord \rinj M_k^{\dagger, \crit}$ of  \eqref{eq: main SES k} has kernel that is canonically isomorphic to $S_\Lambda^\tord$. In particular, the resulting canonical short exact sequence 
   \begin{equation}
   \label{eq: main Hecke SES Lambda}
   0 \to S_\Lambda^\tord \to \bT_\Lambda^\crit \to \bT_\Lambda^\aord \to 0    
   \end{equation}
   has terms which are finitely generated and flat as $\Lambda$-modules. 
\end{thm}

Immediately from Theorem \ref{thm: main construction} we can set up a notion of $\Lambda$-adic critical overconvergent modular forms, knowing that it admits a good control theorem that is dual to the control maps of Theorem \ref{thm: main construction}. 
\begin{defn}
    Let $M_\Lambda^\crit := \Hom_\Lambda(\bT_\Lambda^\crit,\Lambda)$ and let $S_\Lambda^\crit := \Hom_\Lambda(\bT_\Lambda^{\crit,\circ},\Lambda)$. They are equipped with a natural $\bT[U']_\Lambda$-compatible perfect pairing $\lr{} : \bT_\Lambda^\crit \times M_\Lambda^\crit \to \Lambda$ and alternate $q$-series realizations 
    \[
    M_\Lambda^\crit \rinj q\Lambda\lb q\rb, \quad M_\Lambda^\crit \ni f \mapsto \sum_{n \geq 1} \lr{T_n,f} q^n \in q\Lambda \lb q\rb. 
    \]
\end{defn}
\begin{cor}
    \label{cor: critical forms control}
    There are $\bT[U']$-equivariant isomorphisms for $k \in \Z_{\geq 3}$, 
    \[
    M_\Lambda^\crit \otimes_{\phi,k} \Z_p \lrisom M_k^{\dagger,\crit}, \quad S_\Lambda^\crit \otimes_{\phi,k} \Z_p \lrisom S_k^{\dagger,\crit} .
    \]
\end{cor}

One of the first consequences is the following $\Lambda$-adic interpolation of \eqref{eq: main SES k} over weights $k \in \Z_{\geq 3}$.  
\begin{cor}
    \label{cor: main Lambda SES}
    There is a short exact $\bT[U']$ sequence of finitely generated flat $\Lambda$-modules, perfectly $\Lambda$-dual to \eqref{eq: main Hecke SES Lambda}, 
    \begin{gather*}
    0 \to M_\Lambda^\aord \xrightarrow{\Theta_\Lambda} M_\Lambda^\crit \xrightarrow{\pi_\Lambda}  \cH^{1,\ord}_\Lambda \to 0 
    \end{gather*}
    and which admits a control isomorphism to \eqref{eq: main SES k} under $\phi_k$ for all $k \in \Z_{\geq 3}$. In particular, we have control isomorphisms
    \[
    M_\Lambda^\crit \otimes_{\Lambda,\phi_k} \Z_p \risom M_k^{\dagger,\crit}, \quad 
    \bT_\Lambda^\crit \otimes_{\Lambda,\phi_k} \Z_p \risom \bT_k^{\dagger,\crit}.
    \]
\end{cor}

This corollary will play an important role in \S\ref{subsec: Lambda-adic coh} and \S\ref{sec: bi-ordinary}, where we construct $\Lambda$-adic cohomology bi-ordinary complexes, respectively. 

\subsection{Proof of Theorem \ref{thm: main extension}}

The proof relies on the variants of classical Hida theory that we set up in \S\S\ref{subsec: tord}-\ref{subsec: aord}. For easy access to these results, we recall them here. 

The $\bT[U']$-quotient modules $\pi_k : M_k^{\dagger,\crit} \rsurj \cH^{1,\ord}_k$ interpolate $\Lambda$-adically into $\cH^{1,\ord}_\Lambda$ with the usual control theorem. Its Hecke algebra $\bT[U']_\Lambda(\cH^{1,\ord}_\Lambda)$ is naturally isomorphic to the cuspidal twist-ordinary Hecke algebra $\bT^{\tord,\circ}_\Lambda$ discussed in \S\ref{subsec: tord}, while the twist-ordinary Hecke algebra $\bT_\Lambda^\tord$ interpolates the $\bT[U']$-action on $M_k^\tord$, which is just the classical critical forms $M_k^\crit$ taken to be a $\bT[U']$-module: see \eqref{eq: identifications} for a summary. 

The $\bT[U']$-submodules $M_k^\aord\subset M_k^{\dagger,\crit}$, i.e.\ the images of $M_{2-k}^{\dagger,{\rm ord}}(k-1)$ by $\theta^{k-1}$, interpolate $\Lambda$-adically into $M_\Lambda^\aord$ with the usual control theorem (\S\ref{subsec: aord}).

\begin{proof}[{Proof of Theorem \ref{thm: main extension}}]
    The injections $M_k^\aord \rinj M_k^{\dagger,\crit}$ for $k \in \Z_{\geq 3}$ produce a surjection $\bT_\Lambda^\crit \rsurj \bT_\Lambda^\aord$, in light of the defintions of these Hecke algebras (in Definitions \ref{defn: crit Lambda Hecke alg} and \ref{defn: ao hecke algebra}). Let $K$ denote the kernel of this surjection. Our goal is to establish a canonical isomorphism of Hecke modules from $K$ to $S_\Lambda^\tord$. 
    
    We observe that $K$ is naturally equipped with action maps in weights $k \in \Z_{\geq 3}$, 
    \[
    K \to \Hom_{\Z_p}(\cH^{1,\ord}_k, M_k^{\dagger,\crit}) \qquad (k \geq 3)
    \]
    defined as follows, relying on $\pi_k$ from Theorem \ref{thm: lattice main}. Notice that for $T \in K$ and $\eta \in \cH^{1,\ord}_k$ and an arbitrary choice of lift $\tilde \eta \in M_k^{\dagger,\crit}$ along $\pi_k : M_k^{\dagger,\crit} \rsurj \cH^{1,\ord}_k$, $T \cdot \tilde \eta \in M_k^{\dagger,\crit}$ is well defined. (Indeed, any two lifts of $\eta$ differ by an element of $M_k^\aord$, which is annihilated by $T$.) Let $K_k$ denote the image of the action map in weight $k$. The product of these action maps
    \[
    K \rinj \prod_{k \in \Z_{\geq 3}} K_k \subset \prod_{k \in \Z_{\geq 3}} \Hom_{\Z_p}(\cH^{1,\ord}_k, M_k^{\dagger,\crit})
    \]
    is injective by definition of $K$. 
    
    Next we define a $\Z_p$-bilinear pairing
    \begin{equation}
    \label{eq: res pairing}
    \delta_k : K_k \times \cH^{1,\ord}_k \ni (T,\eta) \mapsto a_1(T \cdot \tilde \eta) \in \Z_p     
    \end{equation}
    and claim that it is perfect. This claim will follow form the characterization of the lattice $M_k^{\dagger,\crit} \subset M_{k,\Q_p}^{\dagger,\crit}$ in Theorem \ref{thm: lattice main}. Indeed, the $a_1$ pairing on $M_k^{\dagger,\crit} \times \bT_k^{\dagger,\crit}$ is perfect (Proposition \ref{prop: pairing wt k}), and thereby the passage to a $\Z_p$-saturated ideal $K_k \subset \bT_k^{\dagger,\crit}$ with corresponding quotient $M_k^{\dagger,\crit} \rsurj \cH^{1,\ord}_k$ remains perfect. Because $\delta_k$ is Hecke-equivariant in the sense that $\delta_k(T' \cdot \eta,T) = \delta_k(\eta,T' \cdot T)$ for all $T' \in \bT[U']$, we have an isomorphism of $\bT[U']$-modules $K_k \cong \Hom_{\Z_p}(\cH^{1,\ord}_k,\Z_p)$.

Next we set up a perfect $\Lambda$-linear duality pairing $\delta$ between $K$ and $\cH^{1,\ord}_\Lambda$ by completing the diagram
\begin{equation}\label{eq: K diagram} 
\begin{aligned}
    \xymatrix{
    K \ar@{-->}[rr]^\delta \ar[d] & & \Hom_\Lambda(\cH^{1,\ord}_\Lambda,\Lambda) \ar[d] \\
    \prod_{k \geq 3} K_k \ar[rr]^(.4){\prod_k \delta_k} & & \prod_{k \geq 3} \Hom_{\Z_p}(\cH^{1,\ord}_k,\Z_p).
    }
    \end{aligned}
\end{equation}

Let $T \in K$ and $\eta \in \cH^{1,\ord}_\Lambda$. For each $k \in \Z_{\geq 3}$, we have $\delta_k(T, \eta) \in \Z_p$. Both $T$ and $\eta$ have specializations modulo $\ker \phi_k$ that satisfy the condition of Lemma \ref{lem: interpolate over Lambda}(1). (For $\eta$, this is a consequence of higher Hida theory of Theorem \ref{thm: BP}. For $T$, this is a consequence of the fact that $T$ is some $\Z_p$-linear polynomial in the $T_n \in \bT[U']$, and that these Hecke operators have actions on alternate $q$-series that are continuous in the sense of Lemma \ref{lem: interpolate over Lambda}(1) below. Indeed, their action on alternate $q$-series matches the $\bT[U]$-action on $q$-series, by definition.) Therefore the result of the pairing satisfies Lemma \ref{lem: interpolate over Lambda} as well. Bilinearity over $\Lambda$ is clear, so we now have a pairing $\delta : K \times \cH^{1,\ord}_\Lambda \to \Lambda$ making diagram \eqref{eq: K diagram} commute. 

The other three arrows in \eqref{eq: K diagram} are injective, by Theorem \ref{thm: BP} and our arguments so far. Therefore $\delta$ is also injective. (Thus we already know that $K$ is finitely genereated over $\Lambda$.) Let $C$ denote the cokernel of $\delta$, so we have a short exact sequence
\[
0 \to K \buildrel{\delta}\over\to \Hom_\Lambda(\cH^{1,\ord}_\Lambda,\Lambda) \to C \to 0.
\]
By Theorem \ref{thm: BP}, it reduces modulo $\ker\phi_k$ to an exact sequence 
\[
K \otimes_{\Lambda,\phi_k} \Z_p \to \Hom_{\Z_p}(\cH^{1,\ord}_k,\Z_p) \to C \otimes_{\Lambda,\phi_k} \Z_p \to 0.
\]
The left arrow is surjective because $\delta_k$ of \eqref{eq: res pairing} is perfect. Therefore $C \otimes_{\Lambda,\phi_k} \Z_p = 0$ for all $k \in \Z_{\geq 3}$. Therefore, because $C$ is finitely generated over $\Lambda$, Lemma \ref{lem: interpolate over Lambda}(2) yields $C=0$. Consequently we have established a canonical isomorphism of $K$ with $\Hom_\Lambda(\cH^{1,\ord}_\Lambda,\Lambda)$. 

Next, because Theorem \ref{thm: BP} implies that $\cH^{1,\ord}_\Lambda$ is $\Lambda$-flat, it is straightforward to conclude that $\delta$ is perfect. Indeed, one can substitute $\Hom_\Lambda(\cH^{1,\ord}_\Lambda,\Lambda)$ for $K$ in the other $\Lambda$-linear map $\cH^{1,\ord}_\Lambda \to \Hom_\Lambda(K,\Lambda)$ induced by $\delta$.

By Serre duality in higher Hida theory (Theorem \ref{thm: BP}) with the twists resulting in its $\bT[U']$-compatible version $\lr{}'_\mathrm{SD}$ of Corollary \ref{cor: tord serre duality}, there is a canonical $\bT[U']$-isomorphism $\Hom_\Lambda(\cH^{1,\ord}_\Lambda,\Lambda) \cong S_\Lambda^\tord$. Composing these, we have the claimed canonical isomorphism $K \cong S_\Lambda^\tord$. 
\end{proof}

\begin{lem}
    \label{lem: interpolate over Lambda}
    The following ``control'' results are true. 
    \begin{enumerate}
        \item Let $Y$ be an set of integers that has infinitely many representatives in each congruence class modulo $(p-1)$. Let $b_n \in \Z_p$ for each $n \in Y$. Then there exists $b \in \Lambda$ such that $\phi_n(b) = b_n$ if and only if for all $n,n' \in Y$ and $a \in \Z_{\geq 0}$, $n \equiv n' \pmod{(p-1)p^a} \Rightarrow b_n \equiv b_{n'} \pmod{p^{a+1}}$. 
        \item Let $X$ be a finitely generated $\Lambda$-module. Then $X=0$ if and only if $X \otimes_{\Lambda,\phi_k} \Z_p = 0$ for any single $k \in \Z_{\geq 3}$. 
    \end{enumerate}
\end{lem}

\subsection{A complementary exact sequence}

Next we prove a theorem analogous to Theorem \ref{thm: main extension} but interpolating the natural subspace of classical critical forms $\zeta_k : M^\tord_k \rinj M_k^{\dagger,\crit}$. 
\begin{prop}
    \label{prop: secondary SES Lambda}
    Let $k \in \Z_{\geq 3}$. There is a short exact sequence of $\bT[U']$-modules
    \begin{equation}
    \label{eq: secondary SES k}
        0 \to M_k^\tord \xrightarrow{\zeta_k}  M_k^{\dagger,\crit} \to Q_k \to 0,    
    \end{equation}
    where $Q_k$ is $\Z_p$-flat and $Q_k \otimes_{\Z_p} \Q_p$ is $\bT[U']$-isomorphic (non-canonically) to $S_{k,\Q_p}^\aord$. It arises by $- \otimes_{\Lambda,\phi_k} \Z_p$ from a short exact sequence of $\bT[U']$-modules that are finitely generated and flat as $\Lambda$-modules, 
    \[
    0 \to M_\Lambda^\tord \xrightarrow{\zeta_\Lambda}  M_\Lambda^\crit \to Q_\Lambda \to 0.
    \]
\end{prop}
\begin{proof}
    The claim about $Q_k \otimes_{\Z_p} \Q_p$ is an immediate consequence of Coleman's result, Proposition \ref{prop: coleman non-split}, along with an argument on the invariance of composition series (using the fact that the Hecke action on $M_{k,\C_p}^\tord$ is semi-simple). Our goal is to descend this result to $\Z_p$ and also $\Lambda$. 

    By Definition \ref{defn: tord forms} of $M_k^\tord$, the inclusion $M_k^\tord \subset M_k^{\dagger,\crit}$ in \eqref{eq: secondary SES k} is compatible with alternate $q$-series. Therefore its dual map under the $a_1$ pairing is the surjective homomorphism of $\bT[U']$-Hecke algebras $r_k : \bT_k^{\dagger,\crit} \rsurj \bT_k^\tord$. By acting on the sum of these maps over $k \in \Z_{\geq 3}$ and viewing Definitions \ref{defn: crit Lambda Hecke alg} and \ref{defn: Lambda-adic tord}, we get a surjection of $\Lambda$-adic $\bT[U']_\Lambda$-algebras $r_\Lambda : \bT^\crit_\Lambda \rsurj \bT_\Lambda^\tord$. Applying the control results, Proposition \ref{prop: tord control} and Theorem \ref{thm: main construction}, we conclude that $r_k$ arises by applying $- \otimes_{\Lambda,\phi_k} \Z_p$ to $r_\Lambda$. Because these control results also tell us that these Hecke algebras are $\Lambda$-finite and flat, $\ker r_\Lambda$ is finitely generated and flat as a $\Lambda$-module. Letting $Q_\Lambda := \Hom_\Lambda(\ker r_\Lambda,\Lambda)$ and observing that the cokernel of \eqref{eq: secondary SES k} is canonically isomorphic to $Q_\Lambda \otimes_{\Lambda,\phi_k} \Z_p$, we have the result. 
\end{proof}

\begin{thm}
    \label{thm: secondary Hecke SES Lambda}
   There are natural isomorphisms of Hecke algebras
   \[
    \bT_\Lambda^\aord \cong \bT[U'](Q_\Lambda) \cong \bT[U'](\Hom_\Lambda(Q_\Lambda,\Lambda)).
   \]
\end{thm}

\begin{proof}
    It follows from Proposition \ref{prop: secondary SES Lambda} that $\bT[U'](Q_k) \cong \bT_k^\aord$. A control argument then yields that the Hecke action on $\oplus_{k \geq 3} Q_k$ has image algebra $\bT_\Lambda^\aord$. The rest of the claims follow much as in the proof of Theorem \ref{thm: main extension}. 
\end{proof}

\subsection{$\Lambda$-adic cohomology}
\label{subsec: Lambda-adic coh}

To prepare to interpolate it over integers $k \geq 3$, we specify notation for the $p$-integral cohomology data of Theorem \ref{thm: p-integral dR cris}. 

\begin{defn}
    \label{defn: dR}
    For $k \in \Z_{\geq 3}$, let $\dR_k$ denote the following triple, which we call \emph{$p$-integral $T_p$-ordinary analytic de Rham cohomology of $X$ in weight $k$}. 
    \begin{itemize}
        \item  $e(T_p)H^1_\dR(X_{\Z_p}^\mathrm{an}, \cF_k)^0$ of Definition \ref{defn: integral analytic dR}, equipped with:
        \item the crystalline Frobenius structure $\varphi$ coming from Theorem \ref{thm: p-integral dR cris}; and
        \item the Hodge filtration given by 
        \[
        \Fil^0 = \dR_k, \Fil^i = \KS_k(e(T_p)H^0(X,\omega^k)) \text{ for } 1 \leq i < k, \Fil^k = 0.
        \]
        with quotient $e(T_p)H^1(X, \omega^{2-k})$.
    \end{itemize}
\end{defn}

From Theorem \ref{thm: lattice main}, we know that $\dR_k$ admits a presentation by $p$-integral overconvergent forms, arising from Coleman's presentation (Proposition \ref{prop: Coleman pres dR}) over $\Q_p$. Also, by Proposition \ref{prop: FV on dR}, under this presentation $U' = p^{k-1}\lr{p}_NU^{-1}$ on forms acts $\varphi$ on cohomology. We record some immediate consequences. 

\begin{prop}
    \label{prop: dR summary}
    Let $k \in \Z_{\geq 3}$. There is a $\bT[U']$-equivariant isomorphism
    \[
    M_k^\ord \oplus \frac{M_k^{\dagger,\crit}}{M_k^\aord} \isoto \dR_k
    \]
    coming from the inclusion $M_k^\ord \oplus M_k^{\dagger,\crit} \rinj M_{k,\Q_p}^\dagger$ and the projection $\KS_k : M_{k,\Q_p}^\dagger \rsurj H^1_\dR(X_{\Q_p}, \cF_k)$ of Proposition \ref{prop: Coleman pres dR}. The summands of this isomorphism are equal to the $\varphi$-critical subspace and the $\varphi$-ordinary subspace, 
    \begin{equation}
        \label{eq: dR crit and ord}
        \KS_k(M_k^\ord) \cong \dR_k^{\varphi\text{-}\crit}, \quad \text{ and } \quad \KS_k(M_k^{\dagger,\crit}) \cong \dR_k^{\varphi\text{-}\ord}.
    \end{equation} 
    The stabilization maps of Proposition \ref{prop: duality wt k} 
    \[
    i: e(T_p)H^0(X,\omega^k) \isoto M_k^\ord, \quad j : \cH^{1,\ord}_k \isoto e(T_p)H^1(X,\omega^{2-k})
    \]
    are isomorphic (under $\KS_k$) to the projection maps 
    \begin{itemize}
        \item from the Hodge subspace $\Fil^{k-1}$ to the $\varphi$-critical subspace $\KS_k(M_k^\ord)$ within $\dR_k$
        \item from the $\varphi$-ordinary subspace $\KS_k(M_k^{\dagger,\crit}) \cong \dR_k^{\varphi\text{-}\ord}\subset \dR_k$, identified with $\cH^{1,\ord}_k$ via $\pi_k$, to the Hodge quotient $e(T_p)H^1(X,\omega^{2-k})$. 
    \end{itemize}
    respectively. 
\end{prop}

The crystalline direct sum decomposition \eqref{eq: dR crit and ord} is the anchor for our $\Lambda$-adic interpolation, so we will need to carefully understand the Hodge filtration with respect to it. For $U^\star = U, U'$, we let 
\[
\tau(U^\star) = \tau_k(U^\star) := (1-p^{k-1}\lr{p}_N{U^\star}^{-2})^{-1}.
\]
The idea behind $\tau_k$ is that if $f_\alpha$ is a $U_p$-ordinary stabilization of an eigenform $f \in e(T_p)H^0(X, \omega^k)_E$ with $U_p$-eigenvalue $\alpha$, so that the other root of the $T_p$-Hecke polynomial is $\beta = p^{k-1}\chi_f(p)\alpha^{-1}$ and $f_\beta \in M_k^\crit$ denotes the $U_p$-critical stabilization, then $\tau_k(U_p)(f_\alpha) = \frac{\alpha}{\alpha-\beta}f_\alpha$ and, symmetrically, $\tau_k(U_p)(f_\beta) = \frac{\beta}{\beta-\alpha}f_\beta$. Therefore, combining the facts
\begin{itemize}
    \item  $f = (\beta f_\alpha - \alpha f_\beta)/(\beta-\alpha)$, 
    \item $p^{k-1}\lr{p}_N U^{-1}$ scales $f_\beta$ by $\alpha$, and so $U'$ scales $\zeta^{-1}(f_\beta)$ by $\alpha$, 
    \item $(\zeta_k \circ \iota_k)(f_\alpha) = f_\beta$, using the notation of \eqref{eq: identifications}, 
\end{itemize} 
we calculate that 
\begin{equation}
    \label{eq: find Hodge}
    f = [-p^{k-1}\lr{p}_N U^{-2} \cdot  \tau_k(U) \cdot f_\alpha] + [\tau_k(U')\cdot f_\beta]. 
\end{equation}

We have to be careful about integrality and existence of $\tau_k(U^\star)$, but this is quickly verified. 
\begin{lem}
    \label{lem: tau inverse}
    For $k \in \Z_{\geq 3}$, $\tau_k(U)$ is well-defined on $M_k^\ord$ with slope $0$. Likewise, $\tau_k(U')$ is well-defined on $M_k^{\dagger,\crit}$ of slope $0$. In particular, in each of these cases, $\tau_k(U^\star)$ is an automorphism. 
\end{lem}

In the following proposition, we use identification maps that are important to make clear. (Unlike most other parts of this manuscript, this is the main place where we need to deal with both $\bT[U]$-modules and $\bT[U']$-modules in a single package.) They are, for $k \in \Z_{\geq 3}$, using $\iota_k$ from Proposition \ref{prop: tord is ord}, 
\begin{equation}
    \label{eq: identifications}
    \begin{split}
     \xymatrix{
    M_k^\ord \ar[r]^\sim_{\iota_k} & M_k^\tord \ar@{=}[r] \ar@/^1pc/[rr]^{\zeta_k} & M_k^\crit \ar@{}[r]|{\subset} & M_k^{\dagger,\crit}
    }
    \\ 
\begin{matrix}
    \bT[U] & \isoto & \bT[U'] & \isoto & \bT[U] & \isoto & \bT[U'] \\
    U & \mapsto & U' & \mapsto & p^{k-1}\lr{p}_N U^{-1} & \mapsto & U'
\end{matrix}        
    \end{split}
\end{equation} 
where the lower two lines indicate the ring over which the top line is a module, and the maps between these $\bT[\,]$-algebras over which the top line of isomorphisms is covariant. The key inclusion of $\bT[U']$-modules is $M_k^\tord \subset M_k^{\dagger,\crit}$. We also have the $\Lambda$-adic interpolation 
\begin{equation}
    \label{eq: identificiations Lambda}
    M_\Lambda^\ord \quad  \mathrel{\mathop{\lra}_{\iota_\Lambda}^\sim} \quad  M_\Lambda^\tord \quad   \qquad  \mathrel{\mathop{\rinj}_{\zeta_\Lambda}}   \qquad  \quad M_\Lambda^\crit.
    \end{equation}

We also use $M_k^\ord$, identified with $e(T_p)H^0(X,\omega^k)$ via $i^{-1}$, as the source of the Hodge filtration, in order to ease the expression of the interpolation (since $M_\Lambda^\ord$ interpolates $M_k^\ord$ for $k \in \Z_{\geq 3}$ and $M_k^{\dagger,\ord}$ for $k \in \Z$; and only interpolates $e(T_p)H^0(X,\omega^k)$ for $k \in \Z_{\geq 3}$ via $i$). 

\begin{cor}
    \label{cor: Hodge wrt cris}
    Let $k \in \Z_{\geq 3}$.
    In terms of the crystalline decomposition of Proposition \ref{prop: dR summary}, the Hodge filtration is realized by the injection
    \begin{gather*}
        M_k^\ord \mathrel{\mathop{\lra}_{i^{-1}}^\sim} e(T_p)H^0(X,\omega^k) \buildrel{\KS_k\ }\over\lra M_k^\ord \oplus \frac{M_k^{\dagger,\crit}}{M_k^\aord} \isoto H^1_\dR(X^\mathrm{an}_{\Z_p}, \cF_k)^0 \\
        M_k^\ord \ni g \mapsto (-p^{k-1}\lr{p}_N U^{-2}\tau_k(U)(g), \zeta(\tau_k(U')(\iota_k(g)))) \in M_k^\ord \oplus \frac{M_k^{\dagger,\crit}}{M_k^\aord}
    \end{gather*}
    where we consider $f \in M_k^\crit$ under the identifications \eqref{eq: identifications}. 
\end{cor}

\begin{proof}
    In \eqref{eq: find Hodge} we saw that $f \in e(T_p)H^0(X, \omega^k)$ is recovered by applying the operator expression above to the stabilizations $f_\alpha$ and $f_\beta$. As indicated in and after \eqref{eq: identifications}, the $U$-ordinary stabilization $f_\alpha \in M_k^\ord$ is equal to the $U$-critical stabilization $f_\beta \in M_k^\crit$ when each is sent to $M_k^\tord$, that is, $\iota_k(f_\alpha) = \zeta_k^{-1}(f_\beta)$. 
\end{proof}

We notice that because the formula in Corollary \ref{cor: Hodge wrt cris} is symmetric in each summand, by Lemma \ref{lem: tau inverse} we can precompose with $\tau_k(U_p) : M_k^\ord \isoto M_k^\ord$ and propose the following ``formal'' definition of $\Lambda$-adic de Rham cohomology, where by \emph{formal} we mean that we give a formulation for $\dR_k$ that makes sense in every weight $k \in \Z_{\geq 3}$, but where 
\begin{itemize}
    \item the crystalline Frobenius decompositions are interpolable over varying $k$, but the Frobenius endomorphisms $\varphi_k$ are not; they are determined by a ratio that varies transcendentally over $\Lambda$ and is not an actual $\Lambda$-linear endomorphism 
    \item the Hodge filtration is the image $M_k^\ord$ under the map $\KS_k \circ i^{-1}$ of Corollary \ref{cor: Hodge wrt cris}, and while the $M_k^\ord$ interpolates to $M_\Lambda^\ord$, the maps are determined by scalars that are transcendental over $\Lambda$. 
\end{itemize}
The entire issue is that we must use a symbol ``$p^{\kappa-1}$'' to denote the association $\Z_{\geq 3} \ni k \mapsto p^{k-1} \in \Z_p$, which not interpolable over $\Lambda$. It can be considered as a formal power series variable over $\Lambda$, but we will not emphasize this. 

We make the follow definition slightly less transcendental by using the fact that the automorphism $\tau_k(U) : M_k^\ord \isoto M_k^\ord$ can be applied first as a factor of the map of Corollary \ref{cor: Hodge wrt cris}. 
\begin{defn}
    Let $\dR_\Lambda$ denote the triple
    \[
    \dR_\Lambda := (M_\Lambda^\ord \oplus \cH^{1,\ord}_\Lambda, \varphi_\Lambda, \Fil^{\kappa-1})
    \]  
    where $M_\Lambda^\ord \oplus \cH^{1,\ord}_\Lambda$ is equipped with its $\bT[U \oplus U']$ action (the sum acting coordinate-wise), and the following additional formal data. 
    \begin{itemize}
        \item $\Fil^{\kappa-1} \subset M_\Lambda^\ord \oplus \cH^{1,\ord}_\Lambda$, the image of $M_\Lambda^\ord$ under 
        \[
        h_\Lambda = (-p^{\kappa-1}\lr{p}_NU^{-2}, \pi_\Lambda \circ \zeta_\Lambda \circ \iota_\Lambda).
        \]
        \item $\varphi_\Lambda = U' = p^{\kappa-1}\lr{p}_N U \oplus U'$, where again the rightmost expression refers to an action coordinate-wise on $M_\Lambda^\ord \oplus \cH^{1,\ord}_\Lambda$.
        \item an action of $T_p = U + U'$ (which is formal on both summands), stabilizing $\Fil^{\kappa-1}$. 
    \end{itemize}
\end{defn}
Again using $p^{\kappa-1}$, we note that the formal interpolation of the automorphisms 
which is still not interpolable even though it is a $p$-unit.

\begin{rem}
According to \eqref{eq: find Hodge}, the Hodge filtration map that we are actually able to interpolate sends an eigenform $f$ of weight $k$ (as in the notation $f, f_\alpha, f_\beta$ of \eqref{eq: find Hodge}) to its multiple by $\tau_k(\alpha) = \sum_{n \geq 0} \frac{\beta^n}{\alpha^n}$. We could alternatively use the map $h'_\Lambda = (-p^{\kappa-1}\lr{p}_N U^{-2} \tau_\Lambda(U), \tau_\Lambda(U'))$ where
    $\tau_\Lambda(U) : M_\Lambda^\ord \isoto M_\Lambda^\ord$ is the \emph{formal} map 
        \[
        \tau_\Lambda(U) = (1-p^{\kappa-1}\lr{p}_N U^{-2})^{-1} = \sum_{n \geq 0} (p^{\kappa-1}\lr{p}_N U^{-2})^n. 
        \]
    The map $h'_\Lambda$ specializes along $\phi_k$ to the actual Hodge filtration map $\KS_k \circ i^{-1}$. Indeed, adding the two coordinates of $h'_\Lambda$ (after applying $\iota_\Lambda^{-1}$ to the right coordinate), the sum is $1$. 
\end{rem}

\begin{thm}
    \label{thm: main dR}
    Let $k \in \Z_{\geq 3}$. Upon specialization along $\phi_k : \Lambda \to \Z_p$, there is an $\bT[T_p]$-isomorphism of triples $\dR_\Lambda \otimes_{\Lambda,\phi_k} \Z_p \cong \dR_k$, where \begin{itemize}
        \item the symbol $p^{\kappa-1}$ is interpreted along $\phi_k$ as $p^{k-1}$, giving the specializations of $\varphi_\Lambda$ and the $T_p$-action.
        \item the symbol $\Fil^{\kappa-1}$ is interpreted along $\phi_k$ as the decreasing filtration determined by $\Fil^0 = \dR_\Lambda \otimes_{\Lambda,\phi_k} \Z_p$, $\Fil^i = 0$ for $i \geq k$, and 
        \[
        \text{for }0 < i < k, \  \Fil^i = \mathrm{image\ of\ } \KS_k \circ i^{-1} 
        \]
        as in Corollary \ref{cor: Hodge wrt cris}. Moreover, the formal map $h'_\Lambda$ specailizes to the natural map $e(T_p)H^0(X,\omega^k) \rinj \dR_k$ given by $\KS_k$. 
    \end{itemize}
\end{thm}

In the rest of the paper, we will discuss the kernel of $\pi_\Lambda \circ \zeta_\Lambda : M_\Lambda^\tord \to \cH^{1,\ord}_\Lambda$ (resp.\ of its reduction modulo $\ker\phi_k$) which amounts to those twist-ordinary (that is, classical critical) forms that are in the image of the map $\theta_\Lambda$ of Corollary \ref{cor: main Lambda SES} (resp.\ $\Theta^{k-1}$ of Proposition \ref{prop: ao sublattice}). This map gives us access to the relative position of the Hodge filtration and the crystalline decomposition without the failure of continuity over weights $k$ presented by $\tau_k(U)$ and $p^{k-1}$.

\section{Bi-ordinary complexes}
\label{sec: bi-ordinary}

The main application of critical $\Lambda$-adic forms that we will discuss is the construction of the bi-ordinary complex. This complex measures both the kernel and cokernel of the $\bT[U']$-equivariant map $\pi_\Lambda \circ \zeta_\Lambda$. 

\subsection{Initial observations about the bi-ordinary complex}

We use cochain conventions for complexes. 

\begin{defn}
    Let $k \in \Z_{\geq 3}$. We have a perfect complex of $\Z_p$-modules of length 1 with a $\bT[U']$ action, the \emph{weight $k$ bi-ordinary complex}
    \[
    \BO_k^\bullet := [M_k^\aord \oplus M_k^\tord \mathrel{\mathop{\lra}^{\theta^{k-1} + \zeta_k}}M_k^{\dagger,\crit}]
    \] 
    where the differential is the sum of the natural inclusions inherent in Definitions \ref{defn: ao wt k} and \ref{defn: tord forms}. The grading of the complex is $\BO_k^0 := M_k^\aord \oplus M_k^\tord$, $\BO_k^1 := M_k^{\dagger,\crit}$. We define the \emph{cuspidal weight $k$ bi-ordinary complex} similarly as 
    \[
    \SBO_k^\bullet := [M_k^\aord \oplus S_k^\tord \mathrel{\mathop{\lra}^{\theta^{k-1} + \zeta_k}} M_k^{\dagger,\crit}].
    \]
    
    Likewise, we have a perfect complex of $\Lambda$-modules of length 1 with a $\bT[U']$-action, the \emph{$\Lambda$-adic bi-ordinary complex} 
    \[
    \BO_\Lambda^\bullet := [M_\Lambda^\aord \oplus M_\Lambda^\tord 
    \mathrel{\mathop{\lra}^{\Theta_\Lambda + \zeta_\Lambda}} 
    M_\Lambda^\crit]
    \]
    and its cuspidal variant
    \[
    \SBO_\Lambda^\bullet := [M_\Lambda^\aord \oplus S_\Lambda^\tord \mathrel{\mathop{\lra}^{\Theta_\Lambda + \zeta_\Lambda}}  M_\Lambda^\crit].
    \]
    We also define $\BO_k^\bullet$ for general $k \in \Z$ by specializing $\BO_\Lambda^\bullet$ along $\phi_k$. 
\end{defn}

Here is a list of lemmas with immediate and basic facts about these bi-ordinary complexes. Later we will add its key additional structure, Serre self-duality. 

\begin{lem}
    \label{lem: BO control}
    For $k \in \Z_{\geq 3}$, there are canonical $\bT[U']$-equivariant control maps (of perfect complexes) 
    \[
    \BO_\Lambda^\bullet \otimes_{\Lambda,\phi_k} \Z_p \risom \BO_k^\bullet, \qquad 
    \SBO_\Lambda^\bullet \otimes_{\Lambda,\phi_k} \Z_p \risom \SBO_k^\bullet.
    \]
\end{lem}
\begin{proof}
    This follows from the control results of Corollaries \ref{cor: critical forms control} and \ref{cor: main Lambda SES} and Proposition \ref{prop: secondary SES Lambda}, as well as control results for the twist-ordinary forms and anti-ordinary forms recorded in \S\ref{subsec: tord} and \S\ref{subsec: aord}. 
\end{proof}

\begin{lem}
    \label{lem: Eis difference}
    There is a short exact $\bT[U']$-equivariant sequence
    \[
    0 \to \SBO_\Lambda^\bullet \to \BO_\Lambda^\bullet \to (\Hom_\Lambda(J, \Lambda))^0 \to 0 
    \]
    of perfect $\Lambda$-complexes, where $J$ is the kernel of the natural surjection $\bT_\Lambda^\tord \rsurj \bT_\Lambda^{\tord,\circ}$.
\end{lem}
\begin{proof}
    This follows from the fact that the inclusion $S_\Lambda^\tord \rinj M_\Lambda^\tord$ is dual to the named surjection of flat $\Lambda$-algebras, as mentioned in Definition \ref{defn: Lambda-adic tord forms}. 
\end{proof}

\begin{lem}
\label{lem: Euler characteristic}
Let $k \in \Z_{\geq 3}$. The Euler characteristics of $\SBO_\Lambda^\bullet$ and $\SBO_k^\bullet$ are $0$.
\end{lem}

\begin{proof}
    The vanishing of the Euler characteristics is visible in the short exact sequence of Theorem \ref{thm: lattice main}. 
\end{proof}

 We will sometimes think of the natural $\Lambda$-algebra homomorphism
 \[
\psi_\Lambda : \bT_\Lambda^\crit \to \bT_\Lambda^\aord \times \bT_\Lambda^\tord 
\]
as a complex with source in grade 0 and target in grade 1. We write $(-)^\vee$ for duality of perfect complexes over $\Lambda$; in particular, the dual $(M^i)^\vee := \Hom_\Lambda(M_i, \Lambda)$ of the degree $i$ piece occurs in degree $-i$. The ``$a_1$-duality'' pairing we render as ``$\lr{}$'' is realizable on $\BO_\Lambda^\bullet$ as follows. 

\begin{lem}
\label{lem: BO psi duality}
The $\Lambda$-linear dual $(\BO_\Lambda^\bullet)^\vee[-1]$ of the complex $\BO_\Lambda^\bullet$ is canonically realized by $\psi_\Lambda$. Likewise, the dual of $\SBO_\Lambda^\bullet$ is realized by the natural $\Lambda$-algebra homomomorphism
\[
\psi_\Lambda^\circ : \bT_\Lambda^\crit \to \bT_\Lambda^\aord \times \bT_\Lambda^{\tord,\circ}. 
\]
\end{lem}

\begin{lem}
    \label{lem: PD of H}
    $H^0(\BO_\Lambda^\bullet)$ is $\Lambda$-flat. $H^1(\BO_\Lambda^\bullet)$ has projective dimension at most $1$; what is equivalent, $H^1(\BO_\Lambda^\bullet)$ has no finite cardinality submodule. 
\end{lem}
\begin{proof}
These are standard facts about perfect complexes. In particular, $H^0(\BO_\Lambda^\bullet)$ is flat because it is the kernel of a presentation of a projective dimension 1 module by a projective module. For the equivalence, see Lemma \ref{lem: Lambda duality}. 
\end{proof}

The perfect duality between $\BO_\Lambda^\bullet$ and $\psi_\Lambda$ produces the following structures. 
\begin{lem}
    \label{lem: a1 duality for BO}
    There are canonical isomorphisms
    \begin{gather*}
        \ker \psi_\Lambda \cong \Hom_\Lambda(H^1(\BO_\Lambda^\bullet),\Lambda), \quad 
        H^0(\BO_\Lambda^\bullet) \cong \Hom_\Lambda(\coker \psi_\Lambda, \Lambda)
    \end{gather*}
    and canonical short exact sequences
    \begin{gather*}
    0 \to \Ext^1_\Lambda(\coker \psi_\Lambda,\Lambda) \to H^1(\BO_\Lambda^\bullet) \to \Hom_\Lambda(\ker \psi_\Lambda,\Lambda) \to 0 \\
    0 \to \Ext^1_\Lambda(H^1(\BO_\Lambda^\bullet), \Lambda) \to \coker \psi_\Lambda \to \Hom_\Lambda(H^0(\BO_\Lambda^\bullet),\Lambda) \to 0.
    \end{gather*}
    Likewise, there exist the same canonical data for $(\SBO_\Lambda^\bullet, \psi_\Lambda^\circ)$ in place of $(\BO_\Lambda^\bullet,\psi_\Lambda)$. All of the maps here are $\bT[U']$-equivariant.
\end{lem}

\begin{proof}
These sequences are exactly the output of the standard spectral sequence relating cohomology groups under the duality of perfect complexes $(\BO_\Lambda^\bullet)^\vee[-1] \cong \psi_\Lambda$ of Lemma \ref{lem: BO psi duality}, see e.g.\ \cite[Exer.\ 10.8.3]{weibel1994}
\end{proof}

Here is some useful information about $E^i(X) := \Ext^i_\Lambda(X, \Lambda)$ for finitely generated $\Lambda$-modules $X$, following \cite[\S4]{wake1}. Let $T_0(X)$ denote the maximum finite (cardinality) submodule and let $T_1(X)$ denote the maximum $\Lambda$-torsion submodule. 
\begin{lem}
\label{lem: Lambda duality}
Let $X$ be a finitely generated $\Lambda$-module. 
\begin{enumerate}
\item $E^1(X)$ is torsion.
\item If $X$ is torsion, then $T_0(E^1(X)) = 0$.
\item $E^1(X) \cong E^1(X/T_0(X))$, and $X/T_0(X)$ is flat if and only if $E^1(X) = 0$. 
\item $E^2(X) \cong T_0(X)^*$, where $(-)^*$ denotes Pontryagin duality. 
\item $E^1(E^1(X)) \cong T_1(X)/T_0(X)$.
\item Assume $X$ is torsion. There is a $\Lambda$-linear non-degenerate pairing
\[
X/T_0(X) \times E^1(X) \to Q(\Lambda)/\Lambda. 
\]
\end{enumerate}
\end{lem}

Note that (6), with $\Z$ substituting for $\Lambda$, is simply Pontryagin duality. 
\begin{proof}
For (1)--(5), see \cite[Prop.\ 4.2]{wake1}, which is based on \cite[\S5.4]{NSW2008}. For (6), there is no harm in assuming $T_0(X) = 0$ in addition. Apply the first Grothendieck spectral sequence in \cite{ischebeck1969}, that is, 
\[
\Ext_\Lambda^i(\Ext_\Lambda^j(M,\Lambda),N) \Rightarrow \Tor_{j-i}^\Lambda(N,M)
\]
in the case $M = X$, $N = f^{-1}\Lambda/\Lambda$ where $f \cdot X = 0$. The only non-zero term among those influencing $\Tor_1^\Lambda(\Lambda/(f),X) \cong X$ is $\Hom_\Lambda(E^1(X),f^{-1}\Lambda/\Lambda)$. Thus we obtain a $\Lambda$-linear pairing $X \times E^1(X) \to f^{-1}\Lambda/\Lambda$. Now making $M = E^1(X)$, and applying the natural isomorphism $E^1(E^1(X)) \cong X$ of (5), we see that the pairing is perfect. 
\end{proof}

\subsection{Serre self-duality of the cuspidal bi-ordinary complex}

By applying Corollary \ref{cor: main Lambda SES}, we obtain a $\bT[U']$-equivariant quasi-isomorphism of $\Lambda$-\emph{perfect} complexes 
\begin{gather}
\label{eq: BO replace}
\BO_\Lambda^\bullet \cong [M_\Lambda^\tord \ \mathrel{\mathop{\lra}^{\pi_\Lambda \circ \zeta_\Lambda}} \ \cH^{1,\ord}_\Lambda] \\
\label{eq: SBO replace}
\SBO_\Lambda^\bullet \cong [S_\Lambda^\tord \ \mathrel{\mathop{\lra}^{\pi_\Lambda \circ \zeta_\Lambda}} \ \cH^{1,\ord}_\Lambda]. 
\end{gather}

\begin{defn}[Serre self-duality of the bi-ordinary complex]
    \label{defn: SD for SBO} 
    Let 
    \[
    \lr{}_\SBO : \SBO_\Lambda^\bullet \times \SBO_\Lambda^\bullet \to \Lambda[-1], \quad \text{ that is}, \quad
    \lr{}_\SBO \in \Hom_\Lambda(\SBO_\Lambda^\bullet, \SBO_\Lambda^\bullet[-1]),
    \]
    arising from the $\Lambda$-perfect (alternate) Serre duality pairing $\lr{}'_\mathrm{SD} : S_\Lambda^\tord \times \cH^{1,\ord}_\Lambda \to \Lambda$ of Corollary \ref{cor: tord serre duality}. There is also a variant for $\BO_\Lambda^\bullet$, using the pairing 
        $\lr{}^\prime_\mathrm{SD} : M_\Lambda^\tord \times \cH^{1,\ord}_\Lambda(-C) \to \Lambda$ of Corollary \ref{cor: tord serre duality} along with the projection $\cH^{1,\ord}_\Lambda(-C) \rsurj \cH^{1,\ord}_\Lambda$. We get $\lr{}_\BO \in \Hom_\Lambda(\BO_\Lambda^\bullet, \BO_\Lambda^\bullet[-1])$.
\end{defn}

Because $\lr{}_\SBO$ is an $\bT[U']$-equivariant isomorphism, we get the following standard consequence of the spectral sequence expressing this duality on cohomology, arriving at self-duality for $\SBO_\Lambda^\bullet$. There are also variants for $\BO_\Lambda^\bullet$ using Lemma \ref{lem: Eis difference}. 

\begin{thm}[Serre self-duality for $H^*(\SBO_\Lambda^\bullet)$]
    \label{thm: SD for BO}
    There is a canonical isomorphism
    \[
    H^0(\SBO_\Lambda^\bullet) \cong (H^1(\SBO_\Lambda^\bullet))^\vee
    \]
    and a canonical short exact sequence
    \begin{equation}
    \label{eq: H1 SES}
    0 \to \Ext_\Lambda^1(H^1(\SBO_\Lambda^\bullet),\Lambda) \to H^1(\SBO_\Lambda^\bullet) \to H^0(\SBO_\Lambda^\bullet)^\vee \to 0  
    \end{equation}
    under which $\Ext_\Lambda^1(H^1(\SBO_\Lambda^\bullet), \Lambda)$ maps isomorphically onto the maximal torsion submodule of $H^1(\SBO_\Lambda^\bullet)$. All of these maps are $\bT[U']$-equivariant. 
\end{thm}

The final short exact sequence actually has two duality statements within it. For notational convenience, we use the notation $E^i(-)$ and $T_i(-)$ of Lemma \ref{lem: Lambda duality}. 
\begin{cor}
    \label{cor: SD on HSBO}
    There are perfect $\bT[U']$-compatible duality pairings
    \begin{gather*}
        \frac{H^1(\SBO_\Lambda^\bullet)}{T_1(H^1(\SBO_\Lambda^\bullet))} \times H^0(\SBO_\Lambda^\bullet) \to \Lambda    \\
        T_1(H^1(\SBO_\Lambda^\bullet)) \times T_1(H^1(\SBO_\Lambda^\bullet)) \to Q(\Lambda)/\Lambda. 
    \end{gather*}
\end{cor}
\begin{proof}
    The perfectness of the first pairing is visible in Theorem \ref{thm: SD for BO}. The perfectness of the second pairing follows from the isomorphism $E^1(H^1(\SBO_\Lambda^\bullet)) \isoto T_1(H^1(\SBO_\Lambda^\bullet))$ stated in Theorem \ref{thm: SD for BO} along with Lemma \ref{lem: Lambda duality}. 
\end{proof}

\begin{rem}
    There are similar duality theorems upon specialization to a weight $k \in \Z$, which come from $\lr{}_\SBO \otimes_{\Lambda,\phi_k} \Z_p \in \Hom_{\Z_p}(\SBO_k^\bullet, \SBO_k^\bullet[-1])$.  
\end{rem}

Theorem \ref{thm: SD for BO} can be used to rule out the presence of $\Lambda$-torsion-free parts of $H^*(\BO_\Lambda)$ that are not $\Lambda$-flat; this does not follow directly from Lemma \ref{lem: PD of H} alone. 
\begin{cor}
    \label{cor: flat plus torsion}
    As a $\Lambda$-module, $H^1(\SBO_\Lambda^\bullet)$ is expressable as direct sum of a $\Lambda$-flat submodule projecting canonically isomorphically onto $H^0(\SBO_\Lambda^\bullet)^\vee$ and its torsion submodule. Moreover, its maximal torsion submodule has projective dimension 1 if it is non-zero. 
\end{cor}

\begin{proof}
    This follows from the short exact sequence in Theorem \ref{thm: SD for BO} because the quotient is a $\Lambda$-projective. The final claim follows from Lemma \ref{lem: PD of H}.   
\end{proof}

\begin{rem}
    This $\Lambda$-torsion part of $H^1(\BO_\Lambda^\bullet)$ is a mysterious aspect of the situation. The $H^1(\BO_\Lambda^\bullet)$ is designed to be the maximal quotient of $\Lambda$-adic coherent cohomology $\cH^{1,\ord}_\Lambda$ that supports a $p$-split Galois representation, and its torsion part consists of those $p$-adic modular forms that do not interpolate into a Hida family that supports a $p$-split Galois representation. Of course, it is natural to ask for an analytic expression for its characteristic ideal. 
\end{rem}

\begin{proof}[{Proof of Theorem \ref{thm: SD for BO}}]
The claimed isomorphism and short exact sequence follow just as in the proof of Lemma \ref{lem: a1 duality for BO} from the isomorphism $\SBO_\Lambda^\bullet \cong (\SBO_\Lambda^\bullet)^\vee[-1]$ coming from $\lr{}_\BO$. Because $\lr{}_\BO$ is Hecke-equivariant, the maps above are as well. The remaining parts of the theorem arise from $\mathrm{Eis}_\Lambda^\tord$ (in degree 0) being the kernel of $\BO_\Lambda^\bullet \rsurj \SBO_\Lambda^\bullet$. 
\end{proof}

By the duality between $\SBO_\Lambda^\bullet$ and $\psi_\Lambda^\circ$) of Lemma \ref{lem: a1 duality for BO}, there are also self-dualities for the Hecke algebras that are dual to those for $\BO_\Lambda^\bullet$ in Theorem \ref{thm: SD for BO}.

\subsection{The bi-ordinary Hecke algebras}

Now we define Hecke algebras acting on the bi-ordinary complexes. We let $k \in \Z$ be arbitrary, while there is only a control theorem for $k \in \Z_{\geq 3}$. 

\begin{defn}
Let $k \in \Z$. Let $\bT^\star_k$, $\bT^\star_\Lambda$ be the Hecke algebras (that is, the image algebras of $\bT[U']_\Lambda$) acting on $H^*(\star_k^\bullet)$ and $H^*(\star_\Lambda^\bullet)$, respectively, for $\star = \BO, \SBO$. We will refer to the Hecke algebras labeled by 
\begin{itemize}
    \item $\BO$ as \emph{bi-ordinary} Hecke algebras
    \item $\SBO$ as \emph{cuspidal bi-ordinary} Hecke algebras.
\end{itemize}
\end{defn}

It is also possible to define actions in the derived category, but we defer this to future work. We discuss the ``$\Lambda$-flat part'' of a derived, degree-shifting action in \S\ref{subsec: flat shifting}.

Here are the expressions of duality between the Hecke algebra and its module of forms in the bi-ordinary case. 

\begin{prop}
\label{prop: BO T to forms duality}
The natural projection $\bT^\SBO_\Lambda \isoto \bT[U'](H^1(\SBO_\Lambda^\bullet))$ is an isomorphism. 
There is a perfect $\Lambda$-linear duality 
\begin{equation}
    \label{eq: H0SBO duality}
    \bT[U']_\Lambda(H^0(\SBO_\Lambda^\bullet)) \times H^0(\SBO_\Lambda^\bullet) \to \Lambda;     
\end{equation}
in particular, $\bT[U']_\Lambda(H^0(\SBO_\Lambda^\bullet))$ is $\Lambda$-flat. On the other hand, the projection 
\[
\bT_\Lambda^\SBO \rsurj \bT[U']_\Lambda(H^0(\SBO_\Lambda^\bullet))
\]
has $\Lambda$-torsion kernel that is canonically isomorphic to $E^1(H^1(\SBO_\Lambda^\bullet))$. \end{prop}

For clarity in the proof of Proposition \ref{prop: BO T to forms duality}, we provide the following elementary lemma. 
\begin{lem}
    \label{lem: cokernel of map to two quotients}
    Let $A$ be a commutative ring equipped with ring surjections $A \rsurj B_i$ for $i = 1,2$, from which arises an $A$-module map $\phi : A \buildrel{+}\over\to B_1 \oplus B_2$. There are natural $A$-algebra structures on $\coker \phi$, each of which is the additive inverse of the other, characterized by making the natural projection map $B_i \rsurj \coker \phi$ multiplicative. Such a choice of projection map factors through an $A$-algebra isomorphism $B_1 \otimes_A B_2 \isoto \coker \phi$. 
\end{lem}

\begin{proof}[{Proof of Proposition \ref{prop: BO T to forms duality}}]
    The first claim follows from Theorem \ref{thm: SD for BO}, namely, its result $H^0(\SBO_\Lambda^\bullet) \cong \Hom_\Lambda(H^1(\SBO_\Lambda^\bullet),\Lambda)$.

    We claim that the tensor product $\bT_\Lambda^{\tord,\circ} \otimes_{\bT_\Lambda^\crit} \bT_\Lambda^\aord$ of the two surjections $\bT_\Lambda^\crit \rsurj \bT_\Lambda^{\tord,\circ}$ and $\bT_\Lambda^\crit \rsurj \bT_\Lambda^\aord$ factors the surjection $\bT_\Lambda^\crit \rsurj \bT_\Lambda^\SBO$. The claim follows from considering that quotients of $M_\Lambda^\crit$ by $M_\Lambda^\aord$ (see Theorem \ref{thm: main extension}) and by $M_\Lambda^\tord$ (see Proposition \ref{prop: secondary SES Lambda} and Theorem \ref{thm: secondary Hecke SES Lambda}) have Hecke action given by $\bT_\Lambda^{\tord,\circ}$ and by $\bT_\Lambda^\aord$, respectively. Therefore the $\bT_\Lambda^\crit$ action on $H^1(\SBO_\Lambda^\bullet)$ factors through this tensor product, as desired. Hence we have a natural surjection 
    \begin{equation}
        \label{eq: tensor to SBO}
        \bT_\Lambda^{\tord,\circ} \otimes_{\bT_\Lambda^\crit} \bT_\Lambda^\aord \rsurj \bT_\Lambda^\SBO. 
    \end{equation}
    
    Applying Lemma \ref{lem: cokernel of map to two quotients} to the case $A = \bT_\Lambda^\crit$, $B_1 = \bT_\Lambda^{\tord,\circ}$, $B_2 = \bT_\Lambda^\aord$, $\phi = \psi_\Lambda^\circ$, we arbitrarily select the $\bT_\Lambda^\crit$-algebra structure on $\coker \psi_\Lambda^\circ$ that makes $\bT_\Lambda^{\tord,\circ} \rsurj \coker \psi_\Lambda^\circ$ multiplicative and identifies it with $\bT_\Lambda^{\tord,\circ} \otimes_{\bT_\Lambda^\crit} \bT_\Lambda^\aord$. For brevity, in this proof, we refer to these canonically isomorphic $\Lambda$-algebras by $A$, that is, $A := \bT_\Lambda^{\tord,\circ} \otimes_{\bT_\Lambda^\crit} \bT_\Lambda^\aord \cong \coker \psi_\Lambda^\circ$. 
    
    Dual to the arbitrary selection above, let us select (between the two options $S_\Lambda^\tord$ and $M_\Lambda^\aord$) $S_\Lambda^\tord$ to receive an injection from $H^0(\SBO_\Lambda^\bullet)$ via the composite 
    \[
    H^0(\SBO_\Lambda^\bullet) \rinj \SBO_\Lambda^0 \buildrel{\mathrm{pr}_\tord}\over\rsurj S_\Lambda^\tord,
    \]
    where the composite is injective because the differential of $\SBO_\Lambda^\bullet$ is an addition map on a direct sum of submodules with target being the ambient module. The surjection $\bT_\Lambda^{\tord,\circ} \rsurj \bT[U']_\Lambda(H^0(\SBO_\Lambda^\bullet))$ is dual to this injection under the usual $a_1$ pairing. 
    
    To prove the second claim of the proposition, i.e.\ that under $\bT_\Lambda^{\tord,\circ} \rsurj \bT[U']_\Lambda(H^0(\SBO_\Lambda^\bullet))$ one gets a perfect $a_1$-duality between $H^0(\SBO_\Lambda^\bullet)$ and its Hecke algebra, it suffices to note that $H^0(\SBO_\Lambda^\bullet)$ is $\Lambda$-flat (Lemma \ref{lem: PD of H}) and claim that the cokernel $C$ of $H^0(\SBO_\Lambda^\bullet) \rinj S_\Lambda^\tord$ is $\Lambda$-flat. Indeed, if this cokernel $C$ were not flat, then its $\Lambda$-projective dimension is $1$ and the quasi-isomorphic presentation of $\SBO_\Lambda^\bullet$ in \eqref{eq: SBO replace} could be rewritten quasi-isomorphically as $[C \rinj \cH_\Lambda^{1,\ord}]$, which is absurd because $H^1(\SBO_\Lambda^\bullet)$ has $\Lambda$-projective dimension $1$.

    It remains to prove the final claim of the proposition. Applying the definition of $A$ along with the isomorphism $\Hom_\Lambda(H^0(\SBO_\Lambda^\bullet), \Lambda) \cong \bT[U'](H^0(\SBO_\Lambda^\bullet))$, the $\SBO$ version of the final short exact sequence in Lemma \ref{lem: a1 duality for BO}, takes the form 
    \begin{equation}
    \label{eq: SBO surjection}
    0 \to E^1(H^1(\SBO_\Lambda^\bullet)) \to A \to \bT[U'](H^0(\SBO_\Lambda^\bullet)) \to 0    
    \end{equation}
    where the surjection is of $\bT_\Lambda^\crit$-algebras. In particular, by Lemma \ref{lem: Lambda duality}, we see that we have a canonical isomorphism $E^1(H^1(\SBO_\Lambda^\bullet)) \cong T_1(A) \subset A$, and $T_1(A) \subset A$ is an ideal. Therefore, to prove the final claim of the proposition, it suffices to show that $T_1(A)$ acts faithfully on $H^1(\SBO_\Lambda^\bullet)$. 

    Let $x \in T_1(A)$. Let $0 \neq \lambda \in \Lambda$ such that $\lambda \cdot x = 0$. Let $\tilde x \in \bT_\Lambda^{\tord,\circ} \times \bT_\Lambda^\aord$ be a lift of $x$ (over the projection to $\coker \psi_\Lambda^\circ = A$. Likewise, let $y \in \bT_\Lambda^\crit$ be a lift over $\psi_\Lambda^\circ$ of $\lambda \cdot \tilde x$, where $\lambda \tilde x \in \psi_\Lambda^\circ(\bT_\Lambda^\crit)$ because $\lambda \cdot x = 0$. 

    Interpreting $\tilde x$ and $y$ in terms of the duality between the respective Hecke algebras and their modules, we arrive at a commutative diagram
    \begin{equation}
    \label{eq: chase action}
    \xymatrix{
    S_\Lambda^\tord \oplus M_\Lambda^\aord \ar[r]^+ \ar[d]^{\tilde x} & M_\Lambda^\crit \ar[r] \ar[d]^{\frac{1}{\lambda}y} & H^1(\SBO_\Lambda^\bullet) \ar[d] \ar[r] & \frac{H^1(\SBO_\Lambda^\bullet)}{(H^0(\SBO_\Lambda^\bullet))^\vee} \ar[dl]_{\alpha(x)}
    \\
    \Lambda \ar[r] & \frac{1}{\lambda}\Lambda \ar[r] & \frac{\frac{1}{\lambda}\Lambda}{\Lambda}  & 
    }
    \end{equation}
    where the lower row's maps are the natural ones belonging to a short exact sequence. We claim that there exists a unique $\alpha(x)$ compatible with the diagram, in particular claiming that this map is independent of the choice of $\tilde x$ and of $y$. This follows quickly from the observation that $T_1(A)$ has been shown to be the annihilator of $H^0(\SBO_\Lambda^\bullet)$ and that $\lambda \cdot x = 0$. This construction of $\alpha(x)$ amounts to the standard diagram chase establishing the canonical isomorphism 
    \[
    E^1(A) = E^1(\coker \psi_\Lambda^\circ) \cong \Hom_\Lambda\left(\frac{H^1(\SBO_\Lambda^\bullet)}{(H^0(\SBO_\Lambda^\bullet))^\vee}, \frac{Q(\Lambda)}{\Lambda}\right) 
    \]
    of Lemma \ref{lem: Lambda duality}(6). Consequently, $x=0$ if and only if $\alpha(x) = 0$. 

    The implication $x \neq 0 \implies \alpha(x) \neq 0$ gives us the input we need to finish the proof that $E^1(A)$ acts faithfully on $H^1(\SBO_\Lambda^\bullet)$, as follows. Presume that $\tilde x$ kills $H^1(\SBO_\Lambda^\bullet)$, i.e.\ $\tilde x'$ lies in the kernel of the natural surjection $\bT_\Lambda^\crit \rsurj H^1(\SBO_\Lambda^\bullet)$. Next recall that the downward maps, which are given by $a_1$-duality, arise from acting by $\tilde x$ (resp.\ $y$) and then evaluating the $a_1$ map. Therefore, for all $f \in M_\Lambda^\crit$ and for all choices of lifts $\tilde x, y$, we have $y \cdot f \in (S_\Lambda^\tord + M_\Lambda^\aord)$. Returning to the interpretation of $\tilde x$ and $y$ as downward arrows in \eqref{eq: chase action}, we deduce that $\frac{1}{\lambda}y(f) \in \mathrm{image}(\tilde x)$, which implies that $\alpha(x)(\bar f) = 0 \in \frac{1}{\lambda}\Lambda/\Lambda$. Thus we arrive at the desired conclusion, that $A$ acts faithfully on $H^1(\SBO_\Lambda^\bullet)$. 
\end{proof}

The faithfulness claim at the end of the proof implies that \eqref{eq: tensor to SBO} is an isomorphism, which we record in this corollary. 
\begin{cor}
    \label{cor: BO is to and ao}
    The surjections of $\bT_\Lambda^\crit$-algebras $\bT_\Lambda^{\tord,\circ} \rsurj \bT_\Lambda^\SBO \lsurj \bT_\Lambda^\aord$ produce an isomorphism 
    $\bT_\Lambda^{\tord,\circ} \otimes_{\bT_\Lambda^\crit} \bT_\Lambda^\aord \isoto \bT_\Lambda^\SBO$. 
\end{cor}

\subsection{A $\Lambda$-flat degree-shifting Hecke action}
\label{subsec: flat shifting}

We describe a degree-shifting manifestation of a derived action of $\bT[U']$ on $\SBO_\Lambda^\bullet$. In the broader context of derived actions, the idea is that the degree $-1$ part of the derived endomorphisms of $\SBO_\Lambda^\bullet$ is $\Hom_\Lambda(H^1(\SBO_\Lambda), H^0(\SBO_\Lambda))$. We show that those endomorphisms that might possibly come from a $\bT[U]$-action actually come from $\ker \psi_\Lambda \cong \ker \psi_\Lambda^\circ$.

\begin{thm}
    \label{thm: flat derived action}
    There is a well-defined $\bT[U']_\Lambda$-equivariant isomorphism 
    \[
    \delta : \ker\psi_\Lambda \isoto \Hom_{\bT[U']_\Lambda}(H^1(\SBO_\Lambda^\bullet), H^0(\SBO_\Lambda^\bullet))
    \]
    given by sending the pair $(T,x)$, $T \in \ker \psi_\Lambda$ and $x \in H^1(\SBO_\Lambda^\bullet)$, the $q$-series realization of $T \cdot x \in d(\SBO_\Lambda^0)$. As $\bT[U'](\SBO_\Lambda^\bullet)$-modules, the source and target of $\delta$ are isomorphic to its dualizing module.  
\end{thm}

To make sense of the theorem, we recall that the cuspidal bi-ordinary complex simply arises from addition of submodules of $q \Lambda\lb q \rb$, $\SBO_\Lambda^0 = M_\Lambda^{\tord,\circ} \oplus M_\Lambda^\aord \to M_\Lambda^\crit = \SBO_\Lambda^1$. We write $d$ for the differential, the addition map. This also justifies a fact we will use in the proof, which is that $T \in \bT_\Lambda^\crit$ lies in $\ker \psi_\Lambda$ if and only if $T$ annihilates $\SBO_\Lambda^0$ if and only if it annihilates $d(\SBO_\Lambda^0)$.

\begin{proof}
    Due to this fact, and due to the implication of Corollary \ref{cor: BO is to and ao} that $T$ annihilates $H^1(\SBO_\Lambda^\bullet)$, the action of $T \in \ker_\Lambda$ on $\SBO_\Lambda^1$ factors through the quotient $\SBO_\Lambda^1 \rsurj H^1(\SBO_\Lambda^\bullet)$ and is valued in $d(\SBO_\Lambda^\bullet)$. That is, $T \in \Hom_\Lambda(H^1(\SBO_\Lambda^\bullet), d(\SBO_\Lambda^\bullet))$. Because $d(\SBO_\Lambda^\bullet)$ is $\Lambda$-flat, this map factors through the quotient of its source by $T_1(H^1(\SBO_\Lambda^\bullet)$; moreover, $T$-acts $\bT[U']$-equivariantly. Thus we may now consider 
    \[
    T \in \Hom_{\bT[U']_\Lambda}(H^1(\SBO_\Lambda^\bullet)/T_1(H^1(\SBO_\Lambda^\bullet)), d(\SBO_\Lambda^0)).
    \]

    Consider the case where $\bar x \in H^1(\SBO_\Lambda^\bullet)/T_1(H^1(\SBO_\Lambda^\bullet))$ is $\bT[U']$-eigen. Then, because of the aforementioned Hecke-equivariance, the $T_n$-eigenvalues of $\bar x$ determine the $T_n$-eigenvalues of $T \cdot \bar x$. Therefore $T \cdot \bar x$ is a $\Lambda$-multiple of the normalized $q$-series in $d(\BO_\Lambda^0) \subset q\Lambda \lb q \rb$ specified in the usual way. By Corollary \ref{cor: SD on HSBO}, this $\bT[U']$-eigensystem factors through $\bT[U'](H^0(\SBO_\Lambda^\bullet))$. Therefore, thinking of $d(\SBO_\Lambda^0)$ and $H^0(\SBO_\Lambda^\bullet)$ as submodules of $q\Lambda \lb q\rb$, $T \cdot \bar x$ lies in the intersection of $H^0(\SBO_\Lambda^\bullet)$ and $d(\SBO_\Lambda^0)$. 

    Because $H^0(\SBO_\Lambda^\bullet)$ admits a $\bT[U']_\Lambda$-eigenbasis after a finitely generated integral extension of $\Lambda$, the argument above along with a descent argument (back to $\Lambda$) shows that $T$ produces a $\bT[U']_\Lambda$-equivariant map
    \[
    \delta(T) : H^1(\SBO_\Lambda^\bullet) \to H^0(\SBO_\Lambda^\bullet).
    \]
    This map is injective by the first statement in Lemma \ref{lem: a1 duality for BO}, which also implies that the image of $\delta$ on $\ker \psi_\Lambda \times H^1(\SBO_\Lambda^\bullet)$ in $d(\SBO_\Lambda^\bullet)$ is $\Lambda$-saturated. Since $\ker \psi_\Lambda$, $H^1(\SBO_\Lambda^\bullet)$, and $H^0(\SBO_\Lambda^\bullet)$ all have the same $\Lambda$-rank, the cokernel of $\delta$ is finite cardinality. 

    Due to the $\bT[U']$-compatible dualities of Proposition \ref{prop: BO T to forms duality} and Corollary \ref{cor: SD on HSBO}, there exist $\bT[U'](H^0(\SBO_\Lambda^\bullet))$-module isomorphisms
    \[
    \frac{H^1(\SBO_\Lambda^\bullet)}{T_1(H^1(\SBO_\Lambda^\bullet))} \simeq \bT[U'](H^0(\SBO_\Lambda^\bullet)), \quad H^0(\SBO_\Lambda^\bullet) \simeq \bT[U'](H^0(\SBO_\Lambda^\bullet))^\vee. 
    \]
    It is well known that $\bT[U'](H^0(\SBO_\Lambda^\bullet))^\vee$ is a dualizing module for $\bT[U'](H^0(\SBO_\Lambda^\bullet))$: one needs the fact that $\bT[U'](H^0(\SBO_\Lambda^\bullet))$ is Cohen-Macaulay because it is finite flat over a regular local ring, along with \cite[Thm.\ 3.3.7, pg.\ 111]{BH1993}. Therefore the source and target of $\delta$ are isomorphic to this dualizing module. Since for a commutative ring $R$, the $R$-linear endomorphisms of a dualizing module for $R$ are exactly the $R$-scalars, and $\bT[U'](H^0(\SBO_\Lambda^\bullet))$ is $\Lambda$-flat, the only possible finite cardinality cokernel of $\delta$ is the trivial cokernel. 
\end{proof}

\begin{rem}
    Here is a more homological interpretation of $\delta(T)$. Because $M_\Lambda^\crit$ is $\Lambda$-flat, there is a $\Lambda$-linear $h : \SBO_\Lambda^1 \to \SBO_\Lambda^0$ such that $d \circ h = T$ in $\End_\Lambda(\SBO_\Lambda^1)$. Then one can check that $h\vert_{d(\SBO_\Lambda^0)}$ is valued in $H^0(\SBO_\Lambda^\bullet)$ and that our action map $\delta(T) \in \Hom_\Lambda(H^1(\SBO_\Lambda^\bullet), H^0(\SBO_\Lambda^\bullet))$ is equal to $h \circ T$; in particular, $h \circ T$ is independent of the choice of $h$. One can view $h$ as a chain homotopy expressing that $T$, as an element of $\End_\Lambda(\SBO_\Lambda^\bullet)$, is null-homotopic. 
\end{rem}

\section{Galois representations}
\label{sec: galois}

In this section, we show that the 2-dimensional Galois representations $\rho_f : G_\Q \to \GL_2(\oQ_p)$ associated to classical bi-ordinary eigenforms $f$ are reducible and decomposable upon restriction to a decomposition group $G_p \subset G_\Q$ at $p$, and in a robust way with respect to interpolation. We also prove a partial converse, recovering in Theorem \ref{thm: BE alternate} part of this result of Breuil--Emerton. 
\begin{thm}[{Breuil--Emerton \cite[Thm.\ 1.1.3]{BE2010}}]
    \label{thm: BE}
    Let $g$ be a classical $U_p$-critical cuspidal eigenform of level $\Gamma_1(Np^r)$ and weight $k \in \Z_{\geq 2}$. Then $\rho_g \vert_{G_p}$ is reducible decomposable if and only if there exists a $U$-ordinary overconvergent form $h$ of weight $2-k$ such that $\theta^{k-1}(h) = g$.  
\end{thm}


The main input we will apply is the strong characterization of ordinary $\Lambda$-adic \'etale cohomology of tame level $\Gamma_1(N)$ established by Fukaya--Kato \cite{FK2012}. 

\subsection{Conventions for Galois representations}

We add some conventions to those of \S \ref{subsec: ANT context}. For $a \in A^\times$, $\nu(\alpha) : G_p \to A^\times$ denotes the unique unramified character sending arithmetic Frobenius $\Frob_p$ to $\alpha$.  Given a Galois representation $V$ over a $p$-adic ring, we let $V(i)$ denote its $i$th Tate twist. We let $\kappa$ denote the $p$-adic cyclotomic character.  
 We implicitly always refer to  continuous Galois representations and pseudorepresentations without mentioning continuity. 

For conventions relating modular eigenforms $f$ to Galois representations $\rho_f$, we refer to \S \ref{subsec: Gal rep discussion}. 

\subsection{$\Lambda$-adic \'etale cohomology}

Now we describe some results from \cite{FK2012}. Let $H$ denote the $\bT_\Lambda^\ord[G_\Q]$-module of $\Lambda$-adic \'etale cohomology of $X$ as defined in \cite[\S1.5.1]{FK2012}.  In \cite{FK2012} the additive model of the modular curve and dual Hecke operators is used, but we apply the $w$-type isomorphisms of \cite[\S1.4]{FK2012} to translate these to our setting of the multiplicative model and standard Hecke operators (see also the overview of translation between models and standard/dual operators in \cite[\S1.7.16]{FK2012}). 

Because they use the Igusa tower and weight 2 (trivial coefficients) \'etale cohomology, they get the following normalization. 
\begin{thm}[{\cite[\S\S 1.2.9, 1.7.13]{FK2012}}]
    \label{thm: FK Lamda-adic etale}
    For a prime $\ell \nmid Np$, the characteristic polynomial of $\Frob_\ell$ acting on $H$ is $X^2 - \ell^{-1}T_\ell X + \ell^{-2}[\ell] \lr{\ell}_N$. There is a short exact sequence of $\Lambda$-flat $\bT_\Lambda^{\ord,\circ}[G_p]$-modules
    \begin{equation}
        \label{eq: H reducibility}
        0 \to H_\mathrm{sub} \to H \to H_\mathrm{quo} \to 0,     
    \end{equation}
    and there exist $\bT_\Lambda^{\ord,\circ}$-module isomorphisms $H_\mathrm{sub} \simeq \bT_\Lambda^{\ord,\circ}$ and $H_\mathrm{quo} \simeq \Hom_\Lambda(\bT_\Lambda^{\ord,\circ}, \Lambda) \cong S_\Lambda^\ord$. The action of $G_p$ on $H_\mathrm{sub}$ and $H_\mathrm{quo}$ is described by $\bT_\Lambda^\ord$-valued characters 
    \[
    \chi_\mathrm{sub} = \dia \nu(U_p^{-1}\lr{p}_N)(-1), \qquad \chi_\mathrm{quo} = \nu(U_p)(-1). 
    \]
\end{thm}

Since $H$ is set up to interpolate the Galois representations associated to $U_p$-ordinary forms of weight $2$ and levels $\Gamma_1(Np^r)$, we need to clarify its realization with varying weights $k \in \Z_{\geq 3}$. This is discussed clearly in \cite[\S1.5]{FK2012}. It will suffice to describe the necessary twist when specializing $H$ along a cuspidal eigenform $f \in S_{k,\oQ_p}^\ord$ for $k \in \Z_{\geq 2}$ with $U_p$-eigenvalue $\alpha$ and character $\chi_f : (\Z/N\Z)^\times \to \oQ_p^\times$. We use $f$ to label the map $\bT_\Lambda^{\ord,\circ} \to \oQ_p$ corresponding to $f$. 

\begin{lem}[{\cite[Prop.\ 1.5.8]{FK2012}}]
    Let $k, f$ as above. Then there is a $\bT[U][G_\Q]$-isomorphism $H \otimes_{\Lambda, f} \oQ_p(2-k) \simeq H^1_\et(X_{\oQ}, \cF_k)_f$, these are 2-dimensional $\oQ_p$-vector spaces with the characteristic polynomial of the $\Frob_\ell$-action ($\ell \nmid Np$) being $X^2 - \ell^{1-k}T_\ell X + \chi(\ell) \ell^{1-k}$, and they have a $\oQ_p$-basis with $G_p$-action given by 
    \[
    \ttmat{\nu(\alpha^{-1}\chi_f(p))}{*}{0}{\nu(\alpha)(1-k)}. 
    \]
\end{lem}

Therefore we want the following twist-normalization of the Galois representations associated to $\bT_\Lambda^\ord$ in order to make them vary with weight and not with level. 
\begin{defn}
    \label{defn: weight normalized etale coh}
    Let $V_\Lambda^\ord$ denote $H \otimes_\Lambda \dia^{-1}(1)$, that is, the twist by the character $\dia^{-1}(1)$. (Note that $\dia^{-1}(1) \otimes_{\Lambda, \phi_k} \Z_p \simeq \Z_p(2-k)$.) Let $V_k^\ord := V_\Lambda^\ord \otimes_{\Lambda, \phi_k} \Z_p$. 
\end{defn}

For clarity, we describe $V_\Lambda^\ord$ as in Theorem \ref{thm: FK Lamda-adic etale}. 
\begin{cor}
    \label{cor: geom Lambda-adic etale}
    For a prime $\ell \nmid Np$, the characteristic polynomial of $\Frob_\ell$ acting on $V^\ord$ is $X^2 - [\ell]^{-1}T_\ell X + [\ell]^{-1} \lr{\ell}_N$.  There is a short exact sequence of $\Lambda$-flat $\bT_\Lambda^{\ord,\circ}[G_p]$-modules 
        $0 \to V^\ord_\mathrm{sub} \to V^\ord \to V^\ord_\mathrm{quo} \to 0$ with $\bT_\Lambda^{\ord,\circ}$-module isomorphisms as in Theorem \ref{thm: FK Lamda-adic etale}. The action of $G_p$ on $V^\ord_\mathrm{sub}$ and $V^\ord_\mathrm{quo}$ is described by $\bT_\Lambda^\ord$-valued characters 
    \[
    \chi^\ord_\mathrm{sub} = \nu(U_p^{-1}\lr{p}_N), \qquad \chi^\ord_\mathrm{quo} = \dia^{-1}\nu(U_p). 
    \]
\end{cor}

Let $D_\mathrm{cris}$ denote the usual covariant functor from crystalline $G_p$-representations over $\Q_p$ to weakly admissible filtered isocrystals over $\Q_p$, 
\[
D_\mathrm{cris} : V \mapsto (V \otimes_{\Q_p} B_\mathrm{cris})^{G_p}. 
\]

\begin{prop}
\label{prop: isocrystal and galois rep}
The $G_p$-module $V_k^\ord$  of Definition \ref{defn: weight normalized etale coh} is crystalline for $k\geq 3$. Moreover,  $D_\mathrm{cris}(V_{k,\Q_p}^\ord)$ is naturally isomorphic to the weakly admissible filtered isocrystal $\dR_k \otimes_{\Z_p} \Q_p$ (notation of Definition \ref{defn: dR}). 
\end{prop}

\begin{proof}
    See e.g.\ \cite[Thm.\ 5.6]{faltings1989}. (Recall that $k\in\Z_{\geq 3}$ for simplicity.) 
\end{proof}

\subsection{Bi-ordinary implies $p$-split and a converse statement}

Now we will see that any of the 2-dimensional Galois representations supported by the bi-ordinary Hecke algebra is $p$-split, that is, reducible and decomposable. It will suffice for our purposes to state the following form. 
\begin{prop}
    \label{prop: BO implies split}
    Let $\bT_\Lambda^\SBO \to F$ be an algebraic extension of a residue field of a prime ideal of $\bT_\Lambda^\SBO$ such that the associated semi-simple Galois representation $\rho : G_\Q \to \GL_2(F)$ has absolutely irreducible Jordan--H\"older factors. If $\dia \otimes_\Lambda F \neq  1$ or $\lr{p}_N^{-1} U_p^2 \neq 1 \in F$, then $\rho\vert_{G_p}$ is reducible and decomposable. 
\end{prop}

To understand what these Galois representations are, and to start toward the proof of Proposition \ref{prop: BO implies split}, we describe the $G_p$-stable flags appearing in the realization of $H$ as a $\bT_\Lambda^\tord[G_\Q]$- and $\bT_\Lambda^\aord[G_\Q]$-module. We call these $H^\tord$ and $H^\aord$, respectively. Because the comparison $\bT_\Lambda^{\tord,\circ} \isoto \bT_\Lambda^{\ord,\circ}$ is as simple as the replacement $U' \mapsto U$ (Proposition \ref{prop: tord control}), the form of $H^\tord$ is straightforward from Theorem \ref{thm: FK Lamda-adic etale}. On the other hand, the twist in the isomorphism $\bT_\Lambda^{\aord,\circ} \isoto \bT_\Lambda^{\ord,\circ}$ described in Proposition \ref{prop: aord TU algebra} results in stabilizing the opposite flag. We apply Corollary \ref{cor: geom Lambda-adic etale} and think of the (generalized) matrix realizations of the $G_p$-actions on $V^\tord$ and $V^\aord$ as 
\[
\tord 
: \ttmat{\nu(U'^{-1}\lr{p}_N)}{*}{0}{\dia^{-1}\nu(U')}, \quad 
\aord: \ttmat{\nu(U'^{-1}\lr{p}_N)}{0}{*}{\dia^{-1}\nu(U')}.
\]
We will refer to these as the forms of the \emph{twist- and anti-ordinary flags}. Working over a field, the existence of these two stable flags must imply $p$-splitness as long as the two characters comprising the Jordan--H\"older factors are distinct. 

\begin{proof}[{Proof of Proposition \ref{prop: BO implies split}}]
    We assume that $\rho$ is absolutely irreducible, since the complementary situation where $\rho$ is reducible semi-simple is trivial. 
    
    Due to the surjection $\bT_\Lambda^{\tord,\circ} \rsurj \bT_\Lambda^\SBO$, there is an isomorphism of Galois representations $\rho \simeq H^\tord \otimes_{\bT_\Lambda^{\tord,\circ}} F$ because they are irreducible and the characteristic polynomials of Frobenius elements $\Frob_\ell \in G_\Q$ are identical. Therefore $\rho\vert_{G_p}$ admits a twist-ordinary flag. Likewise, the surjection $\bT_\Lambda^{\aord,\circ} \rsurj \bT_\Lambda^\SBO$ implies that $\rho \simeq H^\aord \otimes_{\bT_\Lambda^{\aord,\circ}} F$ and therefore $\rho\vert_{G_p}$ admits an anti-ordinary flag as well. By the Jordan--H\"older theorem, therefore $\rho\vert_{G_p}$ must be $p$-split unless the two Jordan--H\"older factors are identical. The assumptions of the theorem are equivalent to the distinctness of the two Jordan--H\"older factors of $\rho\vert_{G_p}$. 
\end{proof}

Here is the extent to which we can reprove Breuil--Emerton's Theorem \ref{thm: BE} using the connections between $\SBO_\Lambda^\bullet$ and classical modular forms that we have established so far. Recall that $M_k^\tord$ ($k \in \Z_{\geq 3}$) is nothing other than classical $U_p$-critical forms with $U'$ acting by $\lr{p}_Np^{k-1}U_p^{-1}$. As usual, let $f_\beta$ (resp.\ $f_\alpha$) denote the $U_p$-critical (resp.\ $U_p$-ordinary) $p$-stabilization of $f$, a classical normalized eigenform of level $\Gamma_1(N) \cap \Gamma_0(p)$ with $U_p$-eigenvalue $\beta$ (resp.\ $\alpha$). Let $\phi_f : \bT[T_p](e(T_p)H^0(X,\omega^k)) \to E$ arise from the Hecke action on $f$. Let $\dR_f$ denote $\dR_k \otimes_{\bT[T_p], \phi_f} E$ with $\dR_k$ as in Definition \ref{defn: dR}. Let $\rho_f : G_\Q \to \GL_2(E)$ denote the associated 2-dimensional semi-simple Galois representation, which is $\rho_f \simeq (V_k^\ord)_f(k-1)$ when $f$ is cuspidal. 

\begin{thm}
    \label{thm: BE alternate}
    Let $k \in \Z_{\geq 3}$ and let $f \in e(T_p)H^0(X_E, \omega^k)$ and $f_\beta \in M_{k,E}^\tord$ be as above; in particular, $f_\beta$ is $U_p$-critical.  The following conditions are equivalent. 
    \begin{enumerate}
        \item $f$ is in the image of the map $\theta^{k-1}$ of \eqref{eq: theta} $\otimes E$, or, what is the same, $f \in M_{k,E}^\aord$
        \item $\rho_f\vert_{G_p}$ is reducible and decomposable
        \item the $f$-isotypical part of $H^*(\BO_k^\bullet \otimes_{\Z_p} E)$ is non-zero, i.e.\ its Hecke eigensystem is bi-ordinary. 
    \end{enumerate} 
\end{thm}

\begin{proof}
    In the case that $f$ is Eisenstein, which is well-known to be equivalent to $\rho\vert_{G_\Q}$ being reducible, condition (1) is true by Corollary \ref{cor: Eisenstein theta}, while condition (3) is true by the first statement in Theorem \ref{thm: SD for BO}. 

    When $f$ is cuspidal, (3) implies (2) by Proposition \ref{prop: BO implies split}, while (1) implies (3) by Corollary \ref{cor: BO is to and ao}. We also know (3) implies (1) by Proposition \ref{prop: coleman non-split}. 

    The key additional input is Proposition \ref{prop: isocrystal and galois rep}, which allows us to deduce (2) $\Rightarrow$ (1) from Lemma \ref{lem: isocrystal splittability}. 
\end{proof}

\begin{lem}
    \label{lem: isocrystal splittability}
    $\dR_f$ is splittable into two weakly admissible filtered isocrystals of $E$-dimension $1$ if and only if $f_\beta \in M^\aord_{k,E}$. 
\end{lem}

\begin{proof}
    We assemble some consequences, relevant for his classicality theorem, of Coleman's Proposition \ref{prop: Coleman pres dR} -- that is, the image of $\theta^{k-1}$ has slope $\geq k-1$.  Note that subscripts $f_\alpha,f_\beta$ refer to $\bT[U]$-isotypical parts. 
    \begin{itemize}
        \item $\KS_k(f_\alpha)$ never vanishes in $\dR_f$.
        \item Because $f = (\beta f_\alpha - \alpha f_\beta)/(\beta - \alpha)$, $f_\beta \in M^\aord_{k,E}$ if and only if $\KS_k(f)$ and $\KS_k(f_\alpha)$ are $E$-colinear in $\dR_f$. 
        \item In no case does the submodule $\KS_k(f)$ lie in the $f_\beta$-isotypic line $(\dR_f)_{f_\beta}$. 
    \end{itemize} 
    The 1-dimensional subspaces $(M^\ord_{k,E})_{f_\alpha} = \lr{f_\alpha}$ and $(\dR_f)_{f_\beta}$ in $\dR_f$ are the two lines stable under crystalline Frobenius $\varphi = \lr{p}_N p^{k-1}U^{-1}$. They have $\varphi$-slopes $k-1$ and $0$, respectively. 
    
    Putting together the facts above, we deduce that the $\bT[T_p]$-submodule $\KS_k(\lr{f}) \subset \dR_f$, which is the non-trivial submodule of the Hodge filtration of $\dR_f$, is $\varphi$-stable if and only if $f_\beta \in M^\aord_{k,E}$. Then a brief argument about filtered isocrystals shows that this criterion is equivalent to the splittability statement of the lemma. 
\end{proof}

As a result, we can obtain a partial converse to Proposition \ref{prop: BO implies split}. We think of this as a generalization of Ghate--Vatsal's result (\cite[Thm.\ 3]{GV2004}, Proposition \ref{prop: GV})that modular forms with bi-ordinary Hecke eigensystems capture the $G_p$-decomposability property for general residual eigensystems. They show CM forms capture this property when the residual eigensystem is nice enough. 

Let $\cF$ denote a twist-ordinary Hida family of tame level $N$ valued in a finite flat $\Lambda$-domain $A$. The typical example of $A$ is the normalization of a quotient of $\bT_\Lambda^\tord$ by a minimal prime ideal. In the following statements, subscripts ``$A$'' refer to the base change along $- \otimes_\Lambda A$ of the usual $\Lambda$-valued object. For example, $\cF$ is a Hecke eigen-element of $M_A^\tord$. 

\begin{thm}
\label{thm: BO equals split in Hida families}
Let $\cF$ be a twist-ordinary Hida family valued in $A$ as above. The following conditions are equivalent. 
\begin{enumerate}
    \item $\cF$ lies in the image of $\Theta_A : M_A^\aord \rinj M_A^\crit$
    \item the associated Galois representation $\rho_\cF$ is reducible and decomposable on $G_p$
    \item the $\cF$-isotypic part of $H^*(\BO_A^\bullet \otimes_A Q(A)) $ is non-zero, that is, its Hecke eigensystem is bi-ordinary. 
\end{enumerate}
\end{thm}

\begin{proof}
    We may presume $\cF$ is cuspidal because the Eisenstein case is trivial (Lemma \ref{lem: Eis difference}). The equivalence $(1) \Leftrightarrow (3)$ follows from the definition of the bi-ordinary complex. By Proposition \ref{prop: BO implies split}, we know $(3) \Rightarrow (2)$, so it remains to establish $(2) \Rightarrow (3)$. 

    The Galois representation associated to $\cF$ is $\rho_\cF : G_\Q \to \GL_2(Q(A))$ arising from $H_\Lambda^\tord$; because $\cF$ is cuspidal, $\rho_\cF$ is absolutely irreducible. Combining  Theorem \ref{thm: BE alternate} and the argument of Ghate--Vatsal that we record in Lemma \ref{lem: GV splitting}, we deduce that for all but finitely many $k \in \Z_{\geq 3}$, the specialization of $\cF$ to weight $k$ has a bi-ordinary eigensystem. Due to Proposition \ref{prop: BO T to forms duality}, specifically the statement that the kernel of $\bT_\Lambda^\SBO \rsurj \bT[U']_\Lambda(H^0(\SBO_\Lambda^\bullet))$ is $\Lambda$-torsion, this implies that $\bT_\Lambda^{\tord,\circ} \rsurj S$ factors through $H^0(\SBO_\Lambda^\bullet)$. Thus, applying Corollary \ref{cor: BO is to and ao}, the eigensystem of $\cF$ is bi-ordinary. 
\end{proof}

\begin{lem}[Ghate--Vatsal]
    \label{lem: GV splitting}
    Given $A$ and $\rho_\cF$ as in the proof above, $\rho_\cF\vert_{G_p}$ is reducible and decomposable if and only if it is reducible and decomposable at all but finitely many height 1 primes of $A$. 
\end{lem}

\begin{proof}
    This is found in \cite[Thm.\ 18, proof]{GV2004}. One may make sense of ``$\rho_\cF$ is reducible and decomposable at a height 1 prime $P$ of $A$'' by these properties for $H_\Lambda^\ord \otimes_{\bT_\Lambda^\ord} \overline{Q(A/P)}$, which matches the notion used \textit{ibid}.
\end{proof}

\section{Deformation theory and ``$R = \bT\,$''}
\label{sec: deformations}

In this section, we set up deformation rings $R^\star$ for four deformation problems ``$\star$'' that we expect to coincide with the Galois representations arising from the various flavors of $\Lambda$-adic Hecke algebras, and then prove $R^\star \cong \bT^\star$. The options for $\star$ we will address are: twist- and anti-ordinary, bi-ordinary, and critical.  Of note, the formulation of the critical deformation rings seems to be new. We will treat $\star = $ twist/anti-ordinary as base cases, since they are twists of the $\star = \ord$ case where $R^\ord \cong \bT_\Lambda^\ord$ has been established under mild (Taylor--Wiles) hypotheses since the landmark work on modularity by Wiles and Taylor--Wiles \cite{wiles1995, TW1995} followed by many others. Under that hypothesis, we deduce in Theorems \ref{thm: R=T BO} and \ref{thm: R=T crit} that 
\[
R^\BO_\Lambda \cong \bT_\Lambda^\BO \quad \text{and} \quad R^\crit_\Lambda \cong \bT^\crit_\Lambda. 
\]

This answers a question that we posed in our previous work \cite[\S1.6]{CWE1}. There we called the deformation problem $\cD = \BO$ by the name of ``$p$-split'' and wrote ``$\star = \spl$,'' since it parameterizes global Galois representations that are reducible and decomposable upon restriction to a decomposition group at $p$. The main theorem of \cite{CWE1} is a criterion for $R^\BO_\Lambda$ to be isomorphic to CM Hecke algebra $\bT^{\CM,\ord}_\Lambda$; because it is reasonable to expect that this criterion sometimes fails to be true, there is a need for a notion of modular form matching the $p$-split deformation problem. Our formulation of the bi-ordinary complex and the proof in this section of ``$R^\BO_\Lambda \cong \bT^\BO_\Lambda$'' satisfies this question. 

We remark that while ``SBO'' for ``cuspidal bi-ordinary'' might be more appropriate than ``BO'' in the following discussion, we will only work with deformation theory over residually non-Eisenstein localizations, where the difference vanishes. 

\subsection{Ordinary and critical Galois representations}
\label{subsec: Gal rep discussion}

Here we discuss the usual ordinary deformation condition, and then describe standard facts about Galois representations associated to critical classical and critical overconvergent \emph{generalized} Hecke eigenforms.  

When $f$ is classical eigenform of weight $k$ on $X_1(p^r)$ (the modular curve of level $\Gamma_1(Np^r)$) and character $\chi$, our normalization for the Galois representation $\rho_f : G_\Q \to \GL_2(\oQ_p)$ coming from $f$ is that 
\begin{equation}
    \label{eq: rho normalization}
    \Tr\rho_f(\Frob_\ell) = T_\ell, \qquad  \det \rho_f(\Frob_\ell) = \ell^{k-1}\chi(\ell)    
\end{equation}
for primes $\ell \nmid Np$. When $r=0$ and $f$ is $T_p$-ordinary and thereby has ordinary $p$-stabilization $f_\alpha$ with $U_p$-eigenvalue $\alpha$, or when $r \geq 1$ and $f$ has $U_p$-eigenvalue $\alpha$, the local shape of $\rho_f$ is
\begin{equation}
\label{eq: ordinary shape}
0 \to \nu(\alpha^{-1}\chi(p))\chi_{p^r}\kappa^{k-1} \to \rho_f\vert_{G_p} \to \nu(\alpha) \to 0
\end{equation}
where $\nu(\alpha) : G_p \to \oQ_p^\times$ sends arithmetic Frobenius $\Frob_p$ to $\alpha$ and $\chi_{p^r}$ is the $p$-primary part of the nebencharacter. When discussing arithmetic normalizations, we say that $\kappa$ has Hodge--Tate weight $1$. 

This is the usual reducible shape that we will use to define an ordinary deformation problem (dating back to \cite[\S2.5]{mazur1989}). We will soon define related reducible shapes of Galois representations that we will call \emph{twist-ordinary} and \emph{anti-ordinary} to correspond to the related twists of modular forms discussed in \S\S\ref{subsec: tord}-\ref{subsec: aord}, which are nothing fundamentally new because they are straightforward twists of the ordinary shape. 

In contrast, in order to define a condition on 2-dimensional $G_p$-representations coming from critical overconvergent generalized eigenforms, we must allow for contributions that do not come (directly; only by interpolation) from the \'etale cohomology of varieties; after all, critical overconvergent forms are not spanned by classical forms, so we do not expect their Galois representation, restricted to $G_p$, to be de Rham. The following example is instructive. 

\begin{eg}
	\label{eg: hsu}
    Chi-Yun Hsu \cite{hsu2020} has explained the shape of a Galois representation $\rho_f$ arising from a critical overconvergent generalized eigenform associated to a critical CM (classical cuspidal) eigenform $f$ of weight $k \geq 2$ with Hecke field $E/\Q_p$ and imaginary qudaratic field $K/\Q$ of CM; conversely, this gives an arithmetic expression for the $q$-series of a strictly generalized eigenform [Thm.\ 1.1, \textit{loc.\ cit.}]. 
    
    In this case, $\rho_f \simeq \Ind_K^\Q \nu$ where $\nu : G_\Q \to E^\times$ and $p$ splits in $K/\Q$, hence 
    \[
    \rho_f \vert_{G_p} \simeq \ttmat{\nu}{0}{0}{\nu^c}\Big\vert_{G_p}.
    \]
    (For the definition of $(-)^c$, actions of complex conjugation on characters $\nu$ and extension classes $e$, see \cite[\S5.3]{CWE1}.) For simplicity, we assume that the generalized eigenspace with eigensystem $f$ is 2-dimensional; see \cite[Prop.\ 4.3]{hsu2020} for an equivalent arithmetic characterization. Then generalized eigenspace can be thought of as an eigenform $\tilde f$ valued in $E[\ep]/\ep^2$. The associated Galois representation $\rho_{\tilde f}$ has restriction to $G_p$ of the form
    \begin{equation}
    \label{eq: tilde rho at p} 
    \rho_{\tilde f}\vert_{G_p} \simeq \ttmat{\nu}{\ep \cdot e}{\ep \cdot e^c}{\nu^c}
    \end{equation}
    where $e$ is a representative of an extension class $[e] \in \Ext^1_{\Q_p[G_K]}(\nu^c, \nu)$ such that both $[e]\vert_{G_p}$ and $[e^c]\vert_{G_p}$ are non-trivial. 
    \end{eg}

    Our initial observations are that Galois representations associated to critical overconvergent generalized eigenforms 
    \begin{itemize}
	\item need not be reducible on $G_p$, unlike the case of ordinary eigenforms
	\item but their associated pseudorepresentation (trace and determinant functions) equal those of a reducible representation (namely, $\rho_f \otimes_E E[\ep]/\ep^2$ in this case).
    \end{itemize}

To discuss Example \ref{eg: hsu} further, let us say that we are in the case where $\nu$ has Hodge--Tate weight $k-1$ while $\nu^c$ has Hodge--Tate weight $0$, meaning that $e^c\vert_{G_p}$ (not $e\vert_{G_p}$) produces the failure of the de Rham condition on $\rho_{\tilde f}\vert_{G_p}$ of \eqref{eq: tilde rho at p}: the Hodge--Tate weights are in the ``wrong order.'' 

According to Kisin \cite[Thm.,\ p.\ 376]{kisin2003}, the Galois representations associated to $p$-adic families of overconvergent eigenforms  have an interpolable crystalline eigenvalue in the maximal crystalline quotient. For example, a critical overconvergent eigenform $g$ with $p$-local shape $\rho_g\vert_{G_p} \simeq \sm{\nu}{0}{e^c}{\nu^c}$ 
(just using these expressions as an example of local representations and not insisting that $g$ is globally CM) is not crystalline but has a critical crystalline eigenvalue in the crystalline quotient. This is the form of the Galois representation of an anti-ordinary weight $k$ eigenform that is not also twist-ordinary. On the other hand, a twist-ordinary form has $p$-local shape $\sm{\nu}{e}{0}{\nu^c}$, is crystalline, and has both an ordinary and critical slope crystalline eigenvalue. 

We propose an extension of Kisin's characterization to the case of infinitesimal coefficients, ultimately justifying this characterization by proving $R^\crit \cong \bT^\crit_\Lambda$ (Theorem \ref{thm: R=T crit}) for a \emph{critical deformation problem} represented by $R^\crit$ that we will now go on to define. To develop the critical deformation problem, it will be helpful to discuss the infinitesimal flexibility of critical slope eigenvalue in Example \ref{eg: hsu} a bit more. It is a flexibility that ordinary slope eigenvalues lack. This has of course already been observed Galois-theoretically, e.g.\ in Hsu's Example \ref{eg: hsu} and Bergdall's study \cite{bergdall2014}. 

Consider again the deformation $\rho_{\tilde f}$ of $\rho_f$ as in Example \ref{eg: hsu}. 
\begin{itemize}
    \item The slope $0$ crystalline eigenvalue of $\nu^c$ cannot deform from $E$ to $E[\ep]/\ep^2$ because $e^c\vert_{G_p}$ is a non-crystalline extension class, making it impossible for the maximal crystalline quotient of $\rho_{\tilde f}\vert_{G_p}$ to have a $E[\ep]/\ep^2$-rank 1  composition factor with $G_p$-action deforming $\nu^c$. 
    \item On the other hand, the slope $k-1$ crystalline eigenvalue of $\nu$ does deform: the maximal crystalline $G_p$-linear quotient of $\rho_{\tilde f}$ has a $G_p$-linear sub that is of rank 1 over $E[\ep]/\ep^2$ and deforms $\nu$. 
\end{itemize}

The upshot of this discussion is that a critical slope crystalline eigenvalue (in a rank 2 representation) can interpolate over a non-crystalline extension class as long as the non-crystalline extension can be killed off in a quotient that preserves the interpolation of the eigenvalue. Thus, equivalently, the ``bad'' extensions should not show up in the characteristic polynomials of $\rho_{\tilde f}\vert_{G_p}$. This is what led us to formulate the critical deformation problem using pseudorepresentations. 

\subsection{Deformation conditions}

Now we set up definitions for ``ordinary,'' etc., which we will use in deformation theory. While ``weight $k$'' will mean that the Hodge--Tate weights are $[0,k-1]$, we avoid using Hodge--Tate weights in the following general formulation because formulating Hodge--Tate weights for representations with coefficients in a topological $\Z_p$-algebra is an additional complexity that we do not require in these cases. 

\begin{defn}
    \label{defn: p-local conditions}
    Let $A$ be a topological $\Z_p$-algebra and let $V$ be a projective $A$-module of rank 2 with an $A$-linear action of $G_p$. Let $\alpha, \beta \in A^\times$. Denote by $\nu(\alpha) : G_p \to A^\times$ the unramified character sending arithmetic Frobenius $\Frob_p$ to $\alpha$. Call $V$ 
    \begin{itemize}[leftmargin=1.5em]
        \item \emph{ordinary of weight $k$ and $\Frob_p$-eigenvalue $\alpha$} if $\chi_p := (\det V)(1-k)\vert_{I_p}$ has finite order and $V$ has a $G_p$-stable submodule $W$ of rank $1$ such that $V/W \simeq \nu(\alpha)$. (In particular, $V/W$ is unramified.)
        \item \emph{twist-ordinary of weight $k$ and $\Frob_p$-eigenvalue $\alpha$} if $\chi_p := ((\det V)(1-k)\vert_{I_p})^{-1}$ has finite order and $V$ has a $G_p$-stable submodule $W$ of rank $1$ such that $V/W \otimes \chi_p \simeq \nu(\alpha)$. (In particular, $W$ is crystalline of Hodge--Tate weight $1-k$ and its critical crystalline eigenvalue is determined by $\alpha$ and $\det V$.)
        \item \emph{anti-ordinary of weight $k$ and $\Frob_p$-eigenvalue $\beta$} if $\chi_p := (\det V)(1-k)\vert_{I_p}$ has finite order and $V$ has a $G_p$-stable submodule $W$ of rank $1$ such that $V/W(1-k) \simeq \nu(\beta)$. In particular, $V/W(1-k)$ is unramified. (Note that anti-ordinary representations are not de Rham when $k \geq 2$ and $W$ does not have a complimenary $G_p$-subrepresentation.) 
        \item \emph{bi-ordinary of weight $k$ and with twist-$\Frob_p$-eigenvalue $\alpha$ and anti-$\Frob_p$-eigenvalue $\beta$} if $\chi_p := (\det V)(1-k)\vert_{I_p}$ has finite order and $V$ splits as a $A[G_p]$-module into a direct sum of rank $1$ summands $W, W'$, where $W \otimes \chi_p^{-1} \simeq \nu(\alpha)$ and $W'(1-k) \simeq \nu(\beta)$. In particular, both $W \otimes \chi_p^{-1}$ and $W'(1-k)$ are unramified.  
        \item \emph{critical of weight $k$ and $\varphi$-eigenvalue $\alpha$} if $\chi_p := (\det V)(1-k)\vert_{I_p}^{-1}$ has finite order and there exist rank $1$ $A[G_p]$-modules $W,W'$ such that $W \otimes \chi_p \simeq \nu(\alpha)$ is unramified and there is an equality of 2-dimensional pseudorepresentations $\psi(V) = \psi(W \oplus W')$. In particular, this implies that $W'(1-k)$ is unramified. 
    \end{itemize}
\end{defn}

\begin{rem}
    Twist-ordinary could be called ``classical critical'' instead, but we reserve ``critical'' entirely for the overconvergent setting. Also, in this paper we mainly work at level $\Gamma_1(N)$ so that the Galois representations arising have $(\det V)(1-k)\vert_{I_p}$ trivial, not merely of finite order. In this case, ordinary is equivalent to twist-ordinary, although they have different eigenvalues. We are including the more flexible case, where $p$ divides the level, for future use. 
\end{rem}

\begin{rem}
    \label{rem: relations of conditions}
    Notice the following relation between the conditions, 
    \[
    \BO \iff \mathrm{to} \text{ and } \aord, \quad \mathrm{to} \text{ or } \aord \implies \mathrm{crit}. 
    \]
    When $A$ is a domain, then critical $\Rightarrow$ twist-ordinary or anti-ordinary, but it will be important that there are critical (overconvergent) \emph{generalized} eigenforms that are neither twist-ordinary nor anti-ordinary; indeed, this is what $H^1(\BO_\Lambda^\bullet)$ represents.    
\end{rem}

Next we state the $\Lambda$-adic interpolation of the definition above. These are somewhat more strict in that they specialize along $\phi_k : \Lambda \to \Z_p$ to the definition above with the additional stipulation that $\chi_p = 1$. This is to be expected since $\Lambda$ parameterizes exactly $\det V\vert_{I_p}$, i.e.\ both weight and the $p$-part of the nebencharacter, while our notion of ``weight $k$'' allows for a general nebencharacter. 
\begin{defn}
    \label{defn: p-local conditions Lambda}
    Let $A$ be a topological $\Lambda$-algebra and let $V$ be a projective $A$-module of rank 2 with an $A$-linear action of $G_p$. Let $\alpha \in A^\times$. Denote by $\dia$ the character
    \[
    \dia : G_p \to \Lambda^\times, \qquad G_p \rsurj \Z_p^\times \rinj \Z_p\lb \Z_p^\times \rb^\times = \Lambda^\times.
    \]
    where $G_p \rsurj \Z_p^\times$ is the isomoprhism of local class field theory cut out by the maximal abelian totally ramified algebraic extension of $\Q_p$. 
    Call $V$ 
    \begin{itemize}
        \item \emph{$\Lambda$-ordinary of $\Frob_p$-eigenvalue $\alpha$} if $I_p$ acts on $\det V$ by $\dia$ and $V$ has a $G_p$-stable submodule $W$ of rank $1$ such that $V/W$ unramified with $G_p/I_p$ acting by $\nu(\alpha)$. 
        \item \emph{$\Lambda$-twist-ordinary of $\Frob_p$-eigenvalue $\alpha$}: this is the same as the definition of $\Lambda$-ordinary
        \item \emph{$\Lambda$-anti-ordinary of $\Frob_p$-eigenvalue $\beta$} if $I_p$ acts on $\det V$ by $\dia$ and $V$ has a $G_p$-stable submodule $W$ of rank $1$ such that $W \otimes_A A\dia^{-1}$ is unramified with $G_p/I_p$ acting by $\nu(\beta)$.
        \item \emph{$\Lambda$-bi-ordinary of twist-$\Frob_p$-eigenvalue $\alpha$ and anti-$\Frob_p$-eigenvalue $\beta$} if $I_p$ acts on $\det V$ by $\dia$ and $V$ splits as a $A[G_p]$-module into a direct sum of $A$-rank $1$ summands $W, W'$, where $W$ is unramified with $G_p/I_p$ acting by $\nu(\alpha)$ and $W' \otimes_A A\dia^{-1}$ is unramified with $G_p/I_p$ acting by $\nu(\beta)$. 
        \item \emph{$\Lambda$-critical of $\varphi$-eigenvalue $\alpha$} if $I_p$ acts on $\det V$ by $\dia$ and there exist $A$-rank $1$ $A[G_p]$-modules $W,W'$ such that $W \simeq \nu(\alpha)$ is unramified and there is an equality of 2-dimensional pseudorepresentations $\psi(V) = \psi(W \oplus W')$.
    \end{itemize}
\end{defn}

\subsection{Ordinary universal Galois deformation rings}

Now we will formulate a standard context in which ordinary $R=\bT$ is established, which we will use as a hypothesis in our applications. Particularly, we assume Taylor--Wiles and residually $p$-distinguished hypotheses about the residual representation $\rho$. Of course, because we are interested in Galois representations that are reducible on $G_p$ and satisfy the conditions defined above, we ask that $\rho\vert_{G_p}$ is reducible with an unramified Jordan--H\"older factor. It is not necessary to assume that $\rho\vert_{G_p}$ is splittable, but the interaction between the conditions twist-ordinary, anti-ordinary, and critical will be rather trivial on deformations of $\rho$ unless $\rho\vert_{G_p}$ is splittable. 

\begin{defn}[Assumptions on $\rho$]
    \label{defn: rho ord assumptions}
    Let $\F$ be a finite field of characteristic $p$. Let $\rho : G_\Q \to \GL_2(\F)$ be a representation satisfying the following assumptions. 
    \begin{enumerate}
        \item $p \geq 5$ (which is assumed throughout the paper)
        \item $\rho$ is odd and has Artin conductor $N$
        \item $\rho\vert_{G_M}$ is absolutely irreducible, where $M = \Q(\sqrt{(-1)^{(p-1)/2}p})$
        \item $\rho\vert_{G_p}$ is reducible with semi-simplification isomorphic to a sum of characters $\chi_1 \oplus \chi_2 : G_p \to \GL_2(\F)$, where
        \begin{itemize}
            \item $\chi_2$ is unramified
            \item $\chi_1 \neq \chi_2$.
        \end{itemize}
    \end{enumerate}
\end{defn}

Using assumptions (3) and (4), the following deformation problems will be pro-representable by deformation rings. The category of coefficient rings is $\cA_{W}$, the Artinian local $W=W(\F)$-algebras with residue field identified with $\F$ compatibly with the canonical residue map $W \rsurj \F$. We write $\hat \cA_W$ for Noetherian local $W$-algebras with residue field $\F$. Also write $G_{\Q,S}$ for the Galois group of the maximal algebraic extension of $\Q$ ramified only at the places $S$ dividing $Np\infty$. As we express this deformation problem, we consider a deformation $\rho_A$ to give $A$ a $\Lambda$-algebra structure via the moduli interpretation of $\Lambda$ applied to $\det \rho_A$. Let $\Lambda_W := \Lambda \otimes_{\Z_p} W$. 

\begin{defn}[{The minimal ordinary deformation functor, e.g.\ \cite[\S3.1]{DFG2004}}]
\label{defn: Dord}
Let $D^\ord : \hat\cA_W \ra \mathrm{Sets}$ be the functor associating to $A$ the set of strict equivalence classes of homomorphisms $\rho_A : G_{\Q,S} \to \GL_2(A)$ such that
\begin{enumerate}[label=(\roman*), leftmargin=2em]
\item $\rho_A \otimes_A \F = \rho$;
\item $\rho_A\vert_{G_p}$ is $\Lambda$-ordinary with unramified quotient $\nu(\alpha)$ deforming $\chi_2$ $(\alpha \in A^\times)$
\item for primes $\ell \mid N$ such that $\#\rho(I_\ell) \neq p$, reduction modulo $\m_A$ induces an isomorphism $\rho_A(I_\ell) \risom \rho(I_\ell)$;
\item for primes $\ell \mid N$ such that $\#\rho(I_\ell) = p$, $\rho_A^{I_\ell}$ is $A$-free of rank 1. 
\end{enumerate}

Deformations $\rho_A$ of $\rho$ satisfying the conditions defining $D^\ord$ will be known as \emph{ordinary of tame level $N$}, or just \emph{ordinary}. 

Also, for $k \in \Z$, let $D^\ord_k$ denote the subfunctor of $D^\ord$ cut out by the additional condition that $\det\rho_A(1-k)$ is ramified only at primes dividing $N$. 
\end{defn}

One can readily check that $D^\ord \times_{\Spf \Lambda,\phi_k} \Spf \Z_p \cong D^\ord_k$. Also, under the residually $p$-distinguished assumption that we have imposed, the datum of the $\Frob_p$-eigenvalue $\alpha \in A^\times$ is determined by $\rho_A$. 

\begin{prop}
\label{prop: ord R to T}
Under the assumptions of Definition \ref{defn: rho ord assumptions}, $D^\ord$ is representable by $R^\ord \in \hat\cA_W$ and there is a universal ordinary deformation $\rho^\ord : G_{\Q,S} \ra \GL_2(R^\ord)$ of $\rho$. Moreover, the universal $\chi_2^\ord$ equals $\nu(u)$ for some $u \in R^{\ord,\times}$. Likewise, $D^\ord_k$ is representable by $R^\ord_k \in \hat\cA_W$ and there is a natural surjection $R^\ord \rsurj R^\ord_k$ realizing an isomorphism $R^\ord/(\ker \phi_k) R^\ord$, where $\Lambda \to R^\ord$ arises from $\det \rho^\ord$. 
\end{prop}
\begin{proof}
    The first statement is \cite[Prop.\ 2.3.2]{CWE1}. The remaining statements are standard consequences. 
\end{proof}

Next we discuss the map $R^\ord \to (\bT^\ord_\Lambda)_\rho$ arising from the representation $\rho_{\bT^\ord_\Lambda}: G_{\Q,S} \to \GL_2(\bT_\Lambda^\ord)$ described in e.g.\ \cite[Prop.\ 2.1.7]{CWE1}, which is the twist of the geometrically normalized Galois representation of Corollary \ref{cor: geom Lambda-adic etale} by $\dia$. In particular, it satisfies
\[
 \det(X \cdot I_{2\times2}-\rho(\Frob_\ell)) = X^2 - T_\ell X + [\ell] \lr{\ell}_N ,  \quad \text{ for } \ell \nmid Np \text{ prime},
\]
and has $p$-local form 
\begin{equation}
\label{eq: ord form of rho_T}
\rho_{\bT_\Lambda^\ord}\vert_{G_p} \simeq \ttmat{\dia \cdot \nu({\lr{p}_N} U_p^{-1})}{*}{0}{\nu(U_p)}.
\end{equation}

By the universal property of $R^\ord$, we have $R^\ord \to (\bT^\ord_\Lambda)_\rho$ which is easily seen to be a surjection due to the design of Definitions \ref{defn: rho ord assumptions} and \ref{defn: Dord}. For instance, by Proposition \ref{prop: ord R to T}, $U_p$ is in the image. 

Since we are interested in relying on known proofs of $R^\ord \cong (\bT_\Lambda^\ord)_\rho$ and deducing consequences, we will simply take this isomorphism as a declared \emph{hypothesis} in what follows. For example, sufficient conditions are given at the start of \cite[\S3.1]{DFG2004}, relying on \cite{diamond1997}.

\subsection{Twist/anti-ordinary, bi-ordinary, and critical universal deformation rings}

Next we set up global deformation problems with the $p$-local conditions we need: twist-ordinary, anti-ordinary, and critical. The residual representation $\rho$ is twisted from the ordinary setting above, followed by the critical. The goal is for these first two deformation functors is to be sub-functors of the critical deformation problem after a twist back. 

\begin{defn}
    We define the twist-ordinary deformation functor $D^\tord$ in exactly the same way as the ordinary deformation functor $D^\ord$ in Definition \ref{defn: Dord}. (Or we could say we replace ``$\Lambda$-ordinary'' with ``$\Lambda$-twist-ordinary, which is the same thing.) We define the anti-ordinary deformation functor $D^\aord$ to be the same as in Definition \ref{defn: Dord} but with ``$\Lambda$-anti-ordinary'' replacing ``$\Lambda$-ordinary.'' 
\end{defn}

\begin{lem}
    \label{lem: twists of deformation problems}
    The deformation problems $D^\tord$ and $D^\aord$ are representable by rings $R^\tord$ and $R^\aord$ in $\hat \cA_W$, respectively. 
    
    There is a natural isomorphism $R^\tord \isoto R^\ord$ under which the $\Frob_p$-eigenvalues $\alpha$ match. 

    There is a natural isomorphism $R^\aord \isoto R^\ord_{\rho \otimes (\det\rho)^{-1}}$ sending $\beta \mapsto \alpha$. It is given by applying the anti-ordinary deformation functor to $(\rho')^\ord \otimes (\det (\rho')^\ord_p)^{-1}$, where $(\rho')^\ord$ denotes the universal ordinary deformation, valued in $R^\ord$, of $\rho \otimes (\det\rho)^{-1}_p$; and $(-)_p$ means that we take the $p$-ramified part of a character of $G_\Q$. 
\end{lem}

The $p$-ramified part refers to the character of $G_\Q$ given by the summand $\Z_p^\times \subset G_\Q^\mathrm{ab}$ coming from class field theory. 

Notice that the passage $\rho \mapsto \rho \otimes (\det\rho)^{-1}_p$ ``swaps'' the unramified vs.\ possibly ramified Jordan--H\"older factors of $\rho\vert_{G_p}$: $\chi_1 \otimes (\det\rho)^{-1}_p\vert_{G_p}$ is unramified, while $\chi_2 \otimes (\det\rho)^{-1}_p\vert_{G_p}$ may be ramified. 

\begin{rem}
    Up to one of the deformation rings $R^\tord,R^\aord$ may be the zero ring. This occurs when $\rho\vert_{G_p}$ splits the ``wrong way'' relative to the flag imposed in the definition. Of course, this possibility was already present when we defined $R^\ord$. 
\end{rem}

While $R^\ord$ and $R^\tord$ are identical, the reason we distinguish them is that we are reserving them for comparison with different Hecke algebras. The essential difference is that $\bT^\ord_\Lambda$ is a $\bT[U]$-algebra while $\bT^\tord_\Lambda$ is a $\bT[U']$-algebra, while we will send 
\[
R^\ord \ni \alpha \mapsto U \in \bT^\ord_\Lambda \text{ and } R^\tord \ni \alpha \mapsto U' \in \bT^\tord_\Lambda.
\]
Correspondingly, the twist-ordinary Galois representation varies in exactly that way from \eqref{eq: ord form of rho_T}, that is, 
\begin{equation}
\label{eq: twist-ord form of rho_T}
\rho_{\bT_\Lambda^\tord}\vert_{G_p} \simeq \ttmat{\dia \cdot \nu({\lr{p}_N} U'^{-1})}{*}{0}{\nu(U')}.
\end{equation}

Likewise, applying the twisted $\bT[U']$-algebra structure of $\bT_\Lambda^\aord$ described in Proposition \ref{prop: aord TU algebra} (in particular, $U' \mapsto U_p^{-1}\lr{p}_N \in \bT_\Lambda^\aord$) to \eqref{eq: ord form of rho_T}, the Galois representation $\rho_{\bT_\Lambda^\aord} : G_{\Q,S} \to \GL_2(\bT_\Lambda^\aord)$ has $p$-local form
\begin{equation}
\label{eq: anti-ord form of rho_T}
\rho_{\bT_\Lambda^\aord}\vert_{G_p} \simeq \ttmat{\dia \cdot \nu(\lr{p}_NU'^{-1})}{0}{*}{\nu(U')}. 
\end{equation}

\begin{defn}
    We define the bi-ordinary and critical deformation problems $D^\BO$ and $D^\crit$ in the same way as in Definition \ref{defn: Dord} except for replacing condition (ii) with ``$\Lambda$-bi-ordinary'' and ``$\Lambda$-critical'' respectively. 
\end{defn}

\begin{prop}
    \label{prop: BO and crit rings}
    Under the assumptions on $\rho$ of Definition \ref{defn: rho ord assumptions}, $D^\BO$ and $D^\crit$ are representable by rings $R^\BO$ and $R^\crit$ in $\hat \cA_W$, respectively. There are natural surjections
    \[
    R^\crit \rsurj R^\tord, R^\crit \rsurj R^\aord, \quad R^\tord \rsurj R^\BO, R^\aord \rsurj R^\BO
    \]
    forming a pushout $R^\tord \otimes_{R^\crit} R^\aord \cong R^\BO$. 
\end{prop}
\begin{proof}
    The representability of $R^\BO$ is due to Ghate--Vatsal \cite[Prop.\ 3.1]{GV2011}. 
    
    Let $\tilde R$ denote the deformation ring representing the moduli problem given by Definition \ref{defn: Dord} except that condition (ii), the $p$-local condition, is eliminated. There exists such a $\tilde R$ satisfying conditions (i), (iii), and (iv) due to standard arguments, cf.\ \cite[Prop.\ 2.3.2]{CWE1}; let $\tilde \rho$ denote the universal representation over $\tilde R$. Let $R^\ps_p$ denote the ($p$-local) universal pseudodeformation ring of $\rho\vert_{G_p}$. Then the pseudorepresentation of $\tilde \rho\vert_{G_p}$ produces a homomorphism $R^\ps_p \to \tilde R$ in $\cA_W$. 

    It is a standard consequence of the theory of pseudorepresentations that the $\Lambda$-critical condition of Definition \ref{defn: p-local conditions Lambda} is Zariski closed in $\Spf R^\ps_p$, producing a quotient $R^\ps_p \rsurj R^{\ps,\crit}_p$. Then one may let $R^\crit := \tilde R \otimes_{R^\ps_p} R^{\ps,\crit}_p$. 

    The existences of the claimed surjections and pushout follow from the logical relations of the various deformation conditions discussed in Remark \ref{rem: relations of conditions}. 
\end{proof}

\begin{rem}
    Under our residually $p$-distinguished hypothesis, one can show that $R^{\ps,\crit}_p \cong \Lambda\lb t_1, t_2\rb$ parameterizing the pseudorepresentation of $\dia \nu(1+t_1) \oplus \nu(1+t_2)$. 
\end{rem}

First we characterize how much $R^\crit$ differs from $R^\tord$ and $R^\aord$. 
\begin{prop}
    \label{prop: R square nilpotent}
    The kernel of the natural map $\psi^R : R^\crit \rsurj R^\tord \times R^\aord$ is square-nilpotent, its $R^\crit$-module structure factors through $R^\BO$, and it is finitely generated as a $R^\BO$-module. 
\end{prop}

\begin{proof}
    Let $b, c : G_p \to R^\crit$ denote the functions realizing the $B$ and $C$-coordinates of the universal representation $\rho^\crit$. Viewing the deformation problems, we see the the ideal $(b(G_p)) \subset R^\crit$ is the kernel of $R^\crit \rsurj R^\aord$ while $(c(G_p))$ is the kernel of $R^\crit \rsurj R^\tord$. On the other hand, due to the requirement that the pseudorepresentation of $\rho^\crit\vert_{G_p}$ is reducible, the product ideal $(b(G_p))(c(G_p))$ vanishes (see e.g.\ \cite[Prop.\ 1.5.1]{BC2009}). Therefore the kernel $(b(G_p)) \cap (c(G_p))$ of $\psi^R$ is square-nilpotent. 

    It is also evident that $\ker \psi^R$ is annihilated by $(b(G_p))$ and by $(c(G_p))$, and therefore, by Proposition \ref{prop: BO and crit rings}, its $R^\crit$-module structure factors through $R^\BO$. 

    The kernel is finitely generated as a $R^\BO$-module due to a standard argument using adjoint Galois cohomology; see the proof of Theorem \ref{thm: R=T crit} for a similar argument.
\end{proof}

\subsection{$R^\star = \bT^\star$ results}

Now we prove $R^\star=\bT^\star$ results in the bi-ordinary and critical cases, taking the ordinary result as a hypothesis. 

\begin{prop}
    \label{prop: R=T for twist and anti}
    Assume $R^\ord \cong (\bT_\Lambda^\ord)_\rho$. The maps arising from \eqref{eq: twist-ord form of rho_T} and \eqref{eq: anti-ord form of rho_T} are isomorphisms $R^\tord \isoto (\bT^\tord_\Lambda)_\rho$, $R^\aord \isoto (\bT^\aord_\Lambda)_\rho$. Under each isomorphism, $R^{\tord/\aord} \ni \alpha \mapsto U' \in (\bT_\Lambda^{\tord/\aord})_\rho$. 
\end{prop}

\begin{proof}
    These two maps are twists of $R^\ord \cong \bT_\Lambda^\ord$. In the anti-ordinary case, this follows from observing that the twisting from $\bT^\ord_\Lambda$ to $\bT_\Lambda^\aord$ of Proposition \ref{prop: aord TU algebra} is compatible with the twisting from $R^\ord$ to $R^\aord$ in Lemma \ref{lem: twists of deformation problems}; this also corresponds to the passage from \eqref{eq: ord form of rho_T} to \eqref{eq: anti-ord form of rho_T}. In the twist-ordinary case, this follows from the twisting isomorphisms (just replacing $U$ by $U'$) between the ordinary and twist-ordinary theories in Proposition \ref{prop: tord control} and Lemma \ref{lem: twists of deformation problems}. This also corresponds to the passage from \eqref{eq: ord form of rho_T} to \eqref{eq: twist-ord form of rho_T}. 
\end{proof}

Thus we can arrive at the result that was one of the main motivations of the theory of critical $\Lambda$-adic forms, addressing the question we asked in \cite[\S1.6]{CWE1}. 
\begin{thm}
    \label{thm: R=T BO}
    Assume $R^\ord \cong (\bT_\Lambda^\ord)_\rho$. The natural map $R^\BO \rsurj (\bT^\BO_\Lambda)_\rho$ is an isomorphism. 
\end{thm}

\begin{proof}
    This map arises from the isomorphisms of Proposition \ref{prop: R=T for twist and anti} by pushing out along their sources and targets, according to Proposition \ref{prop: BO and crit rings}) and Corollary \ref{cor: BO is to and ao}. (Note also that $(\bT_\Lambda^\BO)_\rho \isoto (\bT_\Lambda^\SBO)_\rho$ because we are localizing at a residually non-Eisenstein maximal ideal.) 
\end{proof}

\begin{thm}
    \label{thm: R=T crit}
    Assume $R^\ord \cong (\bT_\Lambda^\ord)_\rho$. Then $R^\crit_\Lambda \cong (\bT^\crit_\Lambda)_\rho$. 
\end{thm}

\begin{proof}
    We have a diagram of finite $\Lambda$-algebras maps, each with a square-nilpotent kernel by Proposition \ref{prop: R square nilpotent}, 
    \begin{equation}
        \label{eq: R-T diagram}
        \begin{aligned}
        \xymatrix{
    R_\Lambda^\crit \ar[r]^{\psi_\Lambda^R } \ar@{->>}[d] & R_\Lambda^\tord \times R_\Lambda^\aord \ar[d]^\sim \\
    (\bT_\Lambda^\crit)_\rho \ar[r]^{\psi_\Lambda} & (\bT_\Lambda^\tord)_\rho \times (\bT_\Lambda^\aord)_\rho
    }
    \end{aligned}
    \end{equation}
    Because all rings in the diagram except $R^\crit_\Lambda$ are already known to be finite and flat over $\Lambda$ (by Theorem \ref{thm: main construction} and Proposition \ref{prop: R=T for twist and anti}), it follows from Proposition \ref{prop: R square nilpotent} that $R^\crit$ is $\Lambda$-finite as well. It also follows that it will suffice to ``prove the theorem modulo the maximal ideal $\m_\Lambda \subset \Lambda$'' in the following sense. For convenience, denote
    \[
    \bar R^\star := R^\star/\m_\Lambda R^\star \text{ for } \star = \crit, \tord,\aord, 
    \]
    and likewise write $\bar\bT^\star$ and $\bar \psi^R, \bar\psi$ for the reduction of \eqref{eq: R-T diagram} modulo $\m_\Lambda$. It is enough to prove that $\bar R^\crit \rsurj \bar \bT^\crit$ is an isomorphism. 
    
    Since the right map of \eqref{eq: R-T diagram} is an isomorphism, it also suffices to prove that $\ker \bar\psi_R \rsurj \ker \bar\psi$ is an isomorphism. We will prove this using a dimension count. Since the domain and codomain of $\bar\psi$ have equal $\F$-dimension (by Lemma \ref{lem: Euler characteristic} and the fact that we have localized at a non-Eisenstein eigensystem), and we know that $R^\BO \cong (\bT^\BO_\Lambda)_\rho \simeq (\coker \psi_\Lambda)_\rho$ (the right isomorphism follows from Corollary \ref{cor: BO is to and ao}, as noted in the proof of Proposition \ref{prop: BO T to forms duality}), it will suffice to prove that 
\[
\dim_\F \ker \bar\psi^R  \leq \dim_\F \bar R^\BO. 
\]
We will do this by beginning with the result of Proposition \ref{prop: R square nilpotent} that $\ker \psi^R$ is a $R^\BO$-module and then showing that $\ker \bar\psi^R$ has the cardinality of a cyclic $\bar R^\BO$-module. 

Because of the axioms on $\rho\vert_{G_p}$ declared in Definition \ref{defn: rho ord assumptions}, and letting $\chi = \chi_1 \chi_2^{-1}$ using the notation there, at least one of $\chi$ and $\chi^{-1}$ is not equal to $\omega$. Without loss of generality (because we can swap the role of ``anti-ordinary'' and ``twist-ordinary'' later in this proof) we assume that $\chi \neq \omega$. Write $\rho^\BO : G_{\Q,S} \to \GL_2(\bar R^\BO)$ for the universal Galois representation over $\bar R^\BO$ and $\tilde \rho^\BO$ for the square-nilpotent deformation of $\rho^\BO$ from $\bar R^\BO$ to $\bar R^\BO \oplus \ker \bar\psi^R$. Write the $p$-split form of $\rho^\BO$ as
\[
\rho^\BO\vert_{G_p} \simeq \ttmat{\chi_1^\BO}{0}{0}{\chi_2^\BO}, \qquad \chi_i : G_p \to (R^\BO)^\times
\]
and let $\chi^\BO := \chi_1^\BO (\chi_2^\BO)^{-1}$, which deforms $\chi$. 

It is a standard deformation-theoretic fact that the deformation $\tilde \rho^\BO$ gives rise to an element of 
\[
H^1(\Z[1/Np], (\ad^0 \rho^\BO) \otimes_{\bar R^\BO} \ker\bar\psi^R).
\]
By global Tate duality, using the fact that $\bar R^\BO$ has finite cardinality, there is a dual surjection $u: H^2_{(c)}(\Z[1/Np], (\ad^0 \rho^\BO)^*(1)) \rsurj \ker \bar\psi^R$, where the surjectivity arises from the universality of $\tilde \rho^\BO$ and $(-)^*$ represents the Pontryagin dual with the contragredient action. This manifestation of global Tate duality is visible in \cite[Prop.\ 3.4.2]{CWE1}, where the spectral sequence there is trivial since our coefficient ring can be taken to be $\Lambda/\m_\Lambda \cong \F$. (The reader can also consult \cite[\S3.4]{CWE1} and the references therein for standard definitions of global Galois cohomology with compact support $H^i_{(c)}(\Z[1/Np], -)$ and variants.) 

The fact that $\ker \psi^R$ vanishes under $\bar R^\crit \rsurj \bar R^\aord$ means that the image of $H^1(\Q_p, (\chi^\BO)^*(1)) \to H^2_{(c)}(\Z[1/Np], (\ad^0 \rho^\BO)^*(1))$ -- arising from the Tate-dual map of the composition of the restriction map 
\[
H^1(\Z[1/Np], (\ad^0 \rho^\BO) \otimes_{\bar R^\BO} \ker\bar\psi^R) \to H^1(\Q_p, (\ad^0 \rho^\BO) \otimes_{\bar R^\BO} \ker\bar\psi^R)
\]
with the projection to the summand $\chi^\BO\vert_{G_p} \subset \ad^0 \rho^\BO\vert_{G_p}$, which is the coordinate that must vanish on $G_p$ in order to be anti-ordinary -- projects surjectively to $\ker \bar\psi^R$. That is, we have a surjection $H^1(\Q_p, (\chi^\BO)^*(1)) \rsurj \ker \bar\psi^R$. Because we have arranged that $(\chi^\BO)^{-1}(1)$ is not in the congruence class of $\omega$ (since $\chi \neq 1$) nor the trivial character (since $\chi \neq \omega$), the Tate local duality and Euler characteristic formula imply that 
\[
\dim_\F H^1(\Q_p, (\chi^\BO)^*(1)) = \dim_\F (\chi^\BO)^{-1}(1) = \dim_\F \bar R^\BO. \qedhere 
\] 
\end{proof}

\section{Degree-shifting action of anti-cyclotomic units}
\label{sec:deg-shifting}

Our goals are to interpret the result of Ghate--Vatsal \cite{GV2004} on the CM-ness of Hida families with bi-ordinary Galois representations in terms of the bi-ordinary complex, apply $R^\crit \cong \bT_{\Lambda,\rho}^\crit$ to interpret $(\ker \psi_\Lambda)_\rho$ (along with its degree-shifting action on $H^*(\SBO_\Lambda)_\rho$) in terms of anti-cyclotomic global units, and then specialize this action to weight 1. 

\subsection{Complex multiplication}
\label{subsec: CM} 

Our next goal is to apply the following result of Ghate--Vatsal \cite{GV2004}. For the notion of modular forms with complex multiplication (``CM'') by an imaginary quadratic field $K/\Q$, see e.g.\ \cite[\S2.1.4]{CWE1}. There we also discuss the $\Lambda$-flat module $S_\Lambda^{\ord,\CM} \subset S_\Lambda^\ord$ of $\Lambda$-adic ordinary CM forms, which interpolates ordinary CM forms of weight $k \in \Z_{\geq 3}$, $S_k^{\ord,\CM}$, in the usual sense of Hida theory. Its Hecke algebra $\bT_\Lambda^{\ord,\CM}$ is a $\Lambda$-flat quotient of $\bT_\Lambda^{\ord, \circ}$ and inherits a $\Lambda$-perfect pairing $\bT_\Lambda^{\ord,\CM} \times S_\Lambda^{\ord,\CM} \to \Lambda$ from the ordinary pairing. 

\begin{prop}[{Ghate--Vatsal \cite[Thm.\ 3]{GV2004}}]
\label{prop: GV}
Under hypotheses (1), (3), and (4) of Definition \ref{defn: rho ord assumptions} about $\rho$, there are no non-CM $N$-primitive cuspidal $\bT[U]_\Lambda$-eigenforms in $S_{\Lambda,\rho}^\ord) \otimes_{\Lambda_W} \overline{Q(\Lambda_W)}$ that support a bi-ordinary Galois representation.
\end{prop}

We have the following interpretation of the result of Proposition \ref{prop: GV} in terms of CM forms, writing it generally enough that we do not require any hypotheses on $\rho$, e.g.\ we allow $\rho$ to be Eisenstein. 
\begin{prop}
    \label{prop: equivalent CM conditions}
    Only in this proposition do we allow for an arbitrary residual Hecke eigensystem $\rho$. The following conditions are equivalent. 
    \begin{enumerate}
        \item The result of Proposition \ref{prop: GV} is true, that is, there are no twist-ordinary $\Lambda$-adic non-CM cuspidal $\bT[U']$-eigenforms with residual Hecke eigensystem $\rho$ that support a bi-ordinary Galois representation
        \item $H^0(\SBO_\Lambda^\bullet)_\rho$ is isomorphic to $S_{\Lambda,\rho}^{\tord,\CM}$ under $H^0(\SBO_\Lambda^\bullet)_\rho \rinj S_{\Lambda,\rho}^\tord$
        \item The Hida families of $H^0(\BO_\Lambda^\bullet)_\rho$ are either CM or Eisenstein series 
        \item the maximal $\Lambda$-torsion-free quotient of $H^1(\SBO_\Lambda^\bullet)_\rho$ has a CM Hecke action. 
    \end{enumerate} 
\end{prop}

Note that (2)-(4) are considered to be true when the said $\bT[U']_\Lambda$-modules vanish. 

\begin{proof}
Because all twist-ordinary CM forms admit bi-ordinary Galois representations, conditions (1) and (2) are equivalent by Theorem \ref{thm: BO equals split in Hida families}. Lemma \ref{lem: Eis difference} implies the equivalence of (2) with (3), since Eisenstein series support bi-ordinary Galois representations. And Theorem \ref{thm: SD for BO} implies the equivalence of (2) with (4). 
\end{proof}

Returning to the restrictions of Definition \ref{defn: rho ord assumptions} on $\rho$, we have the following implications of Ghate--Vatsal's result. We also remark that these hypotheses imply that there is at most one imaginary quadratic field $K$ for which some Hecke eigensystem in the congruence class of $\rho$ has CM by $K$, and in this case $\rho \simeq \Ind_K^\Q \eta$ for some $\eta : G_K \to \F^\times$. We will call such $\rho$ \textit{CM}, for short. 

We let $\cH^{1,\ord}_{\Lambda,\CM}$ denote the maximal CM quotient of $\cH^{1,\ord}_\Lambda$. 
\begin{cor}
    \label{cor: SBO is CM}
    Under the hypotheses on $\rho$ of Definition \ref{defn: rho ord assumptions}, 
    and there exist $\bT[U']_\Lambda$-equivariant maps as follows. 
    \begin{enumerate}
        \item A surjection $H^1(\BO_\Lambda^\bullet)_\rho \rsurj \cH^{1,\ord}_{\Lambda,\CM,\rho} \cong H^0(\BO_\Lambda^\bullet)_\rho^\vee$ with $\Lambda$-torsion kernel, factoring the surjection $\cH_{\Lambda,\rho}^{1,\ord} \rsurj \cH^{1,\ord}_{\Lambda,\CM,\rho}$ with complementary factor $\cH_{\Lambda,\rho}^{1,\ord} \rsurj H^1(\BO_\Lambda^\bullet)_\rho$ coming from \eqref{eq: BO replace}
        \item Isomorphisms $H^0(\BO_\Lambda^\bullet)_\rho  \isoto S_{\Lambda,\rho}^{\tord,\CM}$, $H^0(\BO_\Lambda^\bullet)_\rho  \isoto S_{\Lambda,\rho}^{\aord,\CM}$ arising from the natural inclusion $H^0(\BO_\Lambda) \rinj M_\Lambda^\tord$. 
        \item A surjection $\bT_{\Lambda,\rho}^\SBO \rsurj \bT_{\Lambda,\rho}^{\tord, \CM} \cong \bT[U'](H^0(\SBO_\Lambda^\bullet))_\rho$ with $\Lambda$-torsion kernel isomorphic to $T_1(H^1(\SBO_\Lambda^\bullet))_\rho$. The Hecke eigensystems of the $\bT[U']_\Lambda$-action on $T_1(H^1(\SBO_\Lambda^\bullet))_\rho$ are non-CM. 
    \end{enumerate}
    In particular, $\rho$ is non-CM if and only if $H^0(\BO_\Lambda^\bullet)_\rho = 0$ if and only if $H^1(\BO_\Lambda^\bullet)_\rho$ and $\bT_{\Lambda,\rho}^\SBO$ are $\Lambda$-torsion. 
\end{cor}

\begin{proof}
    Because  $\Lambda$-adic Serre duality $\lr{}'_\mathrm{SD}$ of Corollary \ref{cor: tord serre duality} is $\bT[U']$-compatible, the maximal CM sub $S_{\Lambda,\rho}^{\tord,\CM} \subset S_{\Lambda,\rho}^\tord$ corresponds to the maximal CM quotient $H^1(\SBO_\Lambda^\bullet)_\rho \rsurj \cH^{1,\ord}_{\Lambda,\CM, \rho}$ under this duality. Then (1) follows from Proposition \ref{prop: equivalent CM conditions} and Theorem \ref{thm: SD for BO} because we know that the hypothesis of Proposition \ref{prop: GV} is satisfied. Part (2) follows directly from Proposition \ref{prop: equivalent CM conditions}. Part (3) follows from the previous parts and Proposition \ref{prop: BO T to forms duality}, using the fact that $\bT_{\Lambda,\rho}^{\tord,\CM}$ is the maximal CM quotient of $\bT_{\Lambda,\rho}^\tord$. 
\end{proof}

As a result of our discussion of the CM property, we have now identified the obstructions to a positive answer to Question \ref{ques: CG}. 

\begin{cor}
    \label{cor: T1H1 is exceptional}
    Under the assumption on $\rho$ of Definition \ref{defn: rho ord assumptions}, the height 1 primes of $\Lambda$ supporting $T_1(H^1(\BO_\Lambda^\bullet))_\rho$ (resp.\ supporting $\ker(\bT_{\Lambda,\rho}^\SBO \rsurj \bT_{\Lambda,\rho}^{\tord,\CM})$) are the height 1 primes supporting exceptions to Question \ref{ques: CG} that lie in the congruence class of $\rho$. 
\end{cor}

\begin{proof}
    For a height 1 prime ideal $P \subset \Lambda$, let $\SBO_P^\bullet := \SBO_\Lambda^\bullet \otimes_\Lambda \Lambda/P$. There is a $\bT[U']_\Lambda$-equivariant short exact sequence 
    \[
    0 \to H^0(\SBO_\Lambda^\bullet) \otimes_\Lambda \Lambda/P \to H^0(\SBO_P^\bullet) \to \Tor_1^\Lambda(H^1(\SBO_\Lambda^\bullet), \Lambda/P) \to 0
    \]
    and the claim follows from this, Corollary \ref{cor: SBO is CM}, and the fact that the the $\rho$-localization of the $\Tor$ term vanishes unless $P$ is in the support of $H^1(\SBO_\Lambda^\bullet)_\rho$ (because it has projective dimension $\leq 1$, Lemma \ref{lem: PD of H}). 
\end{proof}

\subsection{Arithmetic interpretation of the $\Lambda$-flat degree-shifting action}

We saw in the proof of Theorem \ref{thm: R=T crit} that there is a quotient ring $\tilde \bT_{\Lambda,\rho}^\BO := \bT_{\Lambda,\rho}^\SBO \oplus (\ker \psi_\Lambda)_\rho$ of $\bT_\Lambda^\crit$, a square-nilpotent extension of $\bT_{\Lambda,\rho}^\SBO$. Putting together Theorem \ref{thm: R=T crit} and the degree-shifting action of $\ker \psi_\Lambda$ of Theorem \ref{thm: flat derived action}, we deduce the following action on $H^*(\SBO_\Lambda^\bullet)_\rho = H^0(\SBO_\Lambda^\bullet)_\rho \oplus H^1(\SBO_\Lambda^\bullet)_\rho$. Recall that $\ker \psi_\Lambda^R$ is the natural surjection $R_\Lambda^\crit \to R_\Lambda^\tord \times R_\Lambda^\aord$ of \eqref{eq: R-T diagram}, and notice that Theorem \ref{thm: R=T crit} implies that we have a square-nilpotent lift 
\[
\tilde R^\BO := R^\BO \oplus \ker \psi_\Lambda^R \isoto \tilde \bT_{\Lambda,\rho}^\BO := (\bT_\Lambda^\SBO)_\rho \oplus (\ker \psi_\Lambda)_\rho
\]
of the isomorphism $R^\BO \cong \bT_{\Lambda,\rho}^\BO$ of Theorem \ref{thm: R=T BO}. 

\begin{cor}
    \label{cor: RBO shifts degrees}
    Assume $R^\ord \cong (\bT_\Lambda^\ord)_\rho$. There is a faithful action of $R^\BO \oplus \ker \psi_\Lambda^R$ on $H^*(\SBO_\Lambda^\bullet)_\rho$ where $R^\BO$ acts in degree $0$ and $\ker \psi_\Lambda^R$ acts in degree $-1$. 
\end{cor}
The idea for why $\ker \psi_\Lambda^R$ should act in degree $-1$ is that that $R^\crit \rsurj \tilde R^\BO$ is a classical realization of the derived tensor product of $R^\tord \lsurj R^\crit \rsurj R^\aord$, which surjects onto $R^\BO$ with ``properly derived'' square-nilpotent kernel $\ker \psi_\Lambda^R$. 

Next we interpret square-nilpotent extensions in terms of Galois cohomology. We assume that $\rho$ has CM by an imaginary quadratic field $K$ so that the degree shifting action is non-trivial. 
\begin{defn}
    For $\star = \tord, \aord,\crit, \BO, \tilde\BO$, let $J^\star$ denote the kernel of the surjection
    \[
    R^\star \rsurj \bT_{\Lambda,\rho}^{\tord,\CM}. 
    \]
    Here ``$\tilde\BO$'' refers to $\tilde R^\BO$. 
\end{defn}

In our previous work \cite{CWE1}, we identified the Galois cohomology groups that are canonically isomorphic to $J^\star/{J^\star}^2$ for $\star = \tord, \BO$; of course, $\star=\aord$ will exhibit a  symmetry with $\star=\tord$. To deal with all cases of $\star$, we import some notation from \cite{CWE1} (except that we use $\eta$ as follows and sometimes conflate it with its multiplicative lift). In this list, all of the references point to \cite{CWE1}. 
\begin{itemize}
    \item We continue letting $\eta : G_K \to \F^\times$ such that $\rho \simeq \Ind_K^\Q \eta$.
    \item For any character $\eta'$ of $G_K$, we let $\eta'^- := \eta \cdot (\eta^c)^{-1}$.
    \item Let $p$ factor as $p\cO_K = \frp\frp^*$ in $K$.
    \item Let $\tilde \Lambda$ denote the Iwasawa algebra of \S4.1, which we think of as a minimal deformation ring for $\eta$ in the sense of Lemma 4.1. There is also an isomorphism $\tilde \Lambda \cong \bT_{\Lambda,\rho}^{\tord,\CM}$ (Proposition 4.2). 
    \item Let $\widetilde{\Lambda}^-$ denote the anti-cyclotomic Iwasawa algebra of \S3.1. It admits the isomorphism $\tilde \delta : \widetilde{\Lambda}^- \to \tilde \Lambda$ of (4.2.1). We think of $\widetilde{\Lambda}^-$ of as a universal anti-cyclotomic deformation ring for $\eta^-$; let $\widetilde{\Lambda}_\#^-$ to denote the module of the inverse of the universal character. 
    \item Let $\cY^-_\infty(\eta^-) \rsurj \cX^-_\infty(\eta^-) \lsurj \cY^-_\infty(\eta^-)^*$ denote the anti-cyclotomic Iwasawa class groups defined in \S3.3, which are naturally $\widetilde{\Lambda}^-$-modules. In particular, $\cX^-_\infty(\eta^-)$ is unramified, $\cY^-_\infty(\eta^-)$ allows for ramification at $\frp$, and $\cY^-_\infty(\eta^-)^*$ allows for ramification at $\frp^*$. 
    \item As new notation, let $\frX^-_\infty(\eta^-)$ denote the analogous Iwasawa class group with ramification allowed at both $\frp$ and $\frp^*$, resulting in natural surjections $\cY^-_\infty(\eta^-) \lsurj \frX^-_\infty(\eta^-) \rsurj \cY^-_\infty(\eta^-)^*$. 
    \item Let $\cE^-_\infty(\eta^-)$ denote the Iwasawa module of global units of \S3.3, containing a submodule of elliptic units. 
\end{itemize}

As in \cite[\S3.5]{CWE1}, Kummer theory interprets these objects, but we will be especially interested in a different long exact sequence in Galois  cohomology than the one named \textit{ibid.} Our long exact sequence comes from the short exact Meyer--Vietoris-type sequence of cochains 
\begin{align*}
0 \to R\Gamma_{(p)}(\cO_K[1/Np], -)  \buildrel{(+,-)}\over\lra R\Gamma_{(\frp)}&(\cO_K[1/Np], -) \oplus R\Gamma_{ (\frp^*)}(\cO_K[1/Np], -) \\ 
& \to R\Gamma(\cO_K[1/Np], -) \to 0
\end{align*}
applied to $\widetilde{\Lambda}_\#^-(1)$, where ``$(p)$'' refers to the support condition at both of its divisors, $\frp$ and $\frp^*$. (For definitions of support conditions, see \cite[\S3.4]{CWE1} and references therein.) In the following long exact sequences, to shorten notation, we write ``$H^i_\square$'' in place of $H^i_\square(\cO_K[1/Np], \widetilde{\Lambda}_\#^-(1))$ for the various support conditions $\square$ . 
\begin{prop}[{\cite[Prop.\ 3.5.1, Lem.\ 3.5.4]{CWE1}}]
    \label{prop: kummer}
    Under the assumptions on $\rho$ of \cite[\S1.2.3]{CWE1}, there are naturally isomorphic long exact sequences of $\widetilde{\Lambda}^-$-modules 
    \[
    0 \to H^1 \to H^2_{(p)} \buildrel{(+,-)}\over\lra  H^2_{(\frp)} \oplus H^2_{(\frp^*)} \to H^2 \to 0
    \]
    and
    \[
    0 \to \cE^-_\infty(\eta^-) \to \frX^-_\infty(\eta^-) \buildrel{(+,-)}\over\lra \cY^-_\infty(\eta^-) \oplus \cY^-_\infty(\eta^-)^* \to \cX^-_\infty(\eta^-) \to 0.
    \]
\end{prop}

\begin{rem}
    This exact sequence is related to the study in \cite{BCGKPST2020} of second Chern classes in Iwasawa theory. 
\end{rem}

\begin{proof}
    The exactness is clear from the Meyer--Vietoris construction. The isomorphisms follow from the references above for all terms except $\frX^-_\infty(\eta^-)$; the isomorphism in this last term follows by the same argument as the argument for \cite[Eq (3.5.2)]{CWE1} being an isomorphism, with support at $(p)$ replacing support at $(\frp)$. The terminal zero terms in the long exact sequence vanish; they are exactly the same vanishing terms as in the long exact sequence of \cite[Prop.\ 3.5.1]{CWE1}. 
\end{proof}

\begin{prop}
    \label{prop: conormal}
    The pushout diagram underlying $R^\tord \otimes_{R^\crit} R^\aord \isoto R^\BO$ produces a long sequence of conormal modules relative to these rings' and $\tilde R^\BO$'s 
    surjections to $\widetilde{\Lambda}^- \cong \bT_\Lambda^{\tord,\CM}$,  
    \begin{gather*}
        0 \to \ker\psi_\Lambda^R \to J^\crit/{J^\crit}^2 \buildrel{(+,-)}\over\lra J^\tord/{J^\tord}^2 \oplus J^\aord/{J^\aord}^2 \to J^\BO/{J^\BO}^2 \to 0,
    \end{gather*}
    that is naturally isomorphic to the long exact sequences of Proposition \ref{prop: kummer}.
\end{prop}

We suspect that this sequence is realizable in a long exact sequence of homotopy groups of relative cotangent complexes arising from the intersection of the twist-ordinary and anti-ordinary loci within $\Spec R^\crit$, since $R^\star$ is complete intersection for $\star = \tord,\aord$ under our assumptions. Indeed, $\Tor_1^{R^\crit}(R^\tord,R^\aord) \cong \ker \psi_\Lambda^R$. We just give this down-to-earth argument. 
\begin{proof}
    The existence and exactness of this long exact sequence follows directly from Propositions \ref{prop: BO and crit rings} and \ref{prop: R square nilpotent}. The term-by-term isomorphisms with the exact sequences of Proposition \ref{prop: kummer}, for the final three non-zero terms, follow from the arguments of \cite[\S5]{CWE1}, which culminate in [Thm.\ 5.4.1, \textit{loc.\ cit.}]. There the cases $\star = \tord, \BO$ are dealt with explicitly \textit{ibid}., and then $\star = \aord$ follows by swapping the roles of $\frp$ and $\frp^*$. The isomorphism $J^\crit/{J^\crit}^2 \cong H^2_{(p)}(\cO_K[1/Np], \widetilde{\Lambda}_\#^-(1))$ follow in exactly the same way as the argument for the case for $J^\tord/{J^\tord}^2 \cong H^2_{(\frp)}(-)$ discussed in the proof \textit{ibid.} Finally, the isomorphisms connecting the final three non-zero terms are functorial with respect to the natural arrows in these Meyer--Vietoris type sequences. Then we compatibly produce the isomorphism in the leftmost term of the sequences by naturally identifying the kernels of the arrows labeled with $(+,-)$. 
\end{proof}

\begin{thm}
    \label{thm: elliptic action}
    Under the assumptions on $\rho$ of \cite[\S1.2.3]{CWE1}, there is a faithful degree $-1$ action of $\cE_\infty^-(\eta^-)$ on $H^*(\SBO_\Lambda^\bullet)_\rho$, 
    \[
    \cE_\infty^-(\eta^-) \lrisom \Hom_{\bT[U']_\Lambda}(H^1(\SBO_\Lambda^\bullet)_\rho, H^0(\SBO_\Lambda^\bullet)_\rho) \cong \Hom_{\bT[U']_\Lambda}(\cH^{1,\ord}_{\Lambda,\CM,\rho}, S_{\Lambda,\rho}^{\tord,\CM}). 
    \]
\end{thm}

\begin{proof}
    Combine Proposition \ref{prop: conormal} with the localization at $\rho$ of the action of Theorem \ref{thm: flat derived action}. This action factors through $H^1(\SBO_\Lambda^\bullet)_\rho \rsurj \cH^{1,\ord}_{\Lambda,\CM,\rho}$ because its target is $\Lambda$-flat and, by Corollary \ref{cor: SBO is CM}, $\cH^{1,\ord}_{\Lambda,\CM,\rho}$ is the maximal $\Lambda$-torsion-free quotient. 
\end{proof}

\subsection{Weight 1}

We now specialize to weight 1 and aim to realize the action of a Stark unit group. Let $f$ be a normalized weight 1 classical $T_p$-ordinary CM cuspidal eigenform of level $\Gamma_1(N)$ with coefficients in a number field $L$ such that its associated Artin representation $\rho_f : G_\Q \to \GL_2(\cO_L)$ is definable over $\cO_L$. Let $H$ denote the number field cut out by the trace-zero adjoint representation $\Ad^0 \rho_f$.  As explained in \cite[\S1.2]{DHRV2022}, Dirichlet's unit theorem gives that $U_f := (\cO_H^\times \otimes (\Ad^0 \rho_f)^\vee)^{G_\Q}$, called the \emph{Stark unit group}, has $\cO_L$-rank $1$. 

We let $E$ be our fixed $p$-adic completion of $L$ and let $f_\beta$ be a $U_p$-ordinary stabilization (which is also $U_p$-critical). We think of $f_\beta \in M_{1,\cO_E}^{\dagger,\tord}$ as a $\bT[U']$-eigenform. To reduce the complexity of the proof, and because we are mainly interested in a proof of principle, we assume that the class number $h_K$ of $K$ is not divisible by $p$. This assumption makes the weight map for the CM Hecke algebra $\Lambda \to \bT_{\Lambda,\rho}^{\tord,\CM} \cong \widetilde{\Lambda}^-$ an isomorphism after replacing $\Lambda$ with $\Lambda_W = \Lambda \otimes_{\Z_p} W$ for $W = W(\F)$ \cite[Lem.\ 4.1.4]{CWE1}. Also, we can take $\cO_E = W$,  $f_\beta$ becomes $\phi_1 : \Lambda_W \to W$, and $\eta^-$ is defined over $W$. 
\begin{cor}
    \label{cor: Stark unit action}
    Assume that the residual Hecke eigensystem $\rho$ of $f$ satisfies the assumptions of \cite[\S1.2.3]{CWE1} and that $p \nmid h_K$. If $\cX_\infty^-(\eta^-) = 0$, then the Stark unit group admits a faithful action 
    \[
    U_f \rinj \Hom_{\bT[U']}(\cH^{1,\ord}_{\Lambda,\CM, \rho} \otimes_{\bT_\Lambda^{\tord,\CM},\phi_{f_\beta}} \cO_E, S_{\Lambda,\rho}^{\tord,\CM} \otimes_{\bT_\Lambda^{\tord,\CM},\phi_{f_\beta}} \cO_E)
    \]
    that is an isomorphism on $U_f \otimes_{\cO_L} \cO_E$. 
\end{cor}

\begin{rem}
    \label{rem: X-infty}
    The main theorem of \cite{CWE1} is that $\cX_\infty^-(\eta^-) = 0$ if and only if the surjection $\bT_{\Lambda,\rho}^\BO \rsurj \bT_{\Lambda,\rho}^{\tord,\CM}$ is an isomorphism. However, it is possible to have both $\cX_\infty^-(\eta^-) = 0$ and also proper surjections $\bT_{\Lambda,\rho}^\tord\rsurj \bT_{\Lambda,\rho}^{\tord,\CM}$ and $\bT_{\Lambda,\rho}^\aord\rsurj \bT_{\Lambda,\rho}^{\aord,\CM}$. While there are examples of one of these surjections being proper (some are listed in \cite[\S1.8]{CWE1}), we are not aware of a case where both surjections are proper, nor the even more narrow case where $\cX_\infty^-(\eta^-) \neq 0$. 
\end{rem}

\begin{proof}
    First we point out a few useful facts. As per Remark \ref{rem: X-infty}, our assumptions imply that the surjection $\bT_{\Lambda, \rho}^\SBO \rsurj \bT_\Lambda^{\tord,\CM}$ is an isomorphism, which by Corollary \ref{cor: SBO is CM} and Proposition \ref{prop: BO T to forms duality} implies that $T_1(H^1(\SBO_\Lambda^\bullet)_\rho) = 0$. Next we note that  $H^0(\SBO_\Lambda^\bullet)_\rho \cong S_{\Lambda,\rho}^{\tord,\CM}$ and $H^1(\SBO_\Lambda^\bullet)_\rho \cong \cH^{1,\ord}_{\Lambda,\CM,\rho}$ are $\bT_{\Lambda,\rho}^{\tord,\CM}$-free of rank $1$. This follows from the perfect dualities of Proposition \ref{prop: BO T to forms duality} and Corollary \ref{cor: SD on HSBO}, the fact that $\bT[U'](H^0(\SBO_\Lambda^\bullet)_\rho) \cong \bT_{\Lambda,\rho}^{\tord,\CM}$ (from Corollary \ref{cor: SBO is CM}), and the fact that the Iwasawa algebra $\bT_\Lambda^{\tord,\CM} \cong \widetilde{\Lambda}^-$ is complete intersection (and also regular because $p \nmid h_K$), and therefore Gorenstein. 

    Now that we know this freeness, we can specialize the action of Theorem \ref{thm: elliptic action} along the map $\phi_{f_\beta} : \bT_\Lambda^{\tord,\CM} \to \cO_E$ corresponding to $f_\beta$ and obtain an action map 
    \[
    \cE_\infty^-(\eta^-) \otimes_{\widetilde{\Lambda}^-, \phi} \cO_E \isoto 
    \Hom_{\bT[U']}(\cH^{1,\ord}_{\Lambda,\CM,\rho} \otimes_{\phi_{f_\beta}} \cO_E, S_{\Lambda,\rho}^{\tord,\CM} \otimes_{\phi_{f_\beta}} \cO_E).
    \]

    It remains to show that $U_f \otimes_{\cO_L} \cO_E$ admits an injection to $\cE_\infty^-(\eta^-)$ that is an isomorphism if $\cX_\infty^-(\eta^-) = 0$. Using the abbreviations of Proposition \ref{prop: kummer}, we can compute $H^*(\cO_K[1/Np], (\eta^-)^{-1}(1))$ from the $H^i$ of Proposition \ref{prop: kummer} under a short exact sequence coming from the spectral sequence \cite[Thm.\ 5.6.4]{weibel1994}
    \[
    \cE_\infty^-(\eta^-) \otimes_{\Lambda_W, \phi_1} W \to H^1(\cO_K[1/Np], (\eta^-)^{-1}(1)) \to \Tor_1^{\Lambda_W}(\cX_\infty^-(\eta^-), \Lambda_W/\ker \phi_1). 
    \]
    Kummer theory realizes $U_f \otimes_{\cO_L} \cO_E$ as a subgroup of $H^1(\cO_K[1/Np], (\eta^-)^{-1}(1))$, so we have the desired result provided that $\cX_\infty^-(\eta^-) = 0$. 
\end{proof}

It is possible to apply the stabilization map of Boxer--Pilloni \cite[Thm.\ 4.18(3)]{BP2022} in weight 1 to realize this action at level $\Gamma_1(N)$. 

\begin{prop}[{Weight 1 case of \cite[Thm.\ 4.18]{BP2022}}]
    \label{prop: weight 1 stabilize}
    Stabilization maps produce a commutative diagram of Serre duality pairings $\lr{}_\mathrm{SD}$
    \[
    \xymatrix@C=0em{
    e(U_p)H^0(X^\ord, \omega(-C)) & \times & e(F)H^1_c(X^\ord, \omega) \ar[rr] \ar[d]^j & \ &  \Z_p \\
    e(T_p)H^0(X, \omega(-C)) \ar[u]^i & \times & e(T_p)H^1(X, \omega) \ar[rru] & \ & 
    }
    \]
    where the top pairing is the specialization along $\phi_1 : \Lambda \to \Z_p$ of the $\Lambda$-perfect Serre duality pairing of Theorem \ref{thm: BP}. 
\end{prop}

We know that the top pairing is $\Z_p$-perfect. Because $e(T_p)H^0(X,\omega(-C))$ is $\Z_p$-torsion-free and $i$ is injective, it follows that $j$ is surjective at least after projection to the $\Z_p$-torsion-free quotient, the left kernel of the lower pairing is trivial, and the right kernel consists of $e(T_p)H^1(X,\omega)[p^\infty]$.  

Using $R=\bT$ theorems, we will be able to factor (in the proof of Theorem \ref{thm: classical Stark unit action}) the localization of $i$ and $j$ at $f$ through bi-ordinary cohomology via maps $\lambda$ and $\sigma$, respectively. We will also need the following maps. Thinking of $\SBO_1^\bullet$ in its quasi-isomorphic form 
\[
\SBO_1^\bullet \cong [M_1^{\dagger,\tord} \to e(F)H^1_c(X^\ord,\omega)], 
\]
as in \eqref{eq: SBO replace}, we have a natural injection $\lambda : H^0(\SBO_1^\bullet) \rinj M_1^{\dagger,\tord}$ and a natural surjection $\sigma : e(F)H^1_c(X^\ord,\omega) \rsurj H^1(\SBO_1^\bullet)$.  

Let $\bar\upsilon_f \in e(T_p)H^1(X_{\cO_E},\omega)_f$ denote a choice of $\bT[U']$-eigenclass that pairs to $1$ with $f$ under $\lr{}_\mathrm{SD}$. This $\bar\upsilon_f$ is well defined up to $e(T_p)H^1(X,\omega)_f[p^\infty]$, but this torsion vanishes due to the isomorphism $H^1(\SBO_\Lambda^\bullet)_\rho \isoto \cH^{1,\ord}_{\Lambda,\CM, \rho}$ appearing in the proof of Corollary \ref{cor: Stark unit action}. 

\begin{thm}
    \label{thm: classical Stark unit action}
    Under the assumptions of Corollary \ref{cor: Stark unit action} including $\cX_\infty^-(\eta^-) = 0$ and that $f$ is CM, there exists an faithful action of $U_f$ via 
    \[
    U_f \rinj \Hom_{\cO_E}(e(T_p)H^1(X_{\cO_E},\omega)_f, e(T_p)H^0(X_{\cO_E},\omega(-C))_f).
    \]
    It is determined by sending the $\cO_E$-generator $\bar\upsilon_f$ of $e(T_p)H^1(X_{\cO_E},\omega)_f/[p^\infty]$ to 
    \[
    (i^{-1} \circ \lambda \circ \mathrel{\mathop{\mathrm{act}}^\text{\ref{cor: Stark unit action}}} \circ\, \sigma \circ j^{-1})(\bar\upsilon_f). 
    \]
\end{thm}

\begin{proof}
    We apply the fact that the Galois representations in the congruence class of $\rho$ associated to weight 1 coherent cohomology, even $\Z_p$-torsion classes, are bi-ordinary. See e.g.\ \cite[Thm.\ 1.1]{CS2019}, which states that the $G_\Q$-pseudorepresentations arising from weight 1 coherent cohomology are unramified at $p$; because pseudodeformations of $\rho$ are equivalent to deformations of $\rho$, and because any $p$-local extension class under our residually $p$-distinguished hypothesis (Definition \ref{defn: rho ord assumptions}(4)) is ramified, this implies the bi-ordinary condition. 

    Therefore our choice $\beta$ of $U_p$-eigenvalue of the stabilization $f_\beta$ of $f$ determines a map $R^\BO \to \cO_E$, and the isomorphism $R^\BO \cong \bT_{\Lambda,\rho}^\BO$ of Theorem \ref{thm: R=T BO}, along with the characterizations of $H^1(\SBO_\Lambda^\bullet)_\rho$ described immediately before the statement (of Theorem \ref{thm: classical Stark unit action}), implies that $H^1(\SBO_1^\bullet)_f$ is a free $\cO_E$-module generated by the image under $\sigma$ of $j^{-1}(\bar\upsilon_f)$. Indeed, this follows, by specialization to the $\cO_E$-valued eigensystem of $f$, from the statement early in the proof of Corollary \ref{cor: Stark unit action} that $H^1(\SBO_\Lambda^\bullet)_\rho$ is $\bT_{\Lambda,\rho}^{\tord,\CM}$-free of rank 1.

    Applying now the compatibility of Proposition \ref{prop: weight 1 stabilize} between the two Serre dualities, along with the reduction modulo $\ker \phi_1 \subset \Lambda_W$ of the bi-ordinary Serre duality pairing of Theorem \ref{thm: SD for BO} (which arises from the stabilized Serre duality of Boxer--Pilloni, Theorem \ref{thm: BP}), we deduce that $U_f \cdot \sigma(j^{-1}(\bar\upsilon_f))$ lies in the image of $i$ evaluated on $e(T_p)H^0(X,\omega^k(-C))_f$, and therefore the stated definition of the action of $U_f$ makes sense. 
\end{proof}

\section{{Correction to \cite{CWE1}}}
\label{sec: correction}

We thank Bharath Palvannan and Shaunak Deo for pointing out to us that the commutative algebra arguments of \cite[Prop.\ 6.1.2]{CWE1} are incorrect, partially invalidating this statement. Repairing this statement requires an additional assumption, which also adds a corresponding condition to the theorems stated in \cite[\S1.4]{CWE1}. We will describe these issues and the corrected statement and theorems below. 

First we address the independence of this article on \cite{CWE1}. The arguments in this article are independent of the error in \cite{CWE1} because they do not apply any of the results following from the faulty Prop.\ 6.1.2 \textit{ibid}. More specifically, in this paper we only rely on \cite{CWE1} for theorems stated in \cite[\S1.3]{CWE1}, while only the results stated in \cite[\S1.4]{CWE1} are invalid and need to be repaired as described below. 

\subsection{Repaired statement in commutative algebra}

Here is a claim similar to \cite[Prop.\ 6.1.2]{CWE1}, using similar notation, which has an additional condition that is required to make it valid. We use similar notation to \textit{ibid}. 

\begin{prop}
    \label{prop: corrected from CWE1}
    Let $R$ be a complete Noetherian regular local ring. Let $S$ be an
augmented reduced local $R$-algebra that is finitely generated, torsion-free as an $R$-
module, and monogeneric as an $R$-algebra. Let $T$ be an augmented local $R$-algebra quotient of $S$, and denote by $K$ the kernel of $T \rsurj R$. 

Assume that $K/K^2$ is supported in codimension at least $2$ as an $R$-module. Then $T$ has generic rank $1$. 
\end{prop}

In \cite[Proof of Prop.\ 6.1.2]{CWE1}, $\cG$ denotes a minimal set of generators of the kernel $J$ of $S \rsurj R$, which is also a minimal set of $R$-algebra generators for $S$ and, by Nakayama's lemma, a minimal set of $R$-module generators for $J/J^2$. The claimed reduction step \textit{ibid}.\ from general $\#\cG$ to $\#\cG=1$ is invalid. But the rest of the proof is valid. This is why Proposition \ref{prop: corrected from CWE1}, which differs from \cite[Prop.\ 6.1.2]{CWE1} only in its requirement that $S$ is monogeneric, is a best-possible repair. 

\subsection{{Correction to other claims of \cite{CWE1}}}

As we have just remarked, the new restriction ``$\#\cG=1$'' is equivalent to the conormal module $J/J^2$ being cyclic as an $R$-module. In \cite{CWE1}, a ring $\Lambda$ plays the role of $R$ once we assume that $p$ does not divide the class number of the imaginary quadratic field $K$ of CM, that is, $p \nmid h_K$. By \cite[Thm.\ 5.4.1]{CWE1}, under $p\nmid h_K$,  $J/J^2$ is $\Lambda$-cyclic if and only if a certain Iwasawa class group denoted $\cY^-_\infty(\psi^-)$ \textit{ibid.}\ is $\Lambda$-cyclic. Thus, adding a cyclicity assumption on $\cY^-_\infty(\psi^-)$ to some of the claims in \cite{CWE1}, we can fix the mistakes that are downstream from the error in \cite[Prop.\ 6.1.2]{CWE1}. 

Here we list all of these corrections to \cite{CWE1}, with all references and notation referring to \cite{CWE1}. 

\begin{itemize}
    \item Lemma 6.2.1: add the assumption that $\cY_\infty^-(\psi^-)$ is cyclic as a $\Lambda^-$-module. 
    \item Remark 6.2.2: add ``the cyclicity of $\cY_\infty^-(\psi^-)$'' as an assumption needed to prove Theorem 1.4.1. 
    \item Theorem 1.4.1: add cyclicity of $\cY_\infty^-(\psi^-)$ to the theorem's assumptions
    \item Theorem 1.4.4: add cyclicity of $\cY_\infty^-(\psi^-)$ to the theorem's assumptions. 
\end{itemize}

\begin{rem}
    Along the lines of \cite[Rem.\ 1.4.2]{CWE1}, we point out that, due to the symmetry between twist-ordinary and anti-ordinary forms explained in this article, one could weaken the above assumption on  $\cY_\infty^-(\psi^-)$ to the assumption that just one of two class groups (called $\cY_\infty^-(\eta^-)$ and $\cY_\infty^-(\eta^-)^*$ in this article) is cyclic.     
\end{rem}

\subsection{{Corrections to the proof of \cite[Prop.\ 6.1.2]{CWE1}}}

The proof of Proposition \ref{prop: corrected from CWE1} appearing as \cite[Proof of Prop.\ 6.1.2]{CWE1} is essentially correct once $\#\cG=1$ is assumed and the content to reduce to the case $\#\cG = 1$ is ignored. Nonetheless there are a few misstatements that are only correct case where $R = \Lambda$. Here are corrections to these misstatements that work in for general regular Noetherian local $R$. 

\begin{itemize}
    \item on page 41 of \cite{CWE1}, the claim that ``this follows directly from the assumption that $K/K^2$ is supported in codimension $2$'' should be restated as ``this follows directly from the assumption that $K/K^2$ is supported in codimension at least $2$ and that $R$ is regular Noetherian local.''
    \item in the last sentence on page 41 of \cite{CWE1} the term ``finite'', referring to an $R$-module, should be replaced by ``supported in codimension at least $2$.'' 
\end{itemize}




\bibliographystyle{alpha}
\bibliography{CWEbib-2025-CG}
\end{document}